\documentclass[11pt,reqno]{amsart}
\usepackage{amsmath,amssymb,mathrsfs}
\usepackage{graphicx,cite,times}
\usepackage[colorlinks,
citecolor=red,
urlcolor=blue,
backref=page]{hyperref}

\usepackage{graphicx,subfigure}
\usepackage{bm}
\usepackage[all]{xy}
\usepackage{pifont}

\newtheorem{theo}{Theorem}[section]

\newtheorem{lemm}[theo]{Lemma}

\newtheorem{prop}[theo]{Proposition}
\newtheorem{rema}[theo]{Remark}
\newtheorem{property}[theo]{Property}

\numberwithin{equation}{section}

\allowdisplaybreaks

\begin{document}
\large
\title[spatially quasi-periodic solutions]{Local Existence and Uniqueness of Spatially Quasi-Periodic Solutions to the Generalized KdV Equation}

\author{David Damanik}
\address{\scriptsize (D. Damanik)~Department of Mathematics, Rice University, 6100 S. Main Street, Houston, Texas
77005-1892}
\email{damanik@rice.edu}
\thanks{The first author (D. Damanik) was supported by Simons Fellowship $\# 669836$ and NSF grants DMS--1700131 and DMS--2054752}

\author{Yong Li}
\address{\scriptsize  (Y. Li)~Institute of Mathematics, Jilin University, Changchun 130012, P.R. China. School of Mathematics and Statistics, Center for Mathematics and Interdisciplinary Sciences, Northeast Normal University, Changchun, Jilin 130024, P.R.China.}
\email{liyong@jlu.edu.cn}
\thanks{The second author (Y. Li) was supported in part by NSFC grant: 120701132, 11171132,  11571065 and National Research Program of China Grant 2013CB834100, and Natural Science Foundation of Jilin Province (20200201253JC)}

\author{Fei Xu}
\address{\scriptsize (F. Xu)~Institute of Mathematics, Jilin University, Changchun 130012, P.R. China.}
\email{stuxuf@outlook.com}
\thanks{The third author (F. Xu) was supported by Graduate Innovation Fund of Jilin University (101832018C162). }



\begin{abstract}
In this paper, we study the existence and uniqueness of spatially quasi-periodic solutions to the generalized KdV equation (gKdV for short) on the real line with quasi-periodic initial data whose Fourier coefficients are exponentially decaying. In order to solve for the Fourier coefficients of the solution, we first reduce the nonlinear dispersive partial differential equation to a nonlinear infinite system of coupled ordinary differential equations, and then construct the Picard sequence to approximate them. However, we meet, and have to deal with, the difficulty of studying {\bf the higher dimensional discrete convolution operation for several functions}:
\[\underbrace{c\times\cdots\times c}_{\mathfrak p~\text{times}}~(\text{total distance}):=\sum_{\substack{\clubsuit_1,\cdots,\clubsuit_{\mathfrak p}\in\mathbb Z^\nu\\ \clubsuit_1+\cdots+\clubsuit_{\mathfrak p}=~\text{total distance}}}\prod_{j=1}^{\mathfrak p}c(\clubsuit_j).\]
In order to overcome it, we apply a combinatorial method to reformulate the Picard sequence as a tree. Based on this form, we prove that the Picard sequence is exponentially decaying and fundamental ({\color{red}i.e., a} Cauchy sequence). We first give a detailed discussion of the proof of the existence and uniqueness result in the case $\mathfrak p=3$. Next, we prove existence and uniqueness in the general case $\mathfrak p\geq 2$, which then covers the remaining cases $\mathfrak p\geq 4$. As a byproduct, we recover the local result from \cite{damanik16}. We exhibit the most important combinatorial index $\sigma$ and obtain a relationship with other indices, which is essential to our proofs in the case of general $\mathfrak p$.
\end{abstract}

\maketitle

\tableofcontents

\part{Introduction and Statement of the Main Results}

In this paper we consider the Cauchy problem given by the generalized KdV equation
\begin{align}\label{ppkdv}
\partial_tu+\partial_x^3u+u^{\mathfrak p-1}\partial_xu=0
\end{align}
with quasi-periodic initial data
\begin{align}\label{iei}
u(0,x) = u_0(x)=\sum_{n\in\mathbb Z^\nu}c(n)e^{{\rm i}(n\cdot\omega)x},
\end{align}
where the spatial wave vector $(\omega_1,\cdots,\omega_\nu)=\omega\in\mathbb R^\nu$ is rationally independent (mathematics) or non-resonant (physics), that is, $n\cdot\omega=0$ implies that $n=0\in\mathbb Z^\nu$. We are interested in studying the existence and uniqueness of spatially quasi-periodic solutions
\begin{align}
u(t,x)=\sum_{n\in\mathbb Z^\nu}c(t,n)e^{{\rm i}(n\cdot\omega)x}
\end{align}
to the Cauchy problem \eqref{ppkdv}--\eqref{iei}. For simplicity we refer to \eqref{ppkdv} as $\mathfrak p$-gKdV.

When $\mathfrak p=2$, $2$-gKdV \eqref{ppkdv} is exactly the classical KdV equation; when $\mathfrak p=3$, $3$-gKdV \eqref{ppkdv} is the Gardner equation \cite{M68, MGK68}.

In the case of the KdV equation, there has been a significant amount of recent activity centered around this particular Cauchy problem; see, for example, \cite{BDGL, damanik16, EVY19}. These works were motivated by the so-called \emph{Deift conjecture}, which states that the KdV equation with almost periodic initial data admits global solutions that are almost periodic in both space and time; compare \cite{D08, D17}. It is now known that the Deift conjecture is correct for a class of almost periodic initial conditions, namely those for which the associated Schr\"odinger operator is reflectionless and has a spectrum that has a sufficiently tame gap structure; see \cite{BDGL, EVY19} for detailed statements. It is also known how to satisfy these assumptions within the class of quasi-periodic initial data \eqref{iei}; see \cite{BDGL, damanik16}.

It should be stressed that these works use many features of the KdV equation that are not present in the case of general $\mathfrak p$ and for which no suitable substitutes are known as of yet. We therefore begin with the more modest goal of studying the local existence and uniqueness problem for general values of $\mathfrak p$. Nevertheless, the recent advances in the KdV theory do suggest a way to approach this problem.

In fact, the three papers \cite{BDGL, damanik16, EVY19} employ three different ways to establish the existence of solutions. The paper \cite{damanik16} stays in the setting of the actual functions one studies, both as initial data and as solutions at times different from zero. In the case at hand this would be the class of quasi-periodic functions (along with some decay condition on the Fourier coefficient $c(\cdot)$ we will impose). The papers \cite{BDGL, EVY19} make use of the associated operators. That is, one associates a Schr\"odinger operator with the given function by using the latter as the potential of the former. Within the class of function one studies, one imposes on the associated operators a uniform spectrum along with the so-called reflectionlessness property. The resulting class of potentials/operators can then be associated with a torus of dimension given by the number of the gaps of the spectrum in two different ways. One can either associate Dirichlet data or pass to the dual group of the fundamental group of the complement of the spectrum. The KdV flow can then be related to flows on either of these tori, and one can establish existence of solutions there and then pull them back. Now, \cite{BDGL} establishes solutions on the first torus and \cite{EVY19} on the second. While this approach produces results that are  very general in terms of the spectral assumptions one needs to impose for the machinery to work, it has the downside that the spectral interpretation (and in particular the crucial concept of reflectionlessness of the associated operators; see \cite{C89, GY06, SY95, SY97} for more information) needs to be available.

We also want to mention the paper \cite{K18}, which is related to the papers mentioned above (perhaps most closely to \cite{EVY19}) and uses reflectionlessness (on more restricted sets serving as spectra) as a key notion as well.

As we unfortunately do not have a known analog of the reflectionlessness concept in the case of general $\mathfrak p$, based on the discussion above it is therefore most natural to explore whether the approach developed in \cite{damanik16} can be extended to our setting.

Examining the results and proofs in \cite{damanik16}, we notice that for suitable quasi-periodic initial data, the authors prove global existence and uniqueness. This is in turn obtained by establishing a local result first and then using the deep fact that one can iterate the local argument in \emph{uniform time steps} because the size of the step will be uniform across all quasi-periodic functions that can arise by solving across the time interval in question. This deep fact is again a result of a spectral analysis of the associated Schr\"odinger operators. Indeed, in the earlier paper \cite{damanik14}, the same authors managed to show that there is a bidirectional correspondence between the exponential decay rate of the Fourier coefficients and the exponential decay rate of the lengths of the gaps of the spectrum. Since the KdV evolution preserves the spectrum, the gaps are preserved, hence the decay rate of their lengths is preserved, and hence one can deduce a \emph{uniform} upper bound on the decay rate of the Fourier coefficients of the solution. Since this in turn determines the length of the time interval on which one can prove local existence and uniqueness, an iteration immediately yields a global result. Again, since for general $\mathfrak p$, we cannot make use of the spectral theory aspects of this approach, it is at this point unclear how to pass from a local result to a global result in the general setting.

Given the discussion above, it is now clear what the goal of the present paper is. We wish to show a local existence and uniqueness result for the Cauchy problem associated with the $\mathfrak p$-gKdV equation \eqref{ppkdv} and quasi-periodic initial data \eqref{iei} by generalizing the techniques of \cite{damanik16}. As we have explained, this goal is very natural relative to the existing work and the absence of the spectral theory input for general values of $\mathfrak p$. The main results are Theorem~\ref{eethm} (existence), Theorem~\ref{euthm} (uniqueness), and Theorem~\ref{cccc} (convergence).

As discussed in \cite{damanik16}, the approach in that paper is inspired by and uses ideas from Christ \cite{christ07} and Kenig-Ponce-Vega \cite{KPV91}.

Let us first state the local existence theorem:

\begin{theo}[Existence]\label{eethm}
If the Fourier coefficients $\{c(n)\}$ of $u_0$ are exponentially decaying, that is, there exist $\mathcal A>0$ and $0<\kappa\leq1$ such that
$$
|c(n)|\leq\mathcal A^{\frac{1}{\mathfrak p-1}}e^{-\kappa|n|},\quad \forall n\in\mathbb Z^\nu,
$$
and the wave vector $\omega\in\mathbb R^\nu$ is rationally independent or non-resonant, then
there exists a positive number
$$
t_0 = \min\left\{\frac{\kappa^{(\mathfrak p-1)\nu+1}}{2^{\mathfrak p+1}6^{(\mathfrak p-1)\nu+1}\mathcal A|\omega|},\frac{\kappa^{(\mathfrak p-1)\nu+1}}{2(\mathfrak p-1)e\Box^{\mathfrak p-1}12^{(\mathfrak p-1)\nu+1}|\omega|}\right\} > 0
$$
such that there exists a spatially $\omega$-quasi-periodic function
$$
u^\dag(t,x)=\sum_{n\in\mathbb Z^\nu}c^\dag(t,n)e^{{\rm i}(n\cdot\omega)x},\quad 0\leq t<t_0,
x\in\mathbb R
$$
that solves $\mathfrak p$-gKdV \eqref{ppkdv}, satisfies the initial condition \eqref{iei}, and has exponentially decaying Fourier coefficients $\{c^\dag(t,n)\}$ as well, that is,
$$
|c^\dag(t,n)|\leq\Box e^{-\frac{\kappa}{2}|n|},
$$
where $\Box=2(6\kappa^{-1})^{\nu}\mathcal A^{\frac{1}{\mathfrak p-1}}$.
\end{theo}

We emphasize that the exponential decay assumption for the Fourier coefficients $c(\cdot)$ is crucial, but also that there is no smallness assumption on the coefficients and no Diophantine assumption on $\omega$, aside from rational independence. Moreover, the time of local existence $t_0$, is explicit in terms of the decay parameters of the Fourier coefficients and the size of the frequency vector $\omega$. On the other hand, the estimate for the exponential decay of the Fourier coefficients at positive times that follows from the proof is \emph{weaker} than the one that is assumed about the initial condition. This means that one cannot readily iterate the argument to cover arbitrarily long time intervals. This is precisely the deficit that is addressed, in the KdV case $\mathfrak p = 2$, via the spectral correspondence worked out in \cite{damanik14}.

Next we state our local uniqueness result:

\begin{theo}[Uniqueness]\label{euthm}
Consider the following two spatially $\omega$-quasi-periodic functions
\begin{align*}
u_1(t,x)=\sum_{n\in\mathbb Z^\nu}c(t,n)e^{{\rm i}(n\cdot\omega)x}, \quad 0\leq t<t_1\leq t_0, x\in\mathbb R;\\
u_2(t,x)=\sum_{n\in\mathbb Z^\nu}b(t,n)e^{{\rm i}(n\cdot\omega)x}, \quad 0\leq t<t_2\leq t_0, x\in\mathbb R
\end{align*}
with the following hypotheses:
\begin{itemize}

\item $u_1$ and $u_2$ satisfy $\mathfrak p$-gKdV \eqref{ppkdv}, that is,
\begin{align*}
c(t,n)&=e^{{\rm i}(n\cdot\omega)^3t}c(0,n)-\frac{{\rm i}n\cdot\omega}{\mathfrak p}\int_0^te^{{\rm i}(n\cdot\omega)^3(t-\tau)}\sum_{n_1,\cdots,n_{\mathfrak p}\in\mathbb Z^\nu:n_1+\cdots+n_{\mathfrak p}=n}\prod_{j=1}^{\mathfrak p}c(\tau,n_j){\rm d}\tau;\\
b(t,n)&=e^{{\rm i}(n\cdot\omega)^3t}b(0,n)-\frac{{\rm i}n\cdot\omega}{\mathfrak p}\int_0^te^{{\rm i}(n\cdot\omega)^3(t-\tau)}\sum_{n_1,\cdots,n_{\mathfrak p}\in\mathbb Z^\nu:n_1+\cdots+n_{\mathfrak p}=n}\prod_{j=1}^{\mathfrak p}b(\tau,n_j){\rm d}\tau.
\end{align*}

\item $u_1$ and $u_2$ have the same spatially quasi-periodic initial data, that is $c(0,n)=b(0,n)$ for all $n\in\mathbb Z^\nu$;

\item $c(t,n)$ and $b(t,n)$ are exponentially decaying uniformly in $t$, that is, there exist $\Box>0$ and $0<\rho\leq1$ such that
$$
|c(t,n)|\leq\Box e^{-\rho|n|}, \quad |b(t,n)|\leq\Box e^{-\rho|n|}.
$$
\end{itemize}
If the Cauchy problem \eqref{ppkdv}--\eqref{iei} satisfies the above hypotheses, then
$$
u_1(t,x) \equiv u_2(t,x), \quad \forall 0 \leq t < \min \left\{ t_1, t_2, \frac{1}{2(\mathfrak p-1) e \Box^2(12\rho^{-1})^{(\mathfrak p-1)\nu+1}|\omega|} \right\}~\text{and}~x \in \mathbb R.
$$
\end{theo}

\begin{rema}
{\rm (a)} The case of $\mathfrak p=2$ is exactly the local result of \cite{damanik16}.
\\[1mm]
{\rm (b)} Recall that it is impossible to prove an unconditional uniqueness result, that is, uniqueness can only be shown to hold in a class of functions; compare \cite{CK89}. Then, the larger the class of functions in which a uniqueness statement can be shown, the stronger the result will be. Given this general observation, let us point out that here and in \cite{damanik16}, uniqueness is shown in a class of quasi-periodic functions with fixed frequency vector and fixed exponential decay rate. By contrast, the uniqueness result in \cite{BDGL} holds in a larger class of functions, but again, the tools employed in \cite{BDGL} are not available in the setting of this paper (i.e., for general $\mathfrak p$).
\end{rema}

In addition, we obtain the following convergence result.

\begin{theo}[Convergence]\label{cccc}
All the series involved converge absolutely and uniformly.
\end{theo}

\begin{proof}
On the one hand (for the part of spatial derivative), one has
\begin{align*}
\partial_x^{\sharp}u^\dag(t,n)=\sum_{n\in\mathbb Z^\nu}({\rm i}n\cdot\omega)^{\sharp}c^\dag(t,n)e^{{\rm i}(n\cdot\omega)x}, \quad \sharp=1,2,3.
\end{align*}
By the exponential decay property of $c^\dag(t,n)$, for all $\sharp=1,2,3$, we have
\begin{eqnarray*}
|\partial_x^\sharp u^\dag(t,n)|&\leq&\Box|\omega|^\sharp\sum_{n\in\mathbb Z^\nu}|n|^\sharp e^{-\frac{\kappa}{2}|n|}\\
&\leq&\Box|\omega|^\sharp\sum_{n\in\mathbb Z^\nu}|n|^\sharp e^{-\frac{\kappa}{4}|n|}\cdot e^{-\frac{\kappa}{4}|n|}\\
&\lesssim&\sum_{n\in\mathbb Z^\nu}e^{-\frac{\kappa}{4}|n|}\\
&\leq&(12\kappa^{-1})^\nu.
\end{eqnarray*}
On the other hand (for the part of time derivative), one has
\begin{eqnarray*}
\partial_tu^\dag(t,x)&=&\sum_{n\in\mathbb Z^\nu}\partial_tc^\dag(t,n)e^{{\rm i}(n\cdot\omega)x}\\
&=&\sum_{n\in\mathbb Z^\nu}\left\{{\rm i}(n\cdot\omega)^3c^\dag(t,n)-\frac{{\rm i}n\cdot\omega}{\mathfrak p}\sum_{n_1,\cdots,n_\mathfrak p\in\mathbb Z^\nu:\sum_{j=1}^{\mathfrak p}n_j=n}\prod_{j=1}^{\mathfrak p}c^\dag(t,n_j)\right\}
e^{{\rm i}(n\cdot\omega)x}.
\end{eqnarray*}
By the exponential decay of $c(t,n)$, we have
\begin{eqnarray*}
&&|\partial_tu^\dag(t,x)|\\
&\leq&|\omega|^3\sum_{n\in\mathbb Z^\nu}|n|^3|c^\dag(t,n)|+\frac{|\omega|}{\mathfrak p}\sum_{n\in\mathbb Z^\nu}|n|\sum_{n_1,\cdots,n_\mathfrak p\in\mathbb Z^\nu:\sum_{j=1}^{\mathfrak p}n_j=n}\prod_{j=1}^{\mathfrak p}|c^\dag(t,n_j)|\\
&\leq&\Box|\omega|^3\sum_{n\in\mathbb Z^\nu}|n^3 e^{-\frac{\kappa}{2}|n|}+\frac{\Box^{\mathfrak p}|\omega|}{\mathfrak p}\sum_{n\in\mathbb Z^\nu}|n|\sum_{n_1,\cdots,n_\mathfrak p\in\mathbb Z^\nu:\sum_{j=1}^{\mathfrak p}n_j=n}\prod_{j=1}^{\mathfrak p}e^{-\frac{\kappa}{2}|n_j|}\\
&\leq&\Box|\omega|^3\sum_{n\in\mathbb Z^\nu}|n|^3e^{-\frac{\kappa}{4}|n|}\cdot e^{-\frac{\kappa}{4}|n|}+\frac{\Box^{\mathfrak p}|\omega|}{\mathfrak p}\sum_{n\in\mathbb Z^\nu}|n|e^{-\frac{\kappa}{4}|n|}\sum_{n_1,\cdots,n_\mathfrak p\in\mathbb Z^\nu:\sum_{j=1}^{\mathfrak p}n_j=n}\prod_{j=1}^{\mathfrak p}e^{-\frac{\kappa}{4}|n_j|}\\
&\lesssim&\sum_{n\in\mathbb Z^\nu}(|n|+|n|^3)e^{-\frac{\kappa}{4}|n|}\cdot e^{-\frac{\kappa}{4}|n|}\\
&\lesssim&(12\kappa^{-1})^\nu.
\end{eqnarray*}
Hence all the series involved converge absolutely and uniformly (due to the exponential decay of Fourier coefficients). This completes the proof of Theorem \ref{cccc}.
\end{proof}

\begin{rema}
As we have just seen, the infinite series converge absolutely and uniformly since the Fourier coefficients are exponentially decaying. Hence $\partial_tu$ and $\partial_x^{\#} u~(\#=1,2,3)$ are continuous for $0\leq t<t_0, x\in\mathbb R$, that is to say the solution we construct is a solution in the classical sense. Hence the formal derivations throughout this paper are reasonable.
\end{rema}
%\section{Innovation}

{\bf (Innovation)}~~We partly generalize the result of \cite{damanik16} from $\mathfrak p=2$ to $\mathfrak p\geq3$. In the process of generalization, we find that it is not a simple and straightforward extension. From the perspective of problems, what we study is a gKdV equation and it has more generality and complexity. From the perspective of results, our result includes and extends their local result.
%(we have not obtained the global result, even if we construct a Lax pair for $3$-gKdV).
From the perspective of techniques, we make some modifications to complete the proof.  More precisely, we list them as follows:

${\mathbf 1^{\circ}}$~~In our proof, we modify the power of the constant before the exponential factor assumed for the Fourier coefficients of spatially quasi-periodic initial data, and the definition of $\sigma$. We obtain a similar relation $$\sigma=\ell+\frac{1}{\mathfrak p-1}$$ between $\sigma$ and $\ell$. With the help of this relation, we can successfully complete all the estimates needed. Especially we can combine one sample term $\ast^{\sigma(\gamma)}$ and another sample term $\star^{\ell(\gamma)}$ into the final sample term $\blacksquare^{\ell(\gamma)}$ with a constant $\ast^{\frac{1}{\mathfrak p-1}}$, and complete the estimates for $\sum_{\gamma}\ast^{\sigma(\gamma)}\star^{\ell(\gamma)}$. Otherwise, it may be very difficult to complete these estimates due to the coupling of the indices $\sigma$ and $\ell$. To illustrate the importance of this relation, consider several nonlinear terms such as including $u^2$ and $u^3$. It is impossible to get a similar equality in this case. Obviously the relation for $k=1$ is unique. However, for $k\geq2$, since there are several nonlinear terms, the relation is not the same as for $k=1$. More precisely, we can make only one of the nonlinear terms be the same as the initial relation. Meanwhile the relations for the rest of the nonlinear terms are not the same as the initial relation. This makes the estimates impossible. Luckily, for a single nonlinear term, we can always modify some constant to obtain the relation. Furthermore, we have the exponential decay property and Cauchy property of the Picard sequence, and then complete the proof.

${\mathbf 2^{\circ}}$~~In the process of generalization from the classical convective term $u\partial_xu$ ($\mathfrak p=2$) to the general one, $u^{\mathfrak p-1}\partial_xu$ ($\mathfrak p\geq2$), a main obvious variation is the Descartes product from $2$ to $\mathfrak p$ and it directly leads to the generalization of the definition in \cite{damanik16} of
$$
\mathfrak T^{(k)},\mathfrak N^{(k,\gamma)},\mathfrak F^{(k,\gamma)},\mathfrak I^{(k,\gamma)},
\mathfrak C^{(k,\gamma)},\mathfrak B(n^{(k)}),\mathfrak R^{(k,\gamma)},e^{(k,\gamma)},\mathfrak g^{(k)},\mathbb B^{(k)}
$$
and so forth. It should be emphasized that this is very delicate and many technical difficulties secretly hide behind it. Some definitions of them are obtained only after a great deal of calculation. In addition, on the one hand, for the case of $\mathfrak p=3$, there will be more terms in a sample term
\[
(|n_1|+|n_2|+|n_3|)\cdot(|n_4|+|n_5|+|n_6|)\cdot(|n_7|+|n_8|+|n_9|),
\]
which is associated with the index set $\mathfrak R^{(2,(1,1,1))}$, or more complicated for the case of $\mathfrak p\geq 4$. On the other hand, in the process of estimating $c_k(t,n)-c_{k-1}(t,n)$, we need to interpolate more and more terms as $\mathfrak p$ increases.

{\bf (Skeleton of proof)}~~Before going into the proof of our main results, we give a brief introduction to it. For the sake of simplicity, we divide it into the case of $\mathfrak p=3$ and the general case $\mathfrak p\geq 2$. Generally speaking, we follow the steps in the following diagram to complete the proof of our main results in this paper.
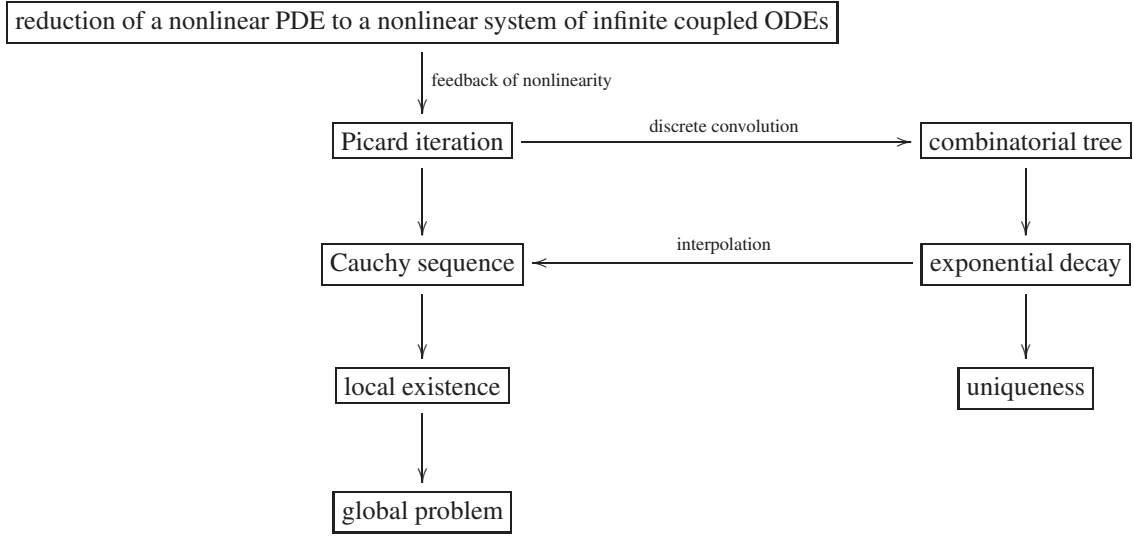
\[
\small
\xymatrix{
\boxed{\text{reduction of a nonlinear PDE to a nonlinear system of infinite coupled ODEs}}\ar[d]^{\text{feedback of nonlinearity}}\\
\boxed{\text{Picard iteration}}\ar[d]\ar[r]^{\text{discrete convolution}}&\boxed{\text{combinatorial tree}}\ar[d]\\
\boxed{\text{Cauchy sequence}}\ar[d]&\boxed{\text{exponential decay}}\ar[l]_{\text{interpolation}}\ar[d]\\
\boxed{\text{local existence}}\ar[d]&\boxed{\text{uniqueness}}\\
\boxed{\text{global problem}}
}
\]

${\mathbf 1^{\circ}}$~~(reduction of a PDE to a system of ODEs)~We formally expand the spatially quasi-periodic solution as a Fourier series and translate $\mathfrak p$-gKdV into an infinite system of nonlinear coupled ODEs \eqref{ded} under the non-resonance condition on the wave vector $\omega$. Hence the PDE can be seen as an infinite-dimensional ODE in the Fourier space.

${\mathbf 2^{\circ}}$~~(Picard iteration)~Due to the feedback of nonlinearity, we construct the Picard sequence \eqref{ppp} to solve the infinite system of coupled nonlinear ODEs.

${\mathbf 3^{\circ}}$~~(combinatorial tree)~Let $N_0=1$ and $N_k=1+N_{k-1}^{\mathfrak p}$, $k\geq1$. From the  formula of the Picard iteration, we know that $N_k$ corresponds to the number of the terms appearing in $c_k$, and consequently $N_k$ will be complicated due to the nonlinearity. This is caused by the operation of higher dimensional discrete convolution. In order to overcome it, we rewrite the Picard sequence \eqref{ppp} as the following combinatorial tree,
      \[c_{k}(t,n)=\sum_{\gamma\in\mathfrak T^{(k)}}\sum_{\substack{n^{(k)}\in\mathfrak N^{(k,\gamma)}\\\mu(n^{(k)})=n}}\mathfrak F^{(k,\gamma)}(n^{(k)})\mathfrak I^{(k,\gamma)}(t,n^{(k)})\mathfrak C^{(k,\gamma)}(n^{(k)}).\]

${\mathbf 4^{\circ}}$~~(exponentially decay)~With the help of combinatorial techniques, we prove that the Picard sequence is exponentially decaying.
  \begin{enumerate}
    \item We first estimate $\mathfrak F^{(k,\gamma)}(n^{(k)}), \mathfrak I^{(k,\gamma)}(t,n^{(k)})$ and $\mathfrak C^{(k,\gamma)}(n^{(k)})$ independently, and obtain the following results:
    \begin{align*}
    |\mathfrak F^{(k,\gamma)}(n^{(k)})|&\leq|\omega|^{\ell(\gamma)}\mathfrak B(n^{(k)}),\\
    |\mathfrak I^{(k,\gamma)}(t,n^{(k)})|&\leq\frac{t^{\ell(\gamma)}}{\mathfrak D(\gamma)},\\
    |\mathfrak C^{(k,\gamma)}(n^{(k)})|&\leq\mathcal A^{\sigma(\gamma)}e^{-\kappa|n|}.
    \end{align*}
    \item Using some combinatorial techniques and introducing a new variable $\alpha$ (an index of combination) to estimate $\mathfrak B(n^{k})$, we find that it can be controlled by the components of $n^{(k)}$ and $\alpha$, that is,
        \[\mathfrak B(n^{(k)})\leq \sum_{\alpha=(\alpha_i)_{1\leq i\leq(\mathfrak p-1)\sigma(\gamma)}\in\mathfrak R^{(k,\gamma)}}\prod_{i}\left|(n^{(k)})_{i}\right|^{\alpha_i}.\]
    \item We will encounter terms of the form
    \[\sum_{{n^{(k)}\in\mathfrak N^{(k,\gamma)}:\mu(n^{(k)})=n}}\prod_i|(n^{(k)})_i|^{\alpha_i}e^{-\kappa|(n^{(k)})_i|}.\]
    To further estimate these terms, we divide $e^{-\kappa|(n^{(k)})_i|}$ into two equivalent parts $e^{-\frac{\kappa}{2}|(n^{(k)})_i|}$. One is used to balance $|(n^{(k)})_i|^{\alpha}$, and consequently the above term can be bounded by \[\sum_{\alpha=(\alpha_i)_{1\leq i\leq(\mathfrak p-1)\sigma(\gamma)}\in\mathfrak R^{(k,\gamma)}}\prod_j\alpha_j!,\] and another is used to control $c_k(t,n)$.
    \item We will  prove that
    \[\sum_{\gamma\in\mathfrak T^{(k)}}\frac{{\bf t}^{\ell(\gamma)}}{\mathfrak D(\gamma)}\sum_{
\alpha=(\alpha_i)_{1\leq i\leq(\mathfrak p-1)\sigma(\gamma)}\in\mathfrak R^{(k,\gamma)}
}\prod_j\alpha_j!\]
is bounded if ${\bf t}$ lies in a suitable interval.
  \end{enumerate}

${\mathbf 5^{\circ}}$~~(Cauchy sequence)~Once we obtain the exponential decay property of the Picard sequence, we use the interpolation method to prove that it is fundamental. Most of the inequalities we use to this end are similar to those used in the estimates leading to exponential decay. One particularly challenging term to estimate is
\[
\sum_{\alpha\in\mathbb R^{(k)}}\prod_j\alpha_j!.
\]
This can be solved by introducing the non-injective mapping $\phi$.

%${\mathbf 6^{\circ}}$ (global problem)~Based on the result of local existence and uniqueness, we will apply the conservation of spectrum for Schr\"odinger operator to prove the global result. Luckily, we construct a Lax pair for $3$-gKdV, where the Lax pair is $L=-\partial^2-(\varphi^2+{\rm i}\varphi_x)$. However, we can not infer some information of $\varphi$ from $\varphi^2+{\rm i}\varphi_x$.??????????? {\color{red}I will add some discussion of the challenges in passing to a global result in the next iteration.}

\part{$3$-gKdV}

In this part we prove Theorems~\ref{eethm} and \ref{euthm} in the special case $\mathfrak p = 3$. On the one hand, this is a case of interest in its own right, as it corresponds to the Gardner equation. On the other hand, this is the simplest case beyond the known case $\mathfrak p = 2$, and hence the complexities of the proofs and intricacies in the notation will be significantly reduced, which we hope will allow the reader to follow the individual steps more easily. Once they have been understood, the extension to the case of general $\mathfrak p$, which will be carried out in the next part, will be easier to process.

\section{Preliminaries}

\subsection{Reduction of a PDE to a System of ODEs}

For a given spatially quasi-periodic function $v$ with wave vector $\omega\in\mathbb R^\nu$, consider the following linear problem %\cite{kpv91}%\cite{opial67jde,}
\begin{align}\label{ae}
\partial_tu+\partial_x^3u-v=0.
\end{align}
Assume that $u$ and $v$ have the following Fourier expansion,
\begin{align}
\label{pw} u(t,x)&=\sum_{n\in\mathbb Z^\nu}c(t,n)e^{{\rm i}(n\cdot\omega)x},\\
\label{ve} v(t,x)&=\sum_{n\in\mathbb Z^\nu}g(t,n)e^{{\rm i}(n\cdot\omega)x}.
\end{align}
If $\partial_t$ and $\partial_x$ commute with $\sum_{}$ in the case at hand, then
\begin{align}
\label{tbe}(\partial_tu)(t,x)&=\sum_{n\in\mathbb Z^\nu}\partial_tc(t,n)e^{{\rm i}(n\cdot\omega)x}, \\
\label{xe}(\partial_x^3u)(t,x)&=\sum_{n\in\mathbb Z^\nu}-{\rm i}(n\cdot\omega)^3c(t,n)e^{{\rm i}(n\cdot\omega)x}.
\end{align}
Inserting \eqref{ve}--\eqref{xe} into \eqref{ae}, we find
\begin{align*}
\sum_{n\in\mathbb Z^\nu}\left(\partial_{t}c(t,n)-{\rm i}(n\cdot\omega)^3c(t,n)-g(t,n)\right)e^{{\rm i}n\cdot\omega x}=0.
\end{align*}
If the wave vector $\omega$ is non-resonant, then we can derive the system of ODEs
\begin{align}\label{3ode}
\frac{{\rm d}}{{\rm d}t}c(t,n)-{\rm i}(n\cdot\omega)^3c(t,n)-g(t,n)=0, \quad \forall n\in\mathbb Z^\nu.%~(\text{\scriptsize by Lemma \autoref{lema1}}).
\end{align}
Here we replace ``$\partial_t$" by ``$\frac{{\rm d}}{{\rm d}t}$" in order to emphasize that \eqref{3ode} is an ODE for any given $n\in\mathbb Z^\nu$. Hence $3$-gKdV can be regarded as an infinite system of coupled nonlinear ODEs for the Fourier coefficients $c(t,n)$. Obviously, if $g(t,0)=0$, then $c(t,0)=c(0)$. Hence we solve $c(t,n)$ only for $n \in \mathbb Z^\nu\backslash\{0\}$. By the non-resonance condition on the wave vector $\omega$ and the variation of constants formula, the Fourier coefficients $c(t,n)$ for the spatially quasi-periodic solution of  \eqref{3ode} are determined by
\begin{align}\label{pe}
c(t,n)=e^{{\rm i}(n\cdot\omega)^3t}c(n)+\int_{0}^{t}e^{{\rm i}(n\cdot\omega)^3(t-\tau)}g(\tau,n){\rm d}\tau, \quad\forall n\in\mathbb Z^\nu\backslash\{0\}.
\end{align}

\subsection{Fourier Coefficients $g(t,n)$ of the Nonlinear Term $-u^2\partial_xu$}

It is easy to see that
\[
v=-u^2\partial_xu=\partial_x\left(-\frac{u^3}{3}\right).
\]
By the Cauchy product of infinite series (i.e., discrete convolution), we have
\[
u^3(t,x)=\sum_{n\in\mathbb Z^\nu}\sum_{p,q,r\in\mathbb Z^\nu:~p+q+r=n}c(t,p)c(t,q)c(t,r)e^{{\rm i}(n\cdot\omega)x}.
\]
Analogously, assuming the commutativity of $\partial_x$ and $\sum_{}$, one can derive that
\[
\partial_x\left(-\frac{u^3}{3}\right)=\sum_{n\in\mathbb Z^\nu}-\frac{{\rm i}n\cdot\omega}{3}\sum_{p,q,r\in\mathbb Z^\nu:~p+q+r=n}c(t,p)c(t,q)c(t,r)e^{{\rm i}(n\cdot\omega)x},
\]
that is,
\begin{align}\label{ge}
g(t,n)=-\frac{{\rm i}n\cdot\omega}{3}\sum_{p,q,r\in\mathbb Z^\nu:~p+q+r=n}c(t,p)c(t,q)c(t,r).
\end{align}

\subsection{Picard Iteration for $c(t,n)$}

Inserting \eqref{ge} into \eqref{pe}, one can see that the Fourier coefficients $c(t,n), n\in\mathbb  Z^\nu\backslash\{0\}$, of \eqref{3ode} are determined by the following integral equation,
\[
c(t,n)=e^{{\rm i}(n\cdot\omega)^3t}c(n)-\frac{{\rm i}n\cdot\omega}{3}\int_{0}^{t}e^{{\rm i}(n\cdot\omega)^3(t-\tau)}\sum_{p,q,r\in\mathbb Z^\nu:~ p+q+r=n}c(\tau,p)c(\tau,q)c(\tau,r){\rm d}\tau, \quad\forall n\in\mathbb Z^\nu\backslash\{0\}.
\]
Define the following Picard iteration,
\begin{align}
&c_k(t,n):=\nonumber\\
&
\begin{cases}
e^{{\rm i}(n\cdot\omega)^3t}c(n),&k=0;\\
\label{ppp}c_0(t,n)-\frac{{\rm i}n\cdot\omega}{3}\int_{0}^{t}e^{{\rm i}(n\cdot\omega)^3(t-\tau)}\sum_{p,q,r\in\mathbb Z^\nu:~ p+q+r=n}c_{k-1}(\tau,p)c_{k-1}(\tau,q)c_{k-1}(\tau,r){\rm d}\tau,&k\geq1.%,\quad k\in\mathbb N_+.
\end{cases}
\end{align}
If $\{c_k(t,n)\}$ is a Cauchy sequence, then we can define its limit as the Fourier coefficient $c(t,n)$. Furthermore, under some condition of $c(t,n)$, e.g.\ exponential decay in $n$ uniformly with respect to $t$, $\partial_t$ and $\partial_x$ do indeed commute with $\sum_{}$, see Theorem~\ref{cccc}, and hence the formal derivation above can be justified in this way.

\section{Estimates for $c_k(t,n)$}

In this section, we will use some combinatorial estimates to prove that $c_k(t,n)$ is exponentially decaying as well, assuming that $c(n)$ of the spatially quasi-periodic initial
datum is exponentially decaying.

\begin{prop}\label{thm}
If $c(n)$ is exponentially decaying in the sense that there exist two constants $\mathcal A>0$ and $0<\kappa\leq1$ such that
$$
|c(n)|\leq \mathcal A^{1/2}e^{-\kappa|n|}, \quad\forall n\in\mathbb Z^\nu,
$$
then for any $k=0,1,\cdots$, we have
\[
|c_{k}(t,n)|\leq\Box e^{-\frac{\kappa}{2}|n|}, \quad \text{\rm for all}\quad 0<t\leq\frac{\kappa^{2\nu+1}}{16\mathcal A|\omega|6^{2\nu+1}},
\]
where $\Box=2(6\kappa^{-1})^\nu\mathcal A^{\frac{1}{2}}.$
\end{prop}

Before proving Proposition~\ref{thm}, we discuss the following combinatorial technique and some useful lemmas; compare \cite{christ07, damanik16}.

\subsection{Combinatorial Tree for $c_{k}(t,n)$}

In order to show that the Picard sequence $c_k(t,n)$ is fundamental (i.e., a Cauchy sequence), we need to estimate some quantities, such as $|c_{k}(t,n)-c_{k-1}(t,n)|$, $|c_{k+\bigstar}(t,n) - c_k(t,n)|$ and $|c_k(t,n)|$. Notice that the right-hand side of the Picard iteration includes the operation of higher-dimensional discrete convolution, $c\ast c\ast c$~$($total distance$)$, and hence the resulting expressions will be very complicated after several iterations and we cannot present the iteration progress clearly.  Applying the combinatorial technique from \cite{christ07,damanik16}, we can reformulate $c_k(t,n)$ as a tree to overcome this difficulty, because in this representation the iteration progress can be seen more clearly.

Throughout of this paper, for a multi-index $\alpha=(\alpha_1,\cdots,\alpha_d)$, where $\alpha_j\in\mathbb N$, denote the length of $\alpha$ by $|\alpha|$, that is, $|\alpha| = \sum_{j=1}^{d} \alpha_j$. For a vector $n = (\clubsuit_1, \clubsuit_2) \in \copyright_1 \times \copyright_2$, set $|n| = |\clubsuit_1| + |\clubsuit_2|$ (we use the same symbol $|\cdot|$ for the length of a multi-index and the size of a vector; this should not lead to any confusion).

Define the function $\mu : (\mathbb R^\nu)^r \rightarrow \mathbb R^\nu$ by
$$
\mu(n^{(1)},\cdots,n^{(r)}) := \sum_{j=1}^{r} n^{(j)}
$$
and set
\begin{align*}
\mathfrak T^{(1)}&:=\{0,1\}, \quad \mathfrak T^{(k)}:=\{0\}\cup\mathfrak T^{(k-1)}\times\mathfrak T^{(k-1)}\times\mathfrak T^{(k-1)}, k\geq2, \\
\mathfrak N^{(k,\gamma)}&:=
\begin{cases}
\mathbb Z^\nu, &\gamma=0\in\mathfrak T^{(k)}, k\geq 1;\\
\mathbb Z^\nu\times\mathbb Z^\nu\times\mathbb Z^\nu, &\gamma=1\in\mathfrak T^{(1)};\\
\prod_{j=1}^3\mathfrak N^{(k-1,\gamma_j)}, &\gamma=(\gamma_j)_{1\leq j\leq3}\in\prod_{j=1}^3\mathfrak T^{(k-1)}, k\geq 2,
\end{cases}\\
\mathfrak F^{(k,\gamma)}(n^{(k)})&:=
\begin{cases}
1, &\gamma=0\in\mathfrak T^{(k)}, n^{(k)}\in\mathfrak N^{(k,0)}, k\geq 1;\\
-\frac{{\rm i}\mu(n^{(1)})\cdot\omega}{3}, &\gamma=1\in\mathfrak T^{(1)}, n^{(1)}\in\mathfrak N^{(1,1)};\\
-\frac{{\rm i}\mu(n^{(k)})\cdot\omega}{3}\prod_{j=1}^{3}\mathfrak F^{(k-1,\gamma_j)}(n_j), &\gamma=(\gamma_j)_{1\leq j\leq 3}\in\prod_{j=1}^3\mathfrak T^{(k-1)},\\
&n^{(k)}=(n_j)_{1\leq j\leq 3}\in\mathfrak N^{(k,\gamma)}, k\geq 2,
\end{cases}\\
\mathfrak I^{(k,\gamma)}(t,n^{(k)})&:=
\begin{cases}
e^{{\rm i}\left(\mu(n^{(k)})\cdot\omega\right)^3t}, &\gamma=0\in\mathfrak T^{(k)}, n^{(k)}\in\mathfrak N^{(k,0)},k\geq 1;\\
\int_0^te^{{\rm i}\left(\mu(n^{(1)})\cdot\omega\right)^3(t-\tau)}\prod_{j=1}^3e^{{\rm i}(n_j\cdot\omega)^3\tau}{\rm d}\tau, &\gamma=1\in\mathfrak T^{(1)},\\
&n^{(1)}=(n_j)_{1\leq j\leq3}\in\mathfrak N^{(1,1)};\\
\int_0^te^{{\rm i}\left(\mu(n^{(k)}\cdot\omega\right)^3(t-\tau)}\prod_{j=1}^3\mathfrak I^{(k-1,\gamma_j)}(\tau,n_j){\rm d}\tau, &\gamma=(\gamma_j)_{1\leq j\leq 3}\in\prod_{j=1}^3\mathfrak T^{(k-1)},\\
&n^{(k)}=(n_j)_{1\leq j\leq3}\in\mathfrak N^{(k,\gamma)}, k\geq 2,
\end{cases}\\
\mathfrak C^{(k,\gamma)}(t,n^{(k)})&:=
\begin{cases}
c(n^{(k)}), &\gamma=0\in\mathfrak T^{(k)}, n^{(k)}\in\mathfrak N^{(k,0)},k\geq 1;\\
\prod_{j=1}^3c(n_j), &\gamma=1\in\mathfrak T^{(1)}, n^{(1)}=(n_j)_{1\leq j\leq3}\in\mathfrak N^{(1,1)};\\
\prod_{j=1}^3\mathfrak C^{(k-1,\gamma_j)}(n_j), &\gamma=(\gamma_j)_{1\leq j\leq 3}\in\prod_{j=1}^3\mathfrak T^{(k-1)},\\
&n^{(k)}=(n_j)_{1\leq j\leq3}\in\mathfrak N^{(k,\gamma)}, k\geq 2.
\end{cases}
\end{align*}
With the help of these abstract symbols, the Picard sequence $c_k(t,n)$ can be represented as a tree. Roughly speaking, every term in $c_k(t,n)$ can be regarded as a branch $\gamma$ of the tree $\mathfrak T^{(k)}$, and hence $c_k(t,n)$ can be viewed as a sum in which the summation index is the branch $\gamma$, which runs over the tree $\mathfrak T^{(k)}$. Each branch $\gamma\in\mathfrak T^{(k)}$ can be split over the following restriction:
total distance $=n$.  We make this explicit in the following lemma.

\begin{lemm}\label{le2}
For the Picard sequence $c_k(t,n)$, we have
\begin{align}\label{lke}
c_k(t,n)=\sum_{\gamma\in\mathfrak T^{(k)}}\sum_{n^{(k)}\in\mathfrak N^{(k,\gamma)}:~\mu(n^{(k)})=n}\mathfrak F^{(k,\gamma)}(n^{(k)})\mathfrak I^{(k,\gamma)}(t,n^{(k)})\mathfrak C^{(k,\gamma)}(n^{(k)}),\quad \forall k\geq1.
\end{align}
\end{lemm}

\begin{proof}
For $k=1$, we have
\begin{eqnarray*}
c_1(t,n)&=&c_0(t,n)-\frac{{\rm i}n\cdot\omega}{3}\int_0^te^{{\rm i}(n\cdot\omega)^3(t-\tau)}\sum_{n_1,n_2,n_3\in\mathbb Z^\nu:~n_1+n_2+n_3=n}\prod_{j=1}^3c_0(n_j){\rm d}\tau\\
&=&\sum_{n^{(1)}\in\mathfrak N^{(1,0)}:~\mu(n^{(1)})=n}1\cdot e^{{\rm i}(\mu(n^{(1)})\cdot\omega)^3t}\cdot\mathfrak C^{(1,0)}(n^{(1)})\\
&+&\sum_{n^{(1)}\in\mathfrak N^{(1,1)}:~\mu(n^{(1)})=n}-\frac{{\rm i}\mu(n^{(1)})\cdot\omega}{3}\cdot\int_0^te^{{\rm i}(\mu(n^{(1)})\cdot\omega)^3(t-\tau)}{\rm d}\tau\cdot\mathfrak C^{(1,1)}(n^{(1)})\\
&=&\sum_{\gamma\in\mathfrak T^{(1)}}\sum_{n^{(1)}\in\mathfrak N^{(1,\gamma)}:~\mu(n^{(1)})=n}\mathfrak F^{(1,\gamma)}(n^{(1)})\mathfrak I^{(1,\gamma)}(t,n^{(1)})\mathfrak C^{(1,\gamma)}(n^{(1)}).
\end{eqnarray*}
This shows that \eqref{lke} holds for $k=1$. Let $k\geq2$. Assume that it holds for all $1\leq k^\prime\leq k-1$. For $k$, we have
\begin{align}\label{lhe}
c_k(t,n)=c_0(t,n)-\frac{{\rm i}n\cdot\omega}{3}\int_0^te^{{\rm i}(n\cdot\omega)^3(t-\tau)}\sum_{p,q,r\in\mathbb Z^\nu:~ p+q+r=n}c_{k-1}(\tau,p)c_{k-1}(\tau,q)c_{k-1}(\tau,r){\rm d}\tau.
\end{align}
After some calculation, we obtain
\begin{eqnarray}
\nonumber&&c_{k-1}(\tau,p)c_{k-1}(\tau,q)c_{k-1}(\tau,r)\\
\nonumber&=&\sum_{\gamma_1\in\mathfrak T^{(k-1)}}\sum_{\clubsuit\in\mathfrak N^{(k-1,\gamma_1)}:~\mu(\clubsuit)=p}\mathfrak F^{(k-1,\gamma_1)}(\clubsuit)\mathfrak I^{(k-1,\gamma_1)}(\tau,\clubsuit)\mathfrak C^{(k-1,\gamma_1)}(\clubsuit)\\
\nonumber&\times&\sum_{\gamma_2\in\mathfrak T^{(k-1)}}\sum_{\spadesuit\in\mathfrak N^{(k-1,\gamma_2)}:~\mu(\spadesuit)=q}\mathfrak F^{(k-1,\gamma_2)}(\spadesuit)\mathfrak I^{(k-1,\gamma_2)}(\tau,\spadesuit)\mathfrak C^{(k-1,\gamma_2)}(\spadesuit)\\
\nonumber&\times&\sum_{\gamma_3\in\mathfrak T^{(k-1)}}\sum_{\heartsuit\in\mathfrak N^{(k-1,\gamma_3)}:~\mu(\heartsuit)=r}\mathfrak F^{(k-1,\gamma_3)}(\heartsuit)\mathfrak I^{(k-1,\gamma_3)}(\tau,\heartsuit)\mathfrak C^{(k-1,\gamma_3)}(\heartsuit)\\
\nonumber&=&\sum_{\substack{\{(\gamma_1,\gamma_2,\gamma_3)\in\mathfrak T^{(k-1)}\\\times\mathfrak T^{(k-1)}\times\mathfrak T^{(k-1)}\}}}\sum_{\substack{\{(\clubsuit,\spadesuit,\heartsuit)\in\mathfrak N^{(k-1,\gamma_1)}\times\\\mathfrak N^{(k-1,\gamma_2)}\times\mathfrak N^{(k-1,\gamma_3)}:\\\mu(\clubsuit,\spadesuit,\heartsuit)=p+q+r\}}}\mathfrak F^{(k-1,\gamma_1)}(\clubsuit)\mathfrak F^{(k-1,\gamma_2)}(\spadesuit)\mathfrak F^{(k-1,\gamma_3)}(\heartsuit)\\
&\times&\mathfrak I^{(k-1,\gamma_1)}(\tau,\clubsuit)\mathfrak I^{(k-1,\gamma_2)}(\tau,\spadesuit)\mathfrak I^{(k-1,\gamma_3)}(\tau,\heartsuit)\nonumber\\
\label{be}&\times&\mathfrak C^{(k-1,\gamma_1)}(\clubsuit)\mathfrak C^{(k-1,\gamma_2)}(\spadesuit)\mathfrak C^{(k-1,\gamma_3)}(\heartsuit).
\end{eqnarray}
Inserting \eqref{be} into \eqref{lhe}, we find
\begin{eqnarray*}
&&c_k(t,n)\\
&=&c_0(t,n)+\sum_{\substack{\{(\gamma_1,\gamma_2,\gamma_3)\in\mathfrak T^{(k-1)}\\\times\mathfrak T^{(k-1)}\times\mathfrak T^{(k-1)}\}}}\sum_{\substack{\{p,q,r\in\mathbb Z^\nu:\\p+q+r=n\}}}\sum_{\substack{\{(\clubsuit,\spadesuit,\heartsuit)\in\mathfrak N^{(k-1,\gamma_1)}\times\\\mathfrak N^{(k-1,\gamma_2)}\times\mathfrak N^{(k-1,\gamma_3)}:\\\mu(\clubsuit,\spadesuit,\heartsuit)=p+q+r\}}}\mathfrak C^{(k-1,\gamma_1)}(\clubsuit)\mathfrak C^{(k-1,\gamma_2)}(\spadesuit)\mathfrak C^{(k-1,\gamma_3)}(\heartsuit)\\
&\times&\int_0^te^{{\rm i}(\mu(\clubsuit,\spadesuit,\heartsuit)\cdot\omega)^3(t-\tau)}\mathfrak I^{(k-1,\gamma_1)}(\tau,\clubsuit)\mathfrak I^{(k-1,\gamma_2)}(\tau,\spadesuit)\mathfrak I^{(k-1,\gamma_3)}(\tau,\heartsuit){\rm d}\tau\\
&\times&\left\{-\frac{{\rm i}\mu(\clubsuit,\spadesuit,\heartsuit)\cdot\omega}{3}\right\}\mathfrak F^{(k-1,\gamma_1)}(\clubsuit)\mathfrak F^{(k-1,\gamma_2)}(\spadesuit)\mathfrak F^{(k-1,\gamma_3)}(\heartsuit)\\
&=&\sum_{n^{(k)}\in\mathfrak N^{(k,0)}:~\mu(n^{(k)})=n}\mathfrak F^{(k,0)}(n^{(k)})\mathfrak I^{(k,0)}(t,n^{(k)})\mathfrak C^{(k,0)}(n^{(k)})\\
&+&\sum_{n^{(k)}\in\mathfrak N^{(k,\gamma)},\gamma\in\mathfrak T^{(k)}\backslash\{0\}:~\mu(n^{(k)})=n}\mathfrak F^{(k,\gamma)}(n^{(k)})\mathfrak I^{(k,\gamma)}(t,n^{(k)})\mathfrak C^{(k,\gamma)}(n^{(k)})\\
&=&\sum_{\gamma\in\mathfrak T^{(k)}}\sum_{n^{(k)}\in\mathfrak N^{(k,\gamma)}:~\mu(n^{(k)})=n}\mathfrak F^{(k,\gamma)}(n^{(k)})\mathfrak I^{(k,\gamma)}(t,n^{(k)})\mathfrak C^{(k,\gamma)}(n^{(k)}).
\end{eqnarray*}
This completes the proof of Lemma \ref{le2}.
\end{proof}

\begin{rema}
By the definition of $\mathfrak C$, $\mathfrak F$, and $\mathfrak I$, it is clear that $\mathfrak C$ is closely connected with the Fourier coefficients $c$ of the quasi-periodic initial data, and $\mathfrak F\cdot\mathfrak I$ is the rest of the Picard sequence. Moreover, $\mathfrak F$ is independent of the time variable, while $\mathfrak I$ does depend on the time variable. This separation property is one of the advantages provided by the combinatorial technique we employ to estimate the formal power series at hand.
\end{rema}

\subsection{Estimates of $\mathfrak F$, $\mathfrak I$ and $\mathfrak C$}

\begin{lemm}[Estimates of $\mathfrak F$]\label{fe}
We have
\begin{align}\label{dcse}
|\mathfrak F^{(k,\gamma)}(n^{(k)})|\leq|\omega|^{\ell(\gamma)}\mathfrak B(n^{(k)}),\quad\forall k\geq1,
\end{align}
where
$$
\ell(\gamma) :=
\begin{cases}
0&\gamma=0\in\mathfrak T^{(k)},k\geq1;\\
1&\gamma=1\in\mathfrak T^{(1)};\\
\ell(\gamma_1)+\ell(\gamma_2)+\ell(\gamma_3)+1,&\gamma=(\gamma_1,\gamma_2,\gamma_3)\in\mathfrak T^{(k-1)}\times\mathfrak T^{(k-1)}\times\mathfrak T^{(k-1)}, k\geq2,
\end{cases}
$$
and
$$
\mathfrak B(n^{(k)}) :=
\begin{cases}
1&\gamma=0\in\mathfrak T^{(k)},k\geq1;\\
|\mu(n^{(1)})|&\gamma=1\in\mathfrak T^{(1)};\\
|\mu(n^{(k)})|\mathfrak B(n_1)\mathfrak B(n_2)\mathfrak B(n_3),&\gamma=(\gamma_j)_{1\leq j\leq 3}\in\prod_{j=1}^3\mathfrak T^{(k-1)},\\
&n^{(k)}=(n_j)_{1\leq j\leq3}\in\mathfrak N^{(k,\gamma)}, k\geq 2.
\end{cases}
$$
\end{lemm}

\begin{proof}
For all $k\geq1$, $\gamma = 0 \in \mathfrak T^{(k)}$ and $n^{(k)} \in \mathfrak N^{(k,0)} = \mathbb Z^\nu$, we have $|\mathfrak F^{(k,0)}(n^{(k)})| = 1 = |\omega|^{\ell(0)}| \mathfrak B(n^{(k)})|$.

For $k=1$ and $\gamma = 1 \in \mathfrak T^{(k)}$, we have
\begin{eqnarray*}
|\mathfrak F^{(1,1)}(n^{(1)})|&=&\left|-\frac{i\mu(n^{(1)})\cdot\omega}{3}\right|\\
&\leq&|\omega||\mu(n^{(1)})|\\
&=&|\omega|^{\ell(1)}|\mu(n^{(1)})|.
\end{eqnarray*}

Let $k\geq2$. Assume that \eqref{dcse} holds for all $1\leq k^\prime\leq k-1$. For $k$, $\gamma=(\gamma_1,\gamma_2,\gamma_3)\in\mathfrak T^{(k-1)}\times\mathfrak T^{(k-1)}\times\mathfrak T^{(k-1)}$, and $n^{(k)}\in\mathfrak N^{(k-1,\gamma_1)}\times\mathfrak N^{(k-1,\gamma_2)}\times\mathfrak N^{(k-1,\gamma_3)}$, we have
\begin{eqnarray*}
&&|\mathfrak F^{(k,\gamma)}(t,n^{(k)})|\\
&=&\left|-\frac{{\rm i}\mu(n^{(k)})\cdot\omega}{3}\mathfrak F^{(k-1,\gamma_1)}(t,n_1)\mathfrak F^{(k-1,\gamma_2)}(t,n_2)\mathfrak F^{(k-1,\gamma_3)}(t,n_3)\right|\\
&\leq&\left|-\frac{{\rm i}\mu(n^{(k)})\cdot\omega}{3}\right||\mathfrak F^{(k-1,\gamma_1)}(t,n_1)||\mathfrak F^{(k-1,\gamma_2)}(t,n_2)||\mathfrak F^{(k-1,\gamma_3)}(t,n_3)|\\
&\leq&|\omega||\mu(n^{(k)})|\cdot|\omega|^{\ell(\gamma_1)}\mathfrak B(n_1)\cdot|\omega|^{\ell(\gamma_2)}\mathfrak B(n_2)\cdot|\omega|^{\ell(\gamma_3)}\mathfrak B(n_3)\\
&=&|\omega|^{\ell(\gamma_1)+\ell(\gamma_2)+\ell(\gamma_3)+1}|\mu(n^{(k)})|\mathfrak B(n_1)\mathfrak B(n_2)\mathfrak B(n_3)\\
&=&|\omega|^{\ell(\gamma)}\mathfrak B(n^{(k)}).
\end{eqnarray*}
This completes the proof of Lemma \ref{fe}.
\end{proof}

\begin{lemm}[Estimates of $\mathfrak I$]\label{iighe}
We have
\begin{align}\label{gfcse}
|\mathfrak I^{(k,\gamma)}(t,n^{(k)})|\leq\frac{t^{\ell(\gamma)}}{\mathfrak D(\gamma)},\quad \forall k\geq1,
\end{align}
where
\begin{align*}
\mathfrak D(\gamma)&:=
\begin{cases}
1,&\gamma=0\in\mathfrak T^{(k)},k\geq1;\\
1,&\gamma=1\in\mathfrak T^{(1)};\\
\ell(\gamma)\mathfrak D(\gamma_1)\mathfrak D(\gamma_2)\mathcal D(\gamma_3),&\gamma=(\gamma_j)_{1\leq j\leq 3}\in\prod_{j=1}^3\mathfrak T^{(k-1)}, k\geq 2.
\end{cases}
\end{align*}
\end{lemm}

\begin{proof}
For all $k\geq1$,  $\gamma = 0 \in \mathfrak T^{(k)}$ and $n^{(k)} \in \mathfrak N^{(k,0)} = \mathbb Z^\nu$, we have $|\mathfrak I^{(k,0)}(n^{(1)})| = |e^{{\rm i}(n^{(1)}\cdot\omega)^3t}| = 1 = |\omega|^{\ell(0)}| \mathfrak B(n^{(k)})|$.

For $k=1$, $\gamma = 1 \in \mathfrak T^{(k)}$, and $n^{(1)} = (n_1,n_2,n_3) \in \mathfrak N^{(1,1)} = \mathbb Z^\nu \times \mathbb Z^\nu \times \mathbb Z^\nu$, we have
\begin{eqnarray*}
|\mathfrak I^{(1,1)}(t,n^{(1)})| & = & \left| \int_0^te^{{\rm i}(\mu(n^{(1)})\cdot\omega)^3 (t-\tau)} \prod_{j=1}^3 e^{{\rm i}(n_j\cdot\omega)^3\tau}{\rm d}\tau \right| \\
& \leq & t \\
& = & \frac{t^{\ell(1)}}{\mathfrak D(1)}.
\end{eqnarray*}

Let $k\geq2$. Assume that \eqref{gfcse} holds for all $1\leq k^\prime\leq k-1$. Then for $k, \gamma = (\gamma_1,\gamma_2,\gamma_3) \in \mathfrak T^{(k-1)} \times \mathfrak T^{(k-1)} \times \mathfrak T^{(k-1)}$ and $n^{(k)} = (n_1,n_2,n_3) \in \mathfrak N^{(k-1,\gamma_1)} \times \mathfrak N^{(k-1,\gamma_2)} \times \mathfrak N^{(k-1,\gamma_3)}$, we have
\begin{eqnarray*}
|\mathfrak I^{(k,\gamma)}(t,n^{(k)})|&=&\left|\int_0^te^{{\rm i}(\mu(n^{(k)}) \cdot \omega)^3(t-\tau)}\prod_{j=1}^3\mathfrak I^{(k-1,\gamma_j)}(\tau,n_j){\rm d}\tau\right|\\
&\leq&\int_0^t\left|e^{{\rm i}(\mu(n^{(k)}) \cdot \omega)^3(t-\tau)}\right|\prod_{j=1}^3\left|\mathfrak I^{(k-1,\gamma_j)}(\tau,n_j)\right|{\rm d}\tau\\
&\leq&\int_0^t\prod_{j=1}^3\frac{\tau^{\ell(\gamma_j)}}{\mathfrak D(\gamma_j)}{\rm d}\tau\\
&=&\frac{1}{\prod_{j=1}^3\mathfrak D(\gamma_j)}\int_0^t\tau^{\sum_{j=1}^3\ell(\gamma_j)}{\rm d}\tau\\
&=&\frac{t^{\sum_{j=1}^3\ell(\gamma_j)+1}}{\left(\sum_{j=1}^3\ell(\gamma_j)+1\right)\mathfrak D(\gamma_1)\mathfrak D(\gamma_2)\mathfrak D(\gamma_3)}\\
&=&\frac{t^{\ell(\gamma)}}{\mathfrak D(\gamma)}.
\end{eqnarray*}
This completes the proof of Lemma \ref{iighe}.
\end{proof}

\begin{lemm}[Estimates of $\mathfrak C$]\label{ee}
If
\[
|c(n)| \leq \mathcal A^{\frac{1}{2}}e^{-\kappa|n|},
\]
then
\begin{align}\label{ceg}
|\mathfrak C^{(k,\gamma)}(n^{(k)})|\leq\mathcal A^{\sigma(\gamma)}e^{-\kappa|n^{(k)}|},\quad \forall k \geq 1,
\end{align}
where
\begin{align*}
\sigma(\gamma):=
\begin{cases}
\frac{1}{2},&\gamma=0\in\mathfrak T^{(k)},k\geq1;\\
\frac{3}{2},&\gamma=1\in\mathfrak T^{(1)};\\
\sigma(\gamma_1)+\sigma(\gamma_2)+\sigma(\gamma_3),&\gamma=(\gamma_1,\gamma_2,\gamma_3)\in\mathfrak T^{(k-1)}\times\mathfrak T^{(k-1)}\times\mathfrak T^{(k-1)}, k\geq2.
\end{cases}
\end{align*}
\end{lemm}

\begin{proof}
For all $k\geq1$,  $\gamma = 0 \in \mathfrak T^{(k)}$ and $n^{(k)} \in \mathfrak N^{(k,0)}$, we have
\[
|\mathfrak C^{(k,0)}(n^{(k)})| = |c(n^{(k)})| \leq \mathcal A^{\frac{1}{2}}e^{-\kappa|n^{(k)}|} = \mathcal A^{\sigma(0)} e^{-\kappa|n^{(k)}|}.
\]

For $k=1$, $\gamma = 1 \in \mathfrak T^{(1)}$, and $n^{(1)} = (n_1,n_2,n_3) \in \mathfrak N^{(1,1)} = \mathbb Z^\nu \times \mathbb Z^\nu \times \mathbb Z^\nu$, we have
\begin{eqnarray*}
|\mathfrak C^{(1,1)}(n^{(1)})|&=&|c(n_1)c(n_2)c(n_3)|\\
&\leq&|c(n_1)|\cdot|c(n_2)|\cdot|c(n_3)|\\
&\leq& \mathcal A^{\frac{1}{2}}e^{-\kappa|n_1|}\cdot\mathcal A^{\frac{1}{2}}e^{-\kappa|n_2|}\cdot\mathcal A^{\frac{1}{2}}e^{-\kappa|n_3|}\\
&=&\mathcal A^{\frac{3}{2}}e^{-\kappa(|n_1|+|n_2|+|n_3|)}\\
&=&\mathcal A^{\sigma(1)}e^{-\kappa|n^{(1)}|}.
\end{eqnarray*}

Let $k\geq2$. Assume that \eqref{ceg} is true for all $1\leq k^\prime\leq k-1$. Then, for $k, \gamma = (\gamma_1,\gamma_2,\gamma_3) \in \mathfrak T^{(k-1)} \times \mathfrak T^{(k-1)} \times \mathfrak T^{(k-1)}$ and $n^{(k)} = (n_1,n_2,n_3) \in \mathfrak N^{(k-1,\gamma_1)} \times \mathfrak N^{(k-1,\gamma_2)} \times \mathfrak N^{(k-1,\gamma_3)}$, we have
\begin{eqnarray*}
|\mathfrak C^{(k,\gamma)}(n^{(k)})|&=&|\mathfrak C^{(k-1,\gamma_1)}(n_1)|\cdot|\mathfrak C^{(k-1,\gamma_2)}(n_2)|\cdot|\mathfrak C^{k-1,\gamma_3}(n_3)|\\
&\leq&|\mathfrak C^{k-1,\gamma_1}(n_1)|\cdot|\mathfrak C^{(k-1,\gamma_2)}(n_2)|\cdot|\mathfrak C^{(k-1,\gamma_3)}(n_3)|\\
&\leq&\mathcal A^{\sigma(\gamma_1)}e^{-\kappa|n_1|}\cdot \mathcal A^{\sigma(\gamma_2)}e^{-\kappa|n_2|}\cdot \mathcal A^{\sigma(\gamma_3)}e^{-\kappa|n_3|}\\
&=&\mathcal A^{\sigma(\gamma_1)+\sigma(\gamma_2)+\sigma(\gamma_3)}e^{-\kappa(|n_1|+|n_2|+|n_3|)}\\
&=&\mathcal A^{\sigma(\gamma)}e^{-\kappa|n^{(k)}|}.
\end{eqnarray*}
This completes the proof of Lemma \ref{ee}.
\end{proof}

\begin{property}\label{re}
We have the following relation between $\sigma$ and $\ell$:
\[\sigma(\gamma)=\ell(\gamma)+\frac{1}{2},\quad \gamma\in\mathfrak T^{(k)}, \quad \forall k \geq 1.
\]
\begin{proof}
Obviously, this is true for $\gamma = 0 \in \mathfrak T^{(k)}, k\geq1,$ and $\gamma \in \mathfrak T^{(1)}$.

Let $k\geq2$. Assume it holds for all $1\leq k^\prime\leq k-1$. Then, for $k, \gamma = (\gamma_1, \gamma_2, \gamma_3) \in \mathfrak T^{(k-1)} \times \mathfrak T^{(k-1)} \times \mathfrak T^{(k-1)}$, we have
\begin{eqnarray*}
\sigma(\gamma)&=&\sigma(\gamma_1)+\sigma(\gamma_2)+\sigma(\gamma_3)\\
&=&\ell(\gamma_1)+\frac{1}{2}+\ell(\gamma_2)+\frac{1}{2}+\ell(\gamma_3)+\frac{1}{2}\\
&=&\ell(\gamma_1)+\ell(\gamma_2)+\ell(\gamma_3)+1+\frac{1}{2}\\
&=&\ell(\gamma)+\frac{1}{2}.
\end{eqnarray*}
This completes the proof of Property \ref{re}.
\end{proof}
\end{property}

\begin{rema}
The definition of $\sigma$ is not a direct generalization of \cite{damanik16}.
It and the relation between $\sigma$ and $\ell$ are different from the discussion in \cite{damanik16}. Using the original definition of $\sigma$, it may be impossible to obtain the exponential decay and Cauchy property of the Picard sequence. In order to complete the proof, it is crucial for us to modify the definition of $\sigma$ and obtain a suitable estimate for the Picard sequence.
\end{rema}

\begin{lemm}[Estimates of $|c_k(t,n)|$]\label{mth1}
For all $k\geq1$, we have
\begin{align}\label{ke}
|c_{k}(t,n)|\leq\mathcal A^{\frac{1}{2}}\sum_{\gamma\in\mathfrak T^{(k)}}\frac{(\mathcal A|\omega|t)^{\ell(\gamma)}}{\mathfrak D(\gamma)}
\sum_{n^{(k)}\in\mathfrak N^{(k,\gamma)}:~\mu(n^{(k)})=n}\mathfrak B(n^{(k)})e^{-\kappa|n^{(k)}|}.
\end{align}
\end{lemm}

\begin{proof}
We have
\begin{eqnarray}
\nonumber|c_{k}(t,n)|&=&\left|\sum_{\gamma\in\mathfrak T^{(k)}}\sum_{n^{(k)}\in\mathfrak N^{(k,\gamma)}:~\mu(n^{(k)})=n}\mathfrak F^{(k,\gamma)}(n^{(k)})\mathfrak I^{(k,\gamma)}(t,n^{(k)})\mathfrak C^{(k,\gamma)}(n^{(k)})\right|\\
\nonumber&\leq&\sum_{\gamma\in\mathfrak T^{(k)}}\sum_{n^{(k)}\in\mathfrak N^{(k,\gamma)}:~\mu(n^{(k)})=n}|\mathfrak F^{(k,\gamma)}(n^{(k)})||\mathfrak I^{(k,\gamma)}(t,n^{(k)})||\mathfrak C^{(k,\gamma)}(n^{(k)})|\\
\nonumber&\stackrel{(\text{\scriptsize Lemma \ref{fe}-\ref{ee}})}{\leq}& \sum_{\gamma\in\mathfrak T^{(k)}}\sum_{n^{(k)}\in\mathfrak N^{(k,\gamma)}:~\mu(n^{(k)})=n}|\omega|^{\ell(\gamma)}\mathfrak B(n^{(k)})\cdot\frac{t^{\ell(\gamma)}}{\mathfrak D(\gamma)}\cdot\mathcal A^{\sigma(\gamma)}e^{-\kappa|n^{(k)}|}~\\
\nonumber&\stackrel{(\text{\scriptsize Property \ref{re}})}{=}&\mathcal A^{\frac{1}{2}}\sum_{\gamma\in\mathfrak T^{(k)}}\frac{(\mathcal A|\omega|t)^{\ell(\gamma)}}{\mathfrak D(\gamma)}
\sum_{n^{(k)}\in\mathfrak N^{(k,\gamma)}:\mu(n^{(k)})=n}\mathfrak B(n^{(k)})e^{-\kappa|n^{(k)}|}~.
\end{eqnarray}
This completes the proof of Lemma \ref{mth1}.
\end{proof}

\subsection{Estimates of $\mathfrak B(n^{(k)})$ by the New Variable $\alpha_j$}

\begin{lemm}\label{iii}
We have
\begin{align}\label{bbe}
\mathfrak B(n^{(k)}) \leq \sum_{\alpha=(\alpha_i)_{1\leq i\leq2\sigma(\gamma)}\in\mathfrak R^{(k,\gamma)}} \prod_{i}\left|(n^{(k)})_{i}\right|^{\alpha_i}, \quad \forall k\geq1,
\end{align}
where
\begin{align*}
\mathfrak R^{(k,\gamma)}:=
\begin{cases}
\{0\in\mathbb Z\}, &\gamma=0\in\mathfrak T^{(k)},k\geq1;\\
\{(1,0,0),(0,1,0),(0,0,1)\}, &\gamma=1\in\mathfrak T^{(1)};\\
\prod_{j=1}^3\mathfrak R^{(k-1,\gamma_j)}+e^{(k,\gamma)}&\gamma=(\gamma_j)_{1\leq j\leq3}\in\prod_{j=1}^3\mathfrak T^{(k-1)}, k\geq2,
\end{cases}
\end{align*}
and
\[e^{(k,\gamma)}:=\{\alpha\in\mathbb Z^{2\sigma(\gamma)}:|\alpha|=1,\alpha_j\geq0\}.\]
\end{lemm}

\begin{proof}
For all $k\geq1$, $\gamma = 0 \in \mathfrak T^{(k)}$ and $n^{(k)} \in \mathfrak N^{(k,0)} = \mathbb Z^\nu$, we have $\mathfrak B(n^{(k)}) = 1$. Since $2 \sigma(0) = 1$ and $\mathfrak R^{(k,0)}=\{0\in\mathbb Z\}$, it follows that
$$
\sum_{\alpha=(\alpha_i)_{1\leq i\leq 2\sigma(0)} \in \mathfrak R^{(k,0)}} \prod_i |(n^{(k)})_i|^{\alpha_i} = 1.
$$

For $k=1$, $\gamma = 1 \in \mathfrak T^{(1)}$, and $n^{(1)} = (n_1,n_2,n_3) \in \mathfrak N^{(1,1,1)} = \mathbb Z^{\nu} \times \mathbb Z^{\nu} \times \mathbb Z^\nu$, we have $2 \sigma(1) = 2 \times 3/2 = 3$ and $\mathfrak R^{(1,1)}=\{(1,0,0),(0,1,0),(0,0,1)\}$. In this case we have
\begin{eqnarray*}
\mathfrak B(n^{(1)})&=&|\mu(n^{(1)})|\\
&=&|n_1+n_2+n_3|\\
&\leq& |n_1|+|n_2|+|n_3|\\
&=&\sum_{\alpha=(\alpha_i)_{1\leq i\leq2\sigma(1)}\in\mathfrak R^{(1,1)}}\prod_{i}|(n^{(1)})_i|^{\alpha_i}.
\end{eqnarray*}
This shows that \eqref{bbe} holds true for $\gamma = 0 \in \mathfrak T^{(k)}$, $k\geq1$, and $\gamma = 1 \in \mathfrak T^{(1)}$. In the proof of this case, we see that
every term in $\mathfrak B(n^{(1)})$ is associated with an index in $\mathfrak R^{(1,\gamma)}$, where $\gamma\in\mathfrak T^{(1)}$. This suggests to us how to deal with the case $k=2$ in a similar way.

Next, we will give a detailed but somewhat involved discussion for $k=2$. It should be emphasized that the following calculations are obtained by hand rather than a machine, because it will make clear how to find the appropriate way to assign each term in $\mathfrak B(n^{(2)})$ to an index in $\mathfrak R^{(2,\gamma)}$).
%In addition, we think a machine can not complete this task.
For the sake of readability, we do not give all the terms, and use only the index to represent them.

Notice that
\begin{align*}
\gamma = (\gamma_1,\gamma_2,\gamma_3) \in \prod_{j=1}^3 \mathfrak T^{(1)} = \left\{ (0,0,0); (1,0,0), (0,1,0), (0,0,1); (1,1,0), (1,0,1), (0,1,1); (1,1,1) \right\}.
\end{align*}
According to the length of these indices, it ie easy to see that there are four different kinds of indices: Case 0, Case 1 (including Case 1.1, Case 1.2 and Case 1.3), Case 2 (including Case 2.1, Case 2.2 and Case 2.3), and Case 3. Let us consider these cases.

Case 0. $\gamma=(0,0,0)$:
  $
  2\sigma(0,0,0)=3,
  $
  $n^{(2)}=(n_1;n_2;n_3)\in\mathfrak N^{(1,0)}\times\mathfrak N^{(1,0)}\times\mathfrak N^{(1,0)}=\mathbb Z^\nu\times\mathbb Z^\nu\times\mathbb Z^\nu$, and
 we have
\begin{eqnarray*}
\mathfrak B(n^{(2)})&=&|\mu(n^{(2)})|\mathfrak B(n_1)\mathfrak B(n_2)\mathfrak B(n_3)\\
&\leq&|n_1+n_2+n_3|\cdot 1\cdot1\cdot 1\\
&\leq&|n_1|+|n_2|+|n_3|\\
&=&\sum_{\alpha=(\alpha_i)_{1\leq i\leq2\sigma(0,0,0)}\in\mathfrak R^{(2,(0,0,0))}}\prod_i|(n^{(2)})_i|^{\alpha_i}.
\end{eqnarray*}

Case 1.1. $\gamma=(1,0,0)$:
  $
  2\sigma(1,0,0)
  =5,
  $ $n^{(2)}=(n_1,n_2,n_3;n_4;n_5)\in\mathfrak N^{(1,1)}\times\mathfrak N^{(1,0)}\times\mathfrak N^{(1,0)}=(\mathbb Z^\nu\times\mathbb Z^\nu\times\mathbb Z^\nu)\times\mathbb Z^\nu\times\mathbb Z^\nu$, and
one can derive that
\begin{eqnarray*}
\mathfrak B(n^{(2)})&=&|\mu(n^{(2)})|\mathfrak B(n_1,n_2,n_3)\mathfrak B(n_4)\mathfrak B(n_5)\\
&\leq&|\mu((n_1,n_2,n_3;n_4;n_5))|\cdot|\mu(n_1,n_2,n_3)|\cdot 1\cdot1\\
&\leq&(|n_1|+|n_2|+|n_3|+|n_4|+|n_5|)\cdot(|n_1|+|n_2|+|n_3|)\\
&=&\sum_{\alpha=(\alpha_i)_{1\leq i\leq2\sigma(1,0,0)}\in\mathfrak R^{(2,(1,0,0))}}\prod_i|(n^{(2)})_i|^{\alpha_i}.
\end{eqnarray*}

Case 1.2. $\gamma=(0,1,0)$:
   $
  2\sigma(0,1,0)
  =5,
  $ $n^{(2)}=(n_1;n_2,n_3,n_4;n_5)\in\mathfrak N^{(1,0)}\times\mathfrak N^{(1,1)}\times\mathfrak N^{(1,0)}=\mathbb Z^\nu\times(\mathbb Z^\nu\times\mathbb Z^\nu\times\mathbb Z^\nu)\times\mathbb Z^\nu$, and
one has
\begin{eqnarray*}
\mathfrak B(n^{(2)})&=&|\mu(n^{(2)})|\mathfrak B(n_1)\mathfrak B(n_2,n_3,n_4)\mathfrak B(n_5)\\
&\leq&|\mu((n_1;n_2,n_3,n_4;n_5))|\cdot1\cdot|\mu(n_2,n_3,n_4)|\cdot 1\\
&\leq&(|n_1|+|n_2|+|n_3|+|n_4|+|n_5|)\cdot(|n_2|+|n_3|+|n_4|)\\
%&=&|n_1|^1|n_2|^1|n_3|^0|n_4|^0|n_5|^0+|n_1|^1|n_2|^0|n_3|^1|n_4|^0|n_5|^0+|n_1|^1|n_2|^0|n_3|^0|n_4|^1|n_5|^0\\
%&+&|n_1|^0|n_2|^2|n_3|^0|n_4|^0|n_5|^0+|n_1|^0|n_2|^1|n_3|^1|n_4|^0|n_5|^0+|n_1|^0|n_2|^1|n_3|^0|n_4|^1|n_5|^0\\
%&+&|n_1|^0|n_2|^1|n_3|^1|n_4|^0|n_5|^0+|n_1|^0|n_2|^0|n_3|^2|n_4|^0|n_5|^0+|n_1|^0|n_2|^0|n_3|^1|n_4|^1|n_5|^0\\
%&+&|n_1|^0|n_2|^1|n_3|^0|n_4|^1|n_5|^0+|n_1|^0|n_2|^0|n_3|^1|n_4|^1|n_5|^0+|n_1|^0|n_2|^0|n_3|^0|n_4|^2|n_5|^0\\
%&+&|n_1|^0|n_2|^1|n_3|^0|n_4|^0|n_5|^1+|n_1|^0|n_2|^0|n_3|^1|n_4|^0|n_5|^1+|n_1|^0|n_2|^0|n_3|^0|n_4|^1|n_5|^1\\
&=&\sum_{\alpha=(\alpha_i)_{1\leq i\leq2\sigma(0,1,0)}\in\mathfrak R^{(2,(0,1,0))}}\prod_i|(n^{(2)})_i|^{\alpha_i}.
\end{eqnarray*}

Case 1.3. $\gamma=(0,0,1)$:
   $
  2\sigma(0,0,1)
  %=2(\sigma(0)+\sigma(0)+\sigma(1))
%  =2(1/2+1/2+3/2)
  =5,
  $ $n^{(2)}=(n_1;n_2;n_3,n_4,n_5)\in\mathfrak N^{(1,0)}\times\mathfrak N^{(1,0)}\times\mathfrak N^{(1,1)}=\mathbb Z^\nu\times\mathbb Z^\nu\times(\mathbb Z^\nu\times\mathbb Z^\nu\times\mathbb Z^\nu)$, and we have
\begin{eqnarray*}
\mathfrak B(n^{(2)})&=&|\mu(n^{(2)})|\mathfrak B(n_1)\mathfrak B(n_2)\mathfrak B(n_3,n_4,n_5)\\
&\leq&|\mu((n_1;n_2;n_3,n_4,n_5))|\cdot 1\cdot1\cdot|\mu(n_3,n_4,n_5)|\\
&\leq&(|n_1|+|n_2|+|n_3|+|n_4|+|n_5|)\cdot(|n_3|+|n_4|+|n_5|)\\
%&=&|n_1|^2|n_2|^0|n_3|^0|n_4|^0|n_5|^0+|n_1|^1|n_2|^1|n_3|^0|n_4|^0|n_5|^0+|n_1|^1|n_2|^0|n_3|^1|n_4|^0|n_5|^0\\
%&+&|n_1|^1|n_2|^1|n_3|^0|n_4|^0|n_5|^0+|n_1|^0|n_2|^2|n_3|^0|n_4|^0|n_5|^0+|n_1|^0|n_2|^1|n_3|^1|n_4|^0|n_5|^0\\
%&+&|n_1|^1|n_2|^0|n_3|^1|n_4|^0|n_5|^0+|n_1|^0|n_2|^1|n_3|^1|n_4|^0|n_5|^0+|n_1|^0|n_2|^0|n_3|^2|n_4|^0|n_5|^0\\
%&+&|n_1|^1|n_2|^0|n_3|^0|n_4|^1|n_5|^0+|n_1|^0|n_2|^1|n_3|^0|n_4|^1|n_5|^0+|n_1|^0|n_2|^0|n_3|^1|n_4|^1|n_5|^0\\
%&+&|n_1|^1|n_2|^0|n_3|^0|n_4|^0|n_5|^1+|n_1|^0|n_2|^1|n_3|^0|n_4|^0|n_5|^1+|n_1|^0|n_2|^0|n_3|^1|n_4|^0|n_5|^1\\
&=&\sum_{\alpha=(\alpha_i)_{1\leq i\leq2\sigma(0,0,1)}\in\mathfrak R^{(2,(0,0,1))}}\prod_i|(n^{(2)})_i|^{\alpha_i}.
\end{eqnarray*}
\begin{rema}
Once one has handled Case 1.1, Cases 1.2 and 1.3 can simply be obtained by a permutation of group $\mathbb Z_5$.
\end{rema}

Case 2.1. $\gamma=(1,1,0)$:
   $
  2\sigma(1,1,0)=2(\sigma(1)+\sigma(1)+\sigma(0))
  =2(3/2+3/2+1/2)
  =7,
  $ $n^{(2)}=(n_1,n_2,n_3;n_4,n_5,n_6;n_7)\in\mathfrak N^{(1,1)}\times\mathfrak N^{(1,1)}\times\mathfrak N^{(1,0)}=(\mathbb Z^\nu\times\mathbb Z^\nu\times\mathbb Z^\nu)\times (\mathbb Z^\nu\times\mathbb Z^\nu\times\mathbb Z^\nu)\times\mathbb Z^\nu$, and we have
\begin{eqnarray*}
&&\mathfrak B(n^{(2)})\\
&=&|\mu(n^{(2)})|\mathfrak B(n_1,n_2,n_3)\mathfrak B(n_4,n_5,n_6)\mathfrak B(n_7)\\
&\leq&|\mu(n_1,n_2,n_3;n_4,n_5,n_6;n_7)|\cdot|\mu(n_1,n_2,n_3)|\cdot|\mu(n_4,n_5,n_6)|\cdot1\\
&\leq&(|n_1|+|n_2|+|n_3|+|n_4|+|n_5|+|n_6|+|n_7|)\cdot(|n_1|+|n_2|+|n_3|)\cdot(|n_4|+|n_5|+|n_6|)\\
&=&\sum_{\alpha=(\alpha_i)_{1\leq i\leq2\sigma(1,1,0)}\in\mathfrak R^{(2,(1,1,0))}}\prod_i|(n^{(2)})_i|^{\alpha_i}.
\end{eqnarray*}

Case 2.2. $\gamma=(1,0,1)$:
  $
  2\sigma(1,0,1)
 % =2(\sigma(1)+\sigma(0)+\sigma(1))
%  =2(3/2+1/2+3/2)
  =7,
  $
%\begin{align*}
%&\mathfrak R^{(2,(1,0,1))}\\
%=&\{(2,0,0;0;1,0,0),(2,0,0;0;0,1,0),(2,0,0;0;0,0,1),(1,1,0;0;1,0,0),(1,1,0;0;0,1,0),\\&(1,1,0;0;0,0,1),(1,0,1;0;1,0,0),(1,0,1;0;0,1,0),(1,0,1;0;0,0,1)\}\\
%\cup~&\{(1,1,0;0;1,0,0),(1,1,0;0;0,1,0),(1,1,0;0;0,0,1),(0,2,0;0;1,0,0),(0,2,0;0;0,1,0),\\&(0,2,0;0;0,0,1),(0,1,1;0;1,0,0),(0,1,1;0;0,1,0),(0,1,1;0;0,0,1)\}\\
%\cup~&\{(1,0,1;0;1,0,0),(1,0,1;0;0,1,0),(1,0,1;0;0,0,1),(0,1,1;0;1,0,0),(0,1,1;0;0,1,0),\\&(0,1,1;0;0,0,1),(0,0,2;0;1,0,0),(0,0,2;0;0,1,0),(0,0,2;0;0,0,1)\}\\
%\cup~&\{(1,0,0;1;1,0,0),(1,0,0;1;0,1,0),(1,0,0;1;0,0,1),(0,1,0;1;1,0,0),(0,1,0;1;0,1,0),\\&(0,1,0;1;0,0,1),(0,0,1;1;1,0,0),(0,0,1;1;0,1,0),(0,0,1;1;0,0,1)\}\\
%\cup~&\{(1,0,0;0;2,0,0),(1,0,0;0;1,1,0),(1,0,0;0;1,0,1),(0,1,0;0;2,0,0),(0,1,0;0;1,1,0),\\&(0,1,0;0;1,0,1),(0,0,1;0;2,0,0),(0,0,1;0;1,1,0),(0,0,1;0;1,0,1)\}\\
%\cup~&\{(1,0,0;0;1,1,0),(1,0,0;0;0,2,0),(1,0,0;0;0,1,1),(0,1,0;0;1,1,0),(0,1,0;0;0,2,0),\\&(0,1,0;0;0,1,1),(0,0,1;0;1,1,0),(0,0,1;0;0,2,0),(0,0,1;0;0,1,1)\}\\
%\cup~&\{(1,0,0;0;1,0,1),(1,0,0;0;0,1,1),(1,0,0;0;0,0,2),(0,1,0;0;1,0,1),(0,1,0;0;0,1,1),\\&(0,1,0;0;0,0,2),(0,0,1;0;1,0,1),(0,0,1;0;0,1,1),(0,0,1;0;0,0,2)\},
%\end{align*}
$n^{(2)}=(n_1,n_2,n_3;n_4;n_5,n_6,n_7)\in\mathfrak N^{(1,1)}\times\mathfrak N^{(1,0)}\times\mathfrak N^{(1,1)}=(\mathbb Z^\nu\times\mathbb Z^\nu\times\mathbb Z^\nu)\times\mathbb Z^\nu\times(\mathbb Z^\nu\times\mathbb Z^\nu\times\mathbb Z^\nu)$, and
\begin{eqnarray*}
&&\mathfrak B(n^{(2)})\\
&=&|\mu(n^{(2)})|\mathfrak B(n_1,n_2,n_3)\mathfrak B(n_4)\mathfrak B(n_5,n_6,n_7)\\
&\leq&|\mu(n_1,n_2,n_3;n_4,n_5,n_6;n_7)|\cdot|\mu(n_1,n_2,n_3)|\cdot1\cdot|\mu(n_5,n_6,n_7)|\\
&\leq&(|n_1|+|n_2|+|n_3|+|n_4|+|n_5|+|n_6|+|n_7|)\cdot(|n_1|+|n_2|+|n_3|)\cdot(|n_5|+|n_6|+|n_7|)\\
&=&\sum_{\alpha=(\alpha_i)_{1\leq i\leq2\sigma(0,1,1)}\in\mathfrak R^{(2,(0,1,1))}}\prod_i|(n^{(2)})_i|^{\alpha_i}.
\end{eqnarray*}

Case 2.3. $\gamma=(0,1,1)$:
  $
  2\sigma(0,1,1)
 % =2(\sigma(0)+\sigma(1)+\sigma(1))
%  =2(1/2+3/2+3/2)
  =7,
  $
%\begin{align*}
%&\mathfrak R^{(2,(0,1,1))}\\
%=&\{(1;1,0,0;1,0,0),(1;1,0,0;0,1,0),(1;1,0,0;0,0,1),(1;0,1,0;1,0,0),(1;0,1,0;0,1,0),\\&(1;0,1,0;0,0,1),(1;0,0,1;1,0,0),(1;0,0,1;0,1,0),(1;0,0,1;0,0,1)\}\\
%\cup~&\{(0;2,0,0;1,0,0),(0;2,0,0;0,1,0),(0;2,0,0;0,0,1),(0;1,1,0;1,0,0),(0;1,1,0;0,1,0),\\&(0;1,1,0;0,0,1),(0;1,0,1;1,0,0),(0;1,0,1;0,1,0),(0;1,0,1;0,0,1)\}\\
%\cup~&\{(0;1,1,0;1,0,0),(0;1,1,0;0,1,0),(0;1,1,0;0,0,1),(0;0,2,0;1,0,0),(0;0,2,0;0,1,0),\\&(0;0,2,0;0,0,1),(0;0,1,1;1,0,0),(0;0,1,1;0,1,0),(0;0,1,1;0,0,1)\}\\
%\cup~&\{(0;1,0,1;1,0,0),(0;1,0,1;0,1,0),(0;1,0,1;0,0,1),(0;0,1,1;1,0,0),(0;0,1,1;0,1,0),\\&(0;0,1,1;0,0,1),(0;0,0,2;1,0,0),(0;0,0,2;0,1,0),(0;0,0,2;0,0,1)\}\\
%\cup~&\{(0;1,0,0;2,0,0),(0;1,0,0;1,1,0),(0;1,0,0;1,0,1),(0;0,1,0;2,0,0),(0;0,1,0;1,1,0),\\&(0;0,1,0;1,0,1),(0;0,0,1;2,0,0),(0;0,0,1;1,1,0),(0;0,0,1;1,0,1)\}\\
%\cup~&\{(0;1,0,0;1,1,0),(0;1,0,0;0,2,0),(0;1,0,0;0,1,1),(0;0,1,0;1,1,0),(0;0,1,0;0,2,0),\\&(0;0,1,0;0,1,1),(0;0,0,1;1,1,0),(0;0,0,1;0,2,0),(0;0,0,1;0,1,1)\}\\
%\cup~&\{(0;1,0,0;1,0,1),(0;1,0,0;0,1,1),(0;1,0,0;0,0,2),(0;0,1,0;1,0,1),(0;0,1,0;0,1,1),\\&(0;0,1,0;0,0,2),(0;0,0,1;1,0,1),(0;0,0,1;0,1,1),(0;0,0,1;0,0,2)\},
%\end{align*}
$n^{(2)}=(n_1;n_2,n_3,n_4;n_5,n_6,n_7)\in\mathfrak N^{(1,1)}\times\mathfrak N^{(1,0)}\times\mathfrak N^{(1,1)}=(\mathbb Z^\nu\times\mathbb Z^\nu\times\mathbb Z^\nu)\times\mathbb Z^\nu\times(\mathbb Z^\nu\times\mathbb Z^\nu\times\mathbb Z^\nu)$, and
\begin{eqnarray*}
&&\mathfrak B(n^{(2)})\\
&=&|\mu(n^{(2)})|\mathfrak B(n_1)\mathfrak B(n_2,n_3,n_4)\mathfrak B(n_5,n_6,n_7)\\
&\leq&|\mu(n_1,n_2,n_3;n_4,n_5,n_6;n_7)|\cdot1\cdot|\mu(n_2,n_3,n_4)|\cdot|\mu(n_5,n_6,n_7)|\\
&\leq&(|n_1|+|n_2|+|n_3|+|n_4|+|n_5|+|n_6|+|n_7|)\cdot(|n_2|+|n_3|+|n_4|)\cdot(|n_5|+|n_6|+|n_7|)\\
&=&\sum_{\alpha=(\alpha_i)_{1\leq i\leq2\sigma(0,1,1)}\in\mathfrak R^{(2,(0,1,1))}}\prod_i|(n^{(2)})_i|^{\alpha_i}.
\end{eqnarray*}
\begin{rema}
Once one has handled Case 2.1, Cases 2.2 and 2.3 can simply be obtained by a permutation of group $\mathbb Z_7$.
\end{rema}

Case 3. $\gamma=(1,1,1)$:
  $
  2\sigma(1,1,1)=2(\sigma(1)+\sigma(1)+\sigma(1))
  =2(3/2+3/2+3/2)
  =9$,
$n^{(2)}=(n_1,n_2,n_3;n_4,n_5,n_6;n_7,n_8,n_9)\in\mathfrak N^{(1,1)}\times\mathfrak N^{(1,1)}\times\mathfrak N^{(1,1)}=(\mathbb Z^\nu\times\mathbb Z^\nu\times\mathbb Z^\nu)\times(\mathbb Z^\nu\times\mathbb Z^\nu\times\mathbb Z^\nu)\times(\mathbb Z^\nu\times\mathbb Z^\nu\times\mathbb Z^\nu),$ and one can derive that
\begin{eqnarray*}
&&\mathfrak B(n^{(2)})\\
&=&|\mu(n^{(2)})|\mathfrak B(n_1,n_2,n_3)\mathfrak B(n_4,n_5,n_6)\mathfrak B(n_7,n_8,n_9)\\
&\leq&|\mu(n_1,n_2,n_3;n_4,n_5,n_6;n_7,n_8,n_9)|\cdot\mu(n_1,n_2,n_3)|\cdot|\mu(n_4,n_5,n_6)|\cdot|\mu(n_7,n_8,n_9)|\\
&\leq&\sum_{j=1}^{9}|n_j|\cdot(|n_1|+|n_2|+|n_3|)\cdot(|n_4|+|n_5|+|n_6|)\cdot(|n_7|+|n_8|+|n_9|)\\
%&=&\sum_{j=1}^{9}\sum_{\alpha=(\alpha_i)_{1\leq i\leq2\sigma(1,1,1)}\in\diamondsuit_j}\prod_{i}|(n^{(2)})_i|^{\alpha_i}\\
%&=&\sum_{\alpha=(\alpha_i)_{1\leq i\leq2\sigma(\gamma)}\in\cup_{j=1}^9\diamondsuit_j}\prod_i|(n^{(2)})_i|^{\alpha_i}\\
&=&\sum_{\alpha=(\alpha_i)_{1\leq i\leq2\sigma(1,1,1)}\in\mathfrak R^{(2,(1,1,1))}}\prod_i|(n^{(2)})_i|^{\alpha_i}.
%=&\sum_{j=1}^{9}\P_j,
%=&\sum_{\alpha=(\alpha_i)_{1\leq i\leq2\sigma(1,1,1)}\in\mathfrak R^{(2,(1,1,1))}}\prod_j|(n^{(2)})_j|^{\alpha_j}.\\
\end{eqnarray*}
This shows that \eqref{bbe} holds true for $\gamma \in \mathfrak T^{(1)} \times \mathfrak T^{(1)}\times\mathfrak T^{(1)}$.

Let $k\geq2$. Suppose now that \eqref{bbe} holds for all $1\leq k^\prime\leq k-1$. Let us consider the case of $k$. Let $\gamma = (\gamma_1, \gamma_2, \gamma_3) \in \mathfrak T^{(k-1)} \times \mathfrak T^{(k-1)}\mathfrak T^{(k-1)}, n^{(k)}=(\clubsuit,\spadesuit,\heartsuit) \in \mathfrak N^{(k-1,\gamma_1)}\times\mathfrak N^{(k-1,\gamma_2)}\times\mathfrak N^{(k-1,\gamma_3)} \subset \prod_{j=1}^{2\sigma(\gamma_1)} \mathbb Z^\nu \times \prod_{j=1}^{2\sigma(\gamma_1)} \mathbb Z^\nu \times \prod_{j=1}^{2\sigma(\gamma_3)} \mathbb Z^\nu \simeq \prod_{j=1}^{2\sigma(\gamma)} \mathbb Z^\nu$,
where
$\clubsuit=(n_1,\cdots,n_{2\sigma(\gamma_1)})\in\mathfrak N^{(k-1,\gamma_1)},\spadesuit=(n_{2\sigma(\gamma_1)+1},,\cdots,n_{2\sigma(\gamma_1)+2\sigma(\gamma_2)})\in\mathfrak N^{(k-1,\gamma_2)}$
and
$\heartsuit=(n_{2\sigma(\gamma_1)+2\sigma(\gamma_2)+1},\cdots,n_{2\sigma(\gamma_1)+2\sigma(\gamma_2)
+2\sigma(\gamma_3)}=n_{2\sigma(\gamma)})\in\mathfrak N^{(k-1,\gamma_3)}$, that is,
\begin{align*}
(n^{(k)})_i=n_i=
\begin{cases}
\clubsuit_i,&1\leq i\leq2\sigma(\gamma_1);\\
\spadesuit_{i-2\sigma(\gamma_1)},&2\sigma(\gamma_1)+1\leq i\leq2\sigma(\gamma_1)+2\sigma(\gamma_2);\\
\heartsuit_{i-2\sigma(\gamma_1)-2\sigma(\gamma_2)},&2\sigma(\gamma_1)+2\sigma(\gamma_2)+1\leq i\leq2\sigma(\gamma_1)+2\sigma(\gamma_2)+2\sigma(\gamma_3).
\end{cases}
\end{align*}
Then,
\begin{eqnarray*}
\mathfrak B(n^{(k)})&=&|\mu(n^{(k)})|\mathfrak B(\clubsuit)\mathfrak B(\spadesuit)\mathfrak B(\heartsuit)\\
&\leq&\sum_{i=1}^{2\sigma(\gamma)}|(n^{(k)})_i|\times\sum_{\alpha^{(1)}=(\alpha^{(1)}_i)_{1\leq i\leq2\sigma(\gamma_1)}\in\mathfrak R^{(k-1,\gamma_1)}}\prod_i|\clubsuit_i|^{\alpha^{(1)}_i}\\
&&\hspace{22.6mm}\times\sum_{\alpha^{(2)}=(\alpha^{(2)}_i)_{1\leq i\leq2\sigma(\gamma_2)}\in\mathfrak R^{(k-1,\gamma_2)}}\prod_i|\spadesuit_i|^{\alpha^{(2)}_i}\\
&&\hspace{22.6mm}\times\sum_{\alpha^{(3)}=(\alpha^{(3)}_i)_{1\leq i\leq2\sigma(\gamma_3)}\in\mathfrak R^{(k-1,\gamma_3)}}\prod_i|\heartsuit_i|^{\alpha^{(3)}_i}\\
&=&\sum_{i=1}^{2\sigma(\gamma)}|(n^{(k)})_i|\times\sum_{\alpha=(\alpha_i)_{1\leq i\leq 2\sigma(\gamma)}\in\mathfrak R^{(k-1,\gamma_1)}\times\mathfrak R^{(k-1,\gamma_2)}\times\mathfrak R^{(k-1,\gamma_3)}}\prod_i|(n^{(k)})_i|^{\alpha_i}\\
&=&\sum_{\alpha=(\alpha_i)_{1\leq i\leq 2\sigma(\gamma)}\in\mathfrak R^{(k,\gamma)}}\prod_i|(n^{(k)})_i|^{\alpha_i}.
\end{eqnarray*}
This completes the proof of Lemma \ref{iii}.
\end{proof}

\begin{property}\label{pro}
%$\mathfrak R^{(k,\gamma)}\subset\{\alpha\in\prod_{j=1}^{2\sigma(\gamma)}\mathbb Z_+^\nu: |\alpha|=\ell(\gamma)\}\triangleq\mathfrak L^{(k,\gamma)}~\left(\forall\alpha\in\mathfrak R^{(k,\gamma)}\Rightarrow|\alpha|=\ell(\gamma)\right)$.
$\alpha \in \mathfrak R^{(k,\gamma)} \Rightarrow |\alpha| = \ell(\gamma)$.
\end{property}

\begin{proof}
Let $\alpha\in\mathfrak R^{(k,\gamma)}$. If $\gamma=0\in\mathfrak T^{(k)}$, $k\geq1$, then $\ell(0)=0$ and $\mathfrak R^{(k,0)}=\{0\in\mathbb Z\}$, so that $|\alpha|=0$. Hence $|\alpha|=\ell(\gamma)$ for $\gamma = 0 \in \mathfrak T^{(k)}$.

If $\gamma = 1 \in \mathfrak T^{(1)}$, then $\ell(\gamma)=1$ and $\mathfrak R^{(1,1)}=\{(1,0,0),(0,1,0),(0,0,1)\}$, and hence $|\alpha|=1$. Thus, $|\alpha| = \ell(\gamma)$ for $\gamma = 1 \in \mathfrak T^{(1)}$.

Let $k\geq2$. Assume that $|\alpha| = \ell(\gamma)$ holds for all $1\leq k^\prime\leq k-1$. For $k$, $\gamma=(\gamma_1,\gamma_2,\gamma_3)\in\mathfrak T^{(k-1)} \times \mathfrak T^{(k-1)} \times \mathfrak T^{(k-1)},\alpha=(\alpha^{(1)},\alpha^{(2)},\alpha^{(3)})+\beta$, where $\alpha^{(j)}\in\mathfrak R^{(k-1,\gamma_j)},j=1,2,3,\beta\in e^{(k,\gamma)}$, it is easy to see that
\begin{eqnarray*}
|\alpha|&=&|\alpha^{(1)}|+|\alpha^{(2)}|+|\alpha^{(3)}|+|\beta|\\
&=&\ell(\gamma_1)+\ell(\gamma_2)+\ell(\gamma_3)+1\\
&=&\ell(\gamma).
\end{eqnarray*}
This completes the proof of Property \ref{pro}.
\end{proof}

\begin{lemm}[Estimates of $|c_k(t,n)|$ in terms of the new variable $\alpha_j$]\label{mth2}
For all $k\geq1$, we have
\begin{align*}
|c_{k}(t,n)|\leq\mathcal A^{\frac{1}{2}}\sum_{\gamma\in\mathfrak T^{(k)}}\frac{(\mathcal A|\omega|t)^{\ell(\gamma)}}{\mathfrak D(\gamma)}\sum_{
\substack{\alpha=(\alpha_i)_{1\leq i\leq2\sigma(\gamma)}\in\mathfrak R^{(k,\gamma)}}
}\sum_{\substack{n^{(k)}\in\mathfrak N^{(k,\gamma)}\\\mu(n^{(k)})=n}}\prod_i|(n^{(k)})_i|^{\alpha_i}e^{-\kappa|(n^{(k)})_i|}.
\end{align*}
\end{lemm}

\begin{proof}
Note that
\begin{eqnarray*}
&&|c_{k}(t,n)|\\
&\stackrel{(\text{\scriptsize Lemma \ref{mth1}})}{\leq}&\mathcal A^{\frac{1}{2}}\sum_{\gamma\in\mathfrak T^{(k)}}\frac{(\mathcal A|\omega|t)^{\ell(\gamma)}}
{\mathfrak D(\gamma)}\sum_{n^{(k)}\in\mathfrak N^{(k,\gamma)}:~\mu(n^{(k)})=n}\mathfrak B(n^{(k)})e^{-\kappa|n^{(k)}|}
~\\
\nonumber&\stackrel{(\text{\scriptsize Lemma \ref{iii}})}{\leq}&\mathcal A^{\frac{1}{2}}\sum_{\gamma\in\mathfrak T^{(k)}}\frac{(\mathcal A|\omega|t)^{\ell(\gamma)}}{\mathfrak D(\gamma)}
\sum_{\substack{n^{(k)}\in\mathfrak N^{(k,\gamma)}\\\mu(n^{(k)})=n}}\sum_{
\substack{\alpha=(\alpha_i)_{1\leq i\leq2\sigma(\gamma)}\in\mathfrak R^{(k,\gamma)}}
}\prod_i|(n^{(k)})_i|^{\alpha_i}\prod_ie^{-\kappa|(n^{(k)})_i|}~\\
\nonumber&=&\mathcal A^{\frac{1}{2}}\sum_{\gamma\in\mathfrak T^{(k)}}\frac{(\mathcal A|\omega|t)^{\ell(\gamma)}}{\mathfrak D(\gamma)}\sum_{
\substack{\alpha=(\alpha_i)_{1\leq i\leq2\sigma(\gamma)}\in\mathfrak R^{(k,\gamma)}}
}\sum_{\substack{n^{(k)}\in\mathfrak N^{(k,\gamma)}:~\mu(n^{(k)})=n}}\prod_i|(n^{(k)})_i|^{\alpha_i}e^{-\kappa|(n^{(k)})_i|}.
\end{eqnarray*}
This completes the proof of Lemma \ref{mth2}.
\end{proof}

\begin{lemm}[Continuation of estimates of $|c_k(t,n)|$ in terms of the new variable $\alpha_j$]\label{again}
For all $k\geq1$, we have
\[
|c_{k}(t,n)|\leq(6\kappa^{-1})^\nu\mathcal A^{\frac{1}{2}}e^{-\frac{\kappa}{2}|n|}\sum_{\gamma\in\mathfrak T^{(k)}}\frac{\left(\mathcal A|\omega|(6\kappa^{-1})^{2\nu+1}t\right)^{\ell(\gamma)}}{\mathfrak D(\gamma)}\sum_{
\alpha=(\alpha_i)_{1\leq i\leq2\sigma(\gamma)}\in\mathfrak R^{(k,\gamma)}
}\prod_j\alpha_j!.
\]
\end{lemm}

\begin{proof}
We estimate $|c_{k}(t,n)|$ as follows:
\begin{eqnarray*}
\nonumber&&|c_{k}(t,n)|\\
&\stackrel{(\text{\scriptsize Lemma \ref{mth2}})}{\leq}&\mathcal A^{\frac{1}{2}}\sum_{\gamma\in\mathfrak T^{(k)}}\frac{(\mathcal A|\omega|t)^{\ell(\gamma)}}{\mathfrak D(\gamma)}\sum_{
\substack{(\alpha_i)_{1\leq i\leq2\sigma(\gamma)}\in\mathfrak R^{(k,\gamma)}}
}\sum_{\substack{n^{(k)}\in\mathfrak N^{(k,\gamma)}\\\mu(n^{(k)})=n}}\prod_i|(n^{(k)})_i|^{\alpha_i}e^{-\kappa|(n^{(k)})_i|}~~\\
\nonumber&\stackrel{(\text{\scriptsize Lemma \ref{4}})}{\leq}&\mathcal A^{\frac{1}{2}}e^{-\frac{\kappa}{2}|n|}\sum_{\gamma\in\mathfrak T^{(k)}}\frac{(\mathcal A|\omega|t)^{\ell(\gamma)}}{\mathfrak D(\gamma)}\sum_{
\alpha=(\alpha_i)_{1\leq i\leq2\sigma(\gamma)}\in\mathfrak R^{(k,\gamma)}
}(6\kappa^{-1})^{|\alpha|+2\sigma(\gamma)\nu}\prod_j\alpha_j!~\\
\nonumber&\stackrel{(\text{\scriptsize Property \ref{pro}})}{\leq}&\mathcal A^{\frac{1}{2}}e^{-\frac{\kappa}{2}|n|}\sum_{\gamma\in\mathfrak T^{(k)}}\frac{(\mathcal A|\omega|t)^{\ell(\gamma)}}{\mathfrak D(\gamma)}\sum_{
\alpha=(\alpha_i)_{1\leq i\leq2\sigma(\gamma)}\in\mathfrak R^{(k,\gamma)}
}(6\kappa^{-1})^{\ell(\gamma)+2\sigma(\gamma)\nu}\prod_j\alpha_j!~\\
\nonumber&=&(6\kappa^{-1})^\nu\mathcal A^{\frac{1}{2}}e^{-\frac{\kappa}{2}|n|}\sum_{\gamma\in\mathfrak T^{(k)}}\frac{\left(\mathcal A|\omega|(6\kappa^{-1})^{2\nu+1}t\right)^{\ell(\gamma)}}{\mathfrak D(\gamma)}\sum_{
\alpha=(\alpha_i)_{1\leq i\leq2\sigma(\gamma)}\in\mathfrak R^{(k,\gamma)}
}\prod_j\alpha_j!.
\end{eqnarray*}
This completes the proof of Lemma \ref{again}.
\end{proof}

\subsection{Estimates of $\sum_{\gamma\in\mathfrak T^{(k)}}\frac{{\bf t}^{\ell(\gamma)}}{\mathfrak D(\gamma)}\sum_{
\alpha=(\alpha_i)_{1\leq i\leq2\sigma(\gamma)}\in\mathfrak R^{(k,\gamma)}
}\prod_j\alpha_j!$}

\begin{lemm}\label{lmmg}
For $0\leq{\bf t}\leq\frac{1}{16}$, we have
\begin{align}\label{101}
\sum_{\gamma\in\mathfrak T^{(k)}}\frac{{\bf t}^{\ell(\gamma)}}{\mathfrak D(\gamma)}\sum_{
\alpha=(\alpha_i)_{1\leq i\leq2\sigma(\gamma)}\in\mathfrak R^{(k,\gamma)}
}\prod_i\alpha_i!\leq2, \quad\forall k\geq1.
\end{align}
\end{lemm}

\begin{proof}
For $k=1$,
\begin{eqnarray*}
&&\sum_{\gamma\in\mathfrak T^{(1)}}\frac{{\bf t}^{\ell(\gamma)}}{\mathfrak D(\gamma)}\sum_{
\alpha=(\alpha_i)_{1\leq i\leq2\sigma(\gamma)}\in\mathfrak R^{(1,\gamma)}
}\prod_i\alpha_i!\\
&=&\frac{{\bf t}^{\ell(0)}}{\mathfrak D(0)}\sum_{\alpha=(\alpha_i)_{1\leq i\leq2\sigma(0)}\in\mathfrak R^{(1,0)}}\prod_i\alpha_i!+\frac{{\bf t}^{\ell(1)}}{\mathfrak D(1)}\sum_{
\alpha=(\alpha_i)_{1\leq i\leq2\sigma(1)}\in\mathfrak R^{(1,1)}
}\prod_i\alpha_i!\\
&=&1+{\bf t}(1+1+1)\\
&=&1+3{\bf t}\\
&\leq&2
\end{eqnarray*}
provided that $0\leq {\bf t}\leq1/16<1/3$.

Let $k\geq2$. Assume that \eqref{101} is true for all $1\leq k^\prime\leq k-1$. For $k$, $\gamma=0\in\mathfrak T^{(k)}$, we can use the same argument as above. For $\gamma=(\gamma_1,\gamma_2,\gamma_2) \in \mathfrak T^{(k-1)} \times \mathfrak T^{(k-1)} \times \mathfrak T^{(k-1)}$, $\alpha=(\alpha^{(1)},\alpha^{(2)},\alpha^{(3)})+\beta$, where $\alpha^{(j)}\in\mathfrak R^{(k-1,\gamma_j)},\beta\in e^{(k,\gamma)}$, we can proceed as follows:
\begin{eqnarray*}
&&\sum_{\gamma\in\mathfrak T^{(k)}\backslash\{0\}}\frac{{\bf t}^{\ell(\gamma)}}{\mathfrak D(\gamma)}\sum_{
\alpha=(\alpha_i)_{1\leq i\leq2\sigma(\gamma)}\in\mathfrak R^{(k,\gamma)}
}\prod_i\alpha_i!\\
&=&\sum_{\substack{\gamma_j\in\mathfrak T^{(k-1)}\\j=1,2,3}}\frac{{\bf t}}{\sum_{j=1}^3\ell(\gamma_j)+1}\prod_{j=1}^3\frac{{\bf t}^{\ell(\gamma_j)}}{\mathfrak D(\gamma_j)}\\
&\times&\Big\{\sum_{i_0=1}^{\sigma(\gamma_1)}\prod_{i=1}^{\sigma(\gamma_1)}
\left((\alpha^{(1)})_i+\delta_{i,i_0}\right)!\prod_{i=1}^{\sigma(\gamma_2)}(\alpha^{(2)})_i!
\prod_{i=1}^{\sigma(\gamma_3)}(\alpha^{(3)})_i!\\
&+&\sum_{i_0=1}^{\sigma(\gamma_2)}\prod_{i=1}^{\sigma(\gamma_2)}
(\alpha^{(1)})_i!\prod_{i=1}^{\sigma(\gamma_2)}\left((\alpha^{(2)})_i+\delta_{i,i_0}\right)!
\prod_{i=1}^{\sigma(\gamma_3)}(\alpha^{(3)})_i!\\
&+&\sum_{i_0=1}^{\sigma(\gamma_3)}\prod_{i=1}^{\sigma(\gamma_3)}
(\alpha^{(1)})_i!\prod_{i=1}^{\sigma(\gamma_2)}(\alpha^{(2)})_i!
\prod_{i=1}^{\sigma(\gamma_3)}\left(\alpha^{(3)})_i+\delta_{i,i_0}\right)!\Big\}\\
&=&\sum_{\substack{\gamma_j\in\mathfrak T^{(k-1)}\\j=1,2,3}}\frac{{\bf t}}{\sum_{j=1}^3\ell(\gamma_j)+1}\prod_{j=1}^3\frac{{\bf t}^{\ell(\gamma_j)}}{\mathfrak D(\gamma_j)}\sum_{j_0=1}^{3}\sum_{i_0=1}^{\sigma(\gamma_{j_0})}\prod_{j=1}^3
\prod_{j=1}^{\sigma(\gamma_j)
}\left((\alpha^{(j)})_i+\delta_{i,i_0}\delta_{j,j_0}\right)!\\
&=&\sum_{\substack{\gamma_j\in\mathfrak T^{(k-1)}\\j=1,2,3}}\frac{{\bf t}}{\sum_{j=1}^3\ell(\gamma_j)+1}\prod_{j=1}^3\frac{{\bf t}^{\ell(\gamma_j)}}{\mathfrak D(\gamma_j)}\sum_{j_0=1}^{3}\sum_{i_0=1}^{\sigma(\gamma_{j_0})}
\left((\alpha^{(j_0)})_{i_0}+1\right)\prod_{j=1}^{3}
\prod_{j=1}^{\sigma(\gamma_j)
}\left((\alpha^{(j)})_i\right)!\\
&=&\sum_{\substack{\gamma_j\in\mathfrak T^{(k-1)}\\j=1,2,3}}\frac{{\bf t}}{\sum_{j=1}^3\ell(\gamma_j)+1}\prod_{j=1}^3\frac{{\bf t}^{\ell(\gamma_j)}}{\mathfrak D(\gamma_j)}\sum_{j_0=1}^{3}\left(\ell(\gamma_{j_0})+\sigma(\gamma_{j_0})\right)
\prod_{j=1}^{3}
\prod_{j=1}^{\sigma(\gamma_j)
}\left((\alpha^{(j)})_i\right)!\\
&=&\sum_{\substack{\gamma_j\in\mathfrak T^{(k-1)}\\j=1,2,3}}\frac{{\bf t}}{\sum_{j=1}^3\ell(\gamma_j)+1}\prod_{j=1}^3\frac{{\bf t}^{\ell(\gamma_j)}}{\mathfrak D(\gamma_j)}\sum_{j_0=1}^{3}\left(2\ell(\gamma_{j_0})+\frac{1}{2}\right)\prod_{j=1}^3
\prod_{j=1}^{\sigma(\gamma_j)
}\left((\alpha^{(j)})_i\right)!\\
&=&{\bf t}\sum_{\substack{\gamma_j\in\mathfrak T^{(k-1)}\\j=1,2,3}}\frac{\sum_{j_0=1}^3\left(2\ell(\gamma_{j_0})+\frac{1}{2}\right)}
{\sum_{j=1}^3\ell(\gamma_j)+1}\sum_{\substack{\alpha^{(j)}\in\mathfrak R^{(k-1,\gamma_j)}\\j=1,2,3}}\prod_{j=1}^3\frac{{\bf t}^{\ell(\gamma_j)}}{\mathfrak D(\gamma_j)}\prod_{i=1}^{\sigma(\gamma_j)}
(\alpha^{(j)})_i!\\
&=&{\bf t}\sum_{\substack{\gamma_j\in\mathfrak T^{(k-1)}\\j=1,2,3}}\frac{2\left(\sum_{j_0=1}^3\ell(\gamma_{j_0})+1\right)-\frac{1}{2}}{\sum_{j=1}^3\ell(\gamma_j)+1}\sum_{\substack{\alpha^{(j)}\in\mathfrak R^{(k-1,\gamma_j)}\\j=1,2,3}}\prod_{j=1}^3\frac{{\bf t}^{\ell(\gamma_j)}}{\mathfrak D(\gamma_j)}\prod_{i=1}^{\sigma(\gamma_j)}
(\alpha^{(j)})_i!\\
&\leq&2{\bf t}\sum_{\substack{\gamma_j\in\mathfrak T^{(k-1)}\\j=1,2,3}}\sum_{\substack{\alpha^{(j)}\in\mathfrak R^{(k-1,\gamma_j)}\\j=1,2,3}}\prod_{j=1}^3\frac{{\bf t}^{\ell(\gamma_j)}}{\mathfrak D(\gamma_j)}\prod_{i=1}^{\sigma(\gamma_j)}
(\alpha^{(j)})_i!\\
&=&2{\bf t}\prod_{j=1}^3\sum_{\gamma_j\in\mathfrak T^{(k-1)}}\sum_{\alpha^{(j)}\in\mathfrak R^{(k-1,\gamma_j)}}\frac{{\bf t}^{\ell(\gamma_j)}}{\mathfrak D(\gamma_j)}\prod_{i=1}^{\sigma(\gamma_j)}(\alpha^{(j)})_i!\\
&\leq&2{\bf t}\times2\times2\times2\\
&=&16{\bf t}.
\end{eqnarray*}
Hence
\begin{eqnarray*}
&&\sum_{\gamma\in\mathfrak T^{(k)}}\frac{{\bf t}^{\ell(\gamma)}}{\mathfrak D(\gamma)}\sum_{
\alpha=(\alpha_i)_{1\leq i\leq2\sigma(\gamma)}\in\mathfrak R^{(k,\gamma)}
}\prod_j\alpha_j!\\
&=&\sum_{\gamma=0\in\mathfrak T^{(k)}}\frac{{\bf t}^{\ell(\gamma)}}{\mathfrak D(\gamma)}\sum_{
\alpha=(\alpha_i)_{1\leq i\leq2\sigma(\gamma)}\in\mathfrak R^{(k,\gamma)}
}\prod_j\alpha_j!+\sum_{\gamma\in\mathfrak T^{(k)}\backslash\{0\}}\frac{{\bf t}^{\ell(\gamma)}}{\mathfrak D(\gamma)}\sum_{
\alpha=(\alpha_i)_{1\leq i\leq2\sigma(\gamma)}\in\mathfrak R^{(k,\gamma)}
}\prod_j\alpha_j!\\
&\leq&1+16{\bf t}\\
&\leq&2
\end{eqnarray*}
provided that $0\leq{\bf t}\leq1/16$. In particular we see that if $0\leq{\bf t}\leq1/16$, then \eqref{101} holds true. This completes the proof of Lemma \ref{lmmg}.
\end{proof}

Now we are in a position to prove Proposition \ref{thm}.

\begin{proof}[Proof of Proposition \ref{thm}.]
By Lemma \ref{lmmg}, we have
\[\sum_{\gamma\in\mathfrak T^{(k)}}\frac{\left(\mathcal A|\omega|(6\kappa^{-1})^{2\nu+1}t\right)^{\ell(\gamma)}}{\mathfrak D(\gamma)}\sum_{
\alpha=(\alpha_i)_{1\leq i\leq2\sigma(\gamma)}\in\mathfrak R^{(k,\gamma)}
}\prod_j\alpha_j!\leq 2, \quad 0\leq t\leq\frac{\kappa^{2\nu+1}}{16\mathcal A|\omega|6^{2\nu+1}}.\]
Hence
\begin{eqnarray*}
&&|c_k(t,n)|\\
&\stackrel{(\text{\scriptsize Lemma \ref{again}})}{\leq}&(6\kappa^{-1})^\nu\mathcal A^{\frac{1}{2}}e^{-\frac{\kappa}{2}|n|}\sum_{\gamma\in\mathfrak T^{(k)}}\frac{\left(\mathcal A|\omega|(6\kappa^{-1})^{2\nu+1}t\right)^{\ell(\gamma)}}{\mathfrak D(\gamma)}\sum_{
\alpha=(\alpha_i)_{1\leq i\leq2\sigma(\gamma)}\in\mathfrak R^{(k,\gamma)}
}\prod_j\alpha_j!~~\\
&\leq& 2(6\kappa^{-1})^\nu\mathcal A^{\frac{1}{2}}e^{-\frac{\kappa}{2}|n|}\\
&\triangleq&\Box e^{-\frac{\kappa}{2}|n|},
\end{eqnarray*}
where $\Box=2(6\kappa^{-1})^\nu\mathcal A^{\frac{1}{2}}$.
This completes the proof of Proposition \ref{thm}.
\end{proof}

\section{Estimates of $|c_{k}(t,n)-c_{k-1}(t,n)|$ and $|c_{k+\bigstar}(t,n)-c_{k}(t,n)|$}

\subsection{Estimates of $|c_{k}(t,n)-c_{k-1}(t,n)|$}

\begin{lemm}[Estimates of $|c_{k}(t,n)-c_{k-1}(t,n)|$]\label{de}
For all $k\geq1$, we have
\begin{align}\label{pp}
|c_k(t,n)-c_{k-1}(t,n)|
\leq\frac{\Box^{2k+1}(|\omega|t)^k}{k!}\sum_{\substack{(n_1,\cdots,n_{2k},n_{2k+1})\in\mathbb Z^{(2k+1)\nu}:~\sum_{j=1}^{2k+1}n_j=n}}\sum_{\alpha\in\mathbb B^{(k)}}\prod_{j=1}^{2k+1}|n_j|^{\alpha_j}e^{-\frac{\kappa}{2}|n_j|},
\end{align}
where
\begin{align*}
\mathbb B^{(k)}:=
\begin{cases}
\{(1,0,0),(0,1,0),(0,0,1)\},&k=1;\\
\mathbb B^{(k-1)}\times\{0\in\mathbb Z\}\times\{0\in\mathbb Z\}+\mathfrak g^{(k)},&k\geq2,
\end{cases}
\end{align*}
and
\[\mathfrak g^{(k)}:=\left\{\alpha\in\mathbb Z^{2k+1}: \sum_{j}\alpha_j=1,\alpha_j\geq0\right\}.\]
\end{lemm}

\begin{proof}
For the case $k=1$, we note that
\begin{eqnarray*}
&&|c_1(t,n)-c_0(t,n)|\\
&=&\left|-\frac{{\rm i}n\cdot\omega}{3}\int_0^te^{{\rm i}(n\cdot\omega)^3(t-\tau)}\sum_{p,q,r\in\mathbb Z^\nu:~p+q+r=n}c_{0}(\tau,p)c_{0}(\tau,q)c_{0}(\tau,r){\rm d}\tau\right|\\
&\leq&|n||\omega|\int_0^t\sum_{p,q,r\in\mathbb Z^\nu:~p+q+r=n}\Box e^{-\frac{\kappa}{2}|p|}\cdot\Box e^{-\frac{\kappa}{2}|q|}\cdot\Box e^{-\frac{\kappa}{2}|r|}{\rm d}\tau\\
%\leq&\Box^3|\omega|t\sum_{p,q,r\in\mathbb Z^\nu: p+q+r=n}(|p|+|q|+|r|)e^{-\kappa(|p|+|q|+|r|)}\\
&\leq&\Box^3|\omega|t\sum_{m=(m_1,m_2,m_3)\in\mathbb Z^{3\nu}:~\mu(m)=n}(|m_1|+|m_2|+|m_3|)\prod_{j=1}^3e^{-\frac{\kappa}{2}|m_j|}\\
%=&\Box^3|\omega|t\sum_{m=(m_1,m_2,m_3)\in\mathbb Z^{3\nu}: \mu(m)=n}(|m_1|^1|m_2|^0|m_3|^0+|m_1|^0|m_2|^1|m_3|^0+|m_1|^0|m_2|^0|m_3|^1)\prod_{j=1}^3e^{-\kappa|m_j|}\\
&=&\Box^3|\omega|t\sum_{m=(m_1,m_2,m_3)\in\mathbb Z^{3\nu}:~\mu(m)=n}\sum_{\alpha\in\mathbb B^{(1)}}\prod_{j=1}^3|m_j|^{\alpha_j}e^{-\frac{\kappa}{2}|m_j|}.
%\leq&\Box^3|\omega|t\sum_{m=(m_1,m_2,m_3)\in\mathbb Z^{3\nu}: \mu(m)=n}\sum_{\alpha\in\mathbb B^{(1)}}\prod_{j=1}^3|m_j|^{\alpha_j}e^{-\frac{\kappa}{2}|m_j|}.
%\leq&\Box^3|\omega|t\sum_{m=(m_1,m_2,m_3)\in\mathbb Z^{3\nu}: \mu(m)=n}\sum_{\alpha\in\mathbb B^{(1)}}\prod_{j=1}^3|m_j|_1^{\alpha_j}e^{-\frac{\kappa}{2}|m_j|_1}.
\end{eqnarray*}

Let $k\geq2$. Assume that \eqref{pp} is true for $1\leq k^\prime\leq k-1$. For $k$, we have
\begin{eqnarray*}
&&|c_k(t,n)-c_{k-1}(t,n)|\\
%&=&\left|-\frac{{\rm i}n\cdot\omega}{3}\int_0^te^{{\rm i}(n\cdot\omega)^3(t-\tau)}\sum_{\substack{p,q,r\in\mathbb Z^\nu\\ p+q+r=n}}\left(c_{k-1}(\tau,p)c_{k-1}(\tau,q)c_{k-1}(\tau,r)-c_{k-2}(\tau,p)c_{k-2}(\tau,q)c_{k-2}(\tau,r)\right){\rm d}\tau\right|\\
&\leq&\frac{|n||\omega|}{3}\int_0^t\sum_{p,q,r\in\mathbb Z^\nu: p+q+r=n}\left|c_{k-1}(\tau,p)c_{k-1}(\tau,q)c_{k-1}(\tau,r)-c_{k-2}(\tau,p)c_{k-2}(\tau,q)c_{k-2}(\tau,r)\right|{\rm d}\tau\\
&\leq&\frac{|n||\omega|}{3}\int_0^t\sum_{p,q,r\in\mathbb Z^\nu:~ p+q+r=n}|c_{k-1}(\tau,p)-c_{k-2}(\tau,p)||c_{k-1}(\tau,q)||c_{k-1}(\tau,r)|{\rm d}\tau\\
&+&\frac{|n||\omega|}{3}\int_0^t\sum_{p,q,r\in\mathbb Z^\nu:~ p+q+r=n}|c_{k-1}(\tau,q)-c_{k-2}(\tau,q)||c_{k-2}(\tau,p)||c_{k-1}(\tau,r)|{\rm d}\tau\\
&+&\frac{|n||\omega|}{3}\int_0^t\sum_{p,q,r\in\mathbb Z^\nu:~ p+q+r=n}|c_{k-1}(\tau,r)-c_{k-2}(\tau,r)||c_{k-2}(\tau,p)||c_{k-2}(\tau,q)|{\rm d}\tau\\
&\triangleq&\flat^1+\flat^2+\flat^3,
\end{eqnarray*}
where
\begin{eqnarray*}
\flat^1&=&\frac{|n||\omega|}{3}\int_0^t\sum_{p,q,r\in\mathbb Z^\nu:~ p+q+r=n}|c_{k-1}(\tau,p)-c_{k-2}(\tau,p)||c_{k-1}(\tau,q)||c_{k-1}(\tau,r)|{\rm d}\tau,\\
\flat^2&=&\frac{|n||\omega|}{3}\int_0^t\sum_{p,q,r\in\mathbb Z^\nu:~ p+q+r=n}|c_{k-1}(\tau,q)-c_{k-2}(\tau,q)||c_{k-2}(\tau,p)||c_{k-1}(\tau,r)|{\rm d}\tau,\\
\flat^3&=&\frac{|n||\omega|}{3}\int_0^t\sum_{p,q,r\in\mathbb Z^\nu:~ p+q+r=n}|c_{k-1}(\tau,r)-c_{k-2}(\tau,r)||c_{k-2}(\tau,p)||c_{k-2}(\tau,q)|{\rm d}\tau.
\end{eqnarray*}
For $\flat^1$ we have
\begin{eqnarray*}
&&\flat^1\\
&\leq&\frac{|n||\omega|}{3}\int_0^t\sum_{\substack{p,q,r\in\mathbb Z^\nu\\p+q+r=n}}\frac{\Box^{2k-1}(|\omega|\tau)^{k-1}}{(k-1)!}\sum_{\substack{n^\prime\in\mathbb Z^{(2k-1)\nu}\\\sum_{j=1}^{2k-1}n_j^\prime=p}}\sum_{\alpha\in\mathbb B^{(k-1)}}\prod_{j=1}^{2k-1}|n_j^\prime|^{\alpha_j}e^{-\frac{\kappa}{2}|n_j^\prime|}\cdot \Box e^{-\frac{\kappa}{2}|q|}\cdot \Box e^{-\frac{\kappa}{2}|r|}{\rm d}\tau\\
&=&\frac{\Box^{2k+1}(|\omega|t)^k|n|}{3\cdot k!}\sum_{\substack{p,q,r\in\mathbb Z^\nu\\p+q+r=n}}\sum_{\substack{n^\prime\in\mathbb Z^{(2k-1)\nu}\\\sum_jn_j^\prime=p}}\sum_{\alpha\in\mathbb B^{(k-1)}}\prod_{j=1}^{2k-1}|n_j^\prime|^{\alpha_j}e^{-\frac{\kappa}{2}|n_j^\prime|}\cdot e^{-\frac{\kappa}{2}|q|}\cdot e^{-\frac{\kappa}{2}|r|}\\
&\leq&\frac{\Box^{2k+1}(|\omega|t)^k}{3\cdot k!}\sum_{\substack{(n^{\prime},n_{2k},n_{2k+1})\in\mathbb Z^{(2k+1)\nu}\\n^\prime+n_{2k}+n_{2k+1}=n}}\left(\sum_{j=1}^{2k-1}|n_j^\prime|+|n_{2k}|+|n_{2k+1}|\right)\\
&\times&\sum_{\alpha\in\mathbb B^{(k-1)}}\prod_{j=1}^{2k-1}|n_j^\prime|^{\alpha_j}e^{-\frac{\kappa}{2}|n_j^\prime|}\cdot e^{-\frac{\kappa}{2}|n_{2k}|}\cdot e^{-\frac{\kappa}{2}|n_{2k+1}|}\\
&=&\frac{\Box^{2k+1}(|\omega|t)^k}{3\cdot k!}\sum_{\substack{(n_1,\cdots,n_{2k},n_{2k+1})\in\mathbb Z^{(2k+1)\nu}:~\sum_{j=1}^{2k+1}n_j=n}}\sum_{j=1}^{2k+1}|n_j|\sum_{\alpha\in\mathbb B^{(k-1)}}\prod_{j=1}^{2k-1}|n_j|^{\alpha_j}\prod_{j=1}^{2k+1}e^{-\frac{\kappa}{2}|n_j|}\\
&=&\frac{\Box^{2k+1}(|\omega|t)^k}{3\cdot k!}\sum_{\substack{(n_1,\cdots,n_{2k},n_{2k+1})\in\mathbb Z^{(2k+1)\nu}\\\sum_{j=1}^{2k+1}n_j=n}}
\sum_{\alpha\in\mathfrak g^{(k)}}\prod_{j=1}^{2k+1}|n_j|^{\alpha_j}\sum_{\alpha\in\mathbb B^{(k-1)}\times\{0\}\times\{0\}}\prod_{j=1}^{2k+1}|n_j|^{\alpha_j}
\prod_{j=1}^{2k+1}e^{-\frac{\kappa}{2}|n_j|}\\
&=&\frac{\Box^{2k+1}(|\omega|t)^k}{3\cdot k!}\sum_{\substack{(n_1,\cdots,n_{2k},n_{2k+1})\in\mathbb Z^{(2k+1)\nu}:~\sum_{j=1}^{2k+1}n_j=n}}\sum_{\alpha\in\mathbb B^{(k)}}\prod_{j=1}^{2k+1}|n_j|^{\alpha_j}e^{-\frac{\kappa}{2}|n_j|}.
\end{eqnarray*}
Similarly we have
\begin{eqnarray*}
\flat^2&\leq&\frac{\Box^{2k+1}(|\omega|t)^k}{3\cdot k!}\sum_{\substack{(n_1,\cdots,n_{2k},n_{2k+1})\in\mathbb Z^{(2k+1)\nu}:~\sum_{j=1}^{2k+1}n_j=n}}\sum_{\alpha\in\mathbb B^{(k)}}\prod_{j=1}^{2k+1}|n_j|^{\alpha_j}e^{-\frac{\kappa}{2}|n_j|},\\
\flat^3&\leq&\frac{\Box^{2k+1}(|\omega|t)^k}{3\cdot k!}\sum_{\substack{(n_1,\cdots,n_{2k},n_{2k+1})\in\mathbb Z^{(2k+1)\nu}:~\sum_{j=1}^{2k+1}n_j=n}}\sum_{\alpha\in\mathbb B^{(k)}}\prod_{j=1}^{2k+1}|n_j|^{\alpha_j}e^{-\frac{\kappa}{2}|n_j|}.
\end{eqnarray*}
Hence
\begin{align*}
&|c_k(t,n)-c_{k-1}(t,n)|\\
\leq&\flat^1+\flat^2+\flat^3\\
\leq&3\cdot\frac{\Box^{2k+1}(|\omega|t)^k}{3\cdot k!}\sum_{\substack{(n_1,\cdots,n_{2k},n_{2k+1})\in\mathbb Z^{(2k+1)\nu}:~\sum_{j=1}^{2k+1}n_j=n}}\sum_{\alpha\in\mathbb B^{(k)}}\prod_{j=1}^{2k+1}|n_j|^{\alpha_j}e^{-\frac{\kappa}{2}|n_j|}\\
=&\frac{\Box^{2k+1}(|\omega|t)^k}{k!}\sum_{\substack{(n_1,\cdots,n_{2k},n_{2k+1})\in\mathbb Z^{(2k+1)\nu}:~\sum_{j=1}^{2k+1}n_j=n}}\sum_{\alpha\in\mathbb B^{(k)}}\prod_{j=1}^{2k+1}|n_j|^{\alpha_j}e^{-\frac{\kappa}{2}|n_j|}.
\end{align*}
This completes the proof of Lemma \ref{pp}.
\end{proof}

\begin{property}\label{ggg}
$\alpha \in \mathbb B^{(k)} \Rightarrow \alpha \in \mathbb R^{2k+1}$ and $|\alpha| = k$.
\end{property}

\begin{proof}
Obviously, for any $\alpha\in\mathbb B^{(k)}$, $\alpha\in\mathbb R^{2k+1}$. Also, for $k=1$, $|\alpha|=1$.

Let $k\geq2$. Assume the property holds for $1\leq k^\prime\leq k-1$. For $k$, there exist $\alpha^\prime\in\mathbb B^{(k-1)}$ and $\mu\in\mathfrak g^{(k)}$ such that $\alpha=\alpha^\prime+\mu$. Hence
\begin{align*}
|\alpha|=|\alpha^\prime|+|\mu|=(k-1)+1=k.
\end{align*}
This completes the proof of Property \ref{ggg}.
\end{proof}

\begin{lemm}[Continuation of the Estimates of $|c_{k}(t,n)-c_{k-1}(t,n)|$]\label{dd}
For all $k\geq1$, we have
\begin{align}
\nonumber&|c_k(t,n)-c_{k-1}(t,n)|
\leq\Theta e^{-\frac{\kappa}{4}|n|}\frac{\left(\Box^{2}(12\kappa^{-1})^{2\nu+1}|\omega|t\right)^k}{k!}\sum_{\alpha\in\mathbb B^{(k)}}\prod_j\alpha_j!,
\quad \Theta=(12\kappa^{-1})^{\nu}\Box.
\end{align}
\end{lemm}

\begin{proof}
Note that
\begin{eqnarray*}
\nonumber&&|c_k(t,n)-c_{k-1}(t,n)|\\
&\stackrel{(\text{\scriptsize Lemma \ref{de}})}{\leq}&\frac{\Box^{2k+1}(|\omega|t)^k}{k!}\sum_{\substack{(n_1,\cdots,n_{2k},n_{2k+1})\in\mathbb Z^{(2k+1)\nu}\\\sum_{j=1}^{2k+1}n_j=n}}
\sum_{\alpha\in\mathbb B^{(k)}}\prod_{j=1}^{2k+1}|n_j|^{\alpha_j}e^{-\frac{\kappa}{2}|n_j|}~\\
&=&\frac{\Box^{2k+1}(|\omega|t)^k}{k!}\sum_{\alpha\in\mathbb B^{(k)}}\sum_{\substack{(n_1,\cdots,n_{2k},n_{2k+1})\in\mathbb Z^{(2k+1)\nu}\\\sum_{j=1}^{2k+1}n_j=n}}\prod_{j=1}^{2k+1}|n_j|^{\alpha_j}e^{-\frac{\kappa}{2}|n_j|}\\
&\stackrel{(\text{\scriptsize Lemma \ref{4}})}{\leq}&e^{-\frac{\kappa}{4}|n|}\frac{\Box^{2k+1}(|\omega|t)^k}{k!}\sum_{\alpha\in\mathbb B^{(k)}}(12\kappa^{-1})^{|\alpha|+(2k+1)\nu}\prod_j\alpha_j!
~\\
&\stackrel{(\text{\scriptsize Property \ref{ggg}})}{=}&e^{-\frac{\kappa}{4}|n|}\frac{\Box^{2k+1}(|\omega|t)^k}{k!}\sum_{\alpha\in\mathbb B^{(k)}}(12\kappa^{-1})^{k+(2k+1)\nu}\prod_j\alpha_j!
~\\
&=&\Theta e^{-\frac{\kappa}{4}|n|}\frac{\left(\Box^{2}(12\kappa^{-1})^{2\nu+1}|\omega|t\right)^k}{k!}\sum_{\alpha\in\mathbb B^{(k)}}\prod_j\alpha_j!.
%\leq&(12\kappa^{-1})^\nu e^{-\frac{\kappa}{2}|n|}\frac{\Box^{2k+1}\left((12\kappa^{-1})^{2\nu+1}|\omega|t\right)^k}{k!}(2N)^k~(N=2k+1)\\
%=&\frac{\Box^{2k+1}(|\omega|t)^k}{k!}\sum_{\substack{(n_1,\cdots,n_{2k},n_{2k+1})\in\mathbb Z^{(2k+1)\nu}:\sum_{j=1}^{2k+1}n_j=n}}\sum_{\alpha\in\mathbb B^{(k)}}\prod_{j=1}^{2k+1}|n_j|^{\alpha_j}e^{-\frac{\kappa}{4}|n_j|}\cdot\prod_{j=1}^{2k+1}e^{-\frac{\kappa}{4}|n_j|}\\
%=&\frac{\Box^{2k+1}(|\omega|t)^k}{k!}\sum_{\substack{(n_1,\cdots,n_{2k},n_{2k+1})\in\mathbb Z^{(2k+1)\nu}:\sum_{j=1}^{2k+1}n_j=n}}\sum_{\alpha\in\mathbb B^{(k)}}\prod_{j=1}^{2k+1}|n_j|^{\alpha_j}e^{-\frac{\kappa}{4}|n_j|}\cdot e^{-\frac{\kappa}{4}\sum_{j=1}^{2k+1}|n_j|}\\
%\leq&\frac{\Box^{2k+1}(|\omega|t)^k}{k!}e^{-\frac{\kappa}{4}|n|}\sum_{\alpha\in\mathbb B^{(k)}}\sum_{\substack{(n_1,\cdots,n_{2k},n_{2k+1})\in\mathbb Z^{(2k+1)\nu}:\sum_{j=1}^{2k+1}n_j=n}}\prod_{j=1}^{2k+1}|n_j|^{\alpha_j}e^{-\frac{\kappa}{4}|n_j|}\\
%\leq&\frac{\Box^{2k+1}(|\omega|t)^k}{k!}e^{-\frac{\kappa}{4}|n|}\sum_{\alpha\in\mathbb B^{(k)}}
%(24\kappa^{-1})^{|\alpha|+(2k+1)\nu}\prod_{j=1}^{2k+1}\alpha_j!~(\text{\scriptsize by Lemma \ref{4}})\\
%=&\frac{\Box^{2k+1}(|\omega|t)^k}{k!}\cdot(24\kappa^{-1})^{k+(2k+1)\nu}\cdot e^{-\frac{\kappa}{4}|n|}\sum_{\alpha\in\mathbb B^{(k)}}
%\prod_{j=1}^{2k+1}\alpha_j!\\
%=&\Theta e^{-\frac{\kappa}{4}|n|}\frac{\left(\Box^{2}(24\kappa^{-1})^{2\nu+1}|\omega|t\right)^k}{k!}\sum_{\alpha\in\mathbb B^{(k)}}
%\prod_{j=1}^{2k+1}\alpha_j!,
\end{eqnarray*}
where $\Theta=(12\kappa^{-1})^\nu\Box$. This completes the proof of Lemma \ref{dd}.
\end{proof}

\subsection{Estimates of $\sum_{\alpha\in\mathbb B^{(k)}} \prod_{j}\alpha_j!$}
%Next we introduce some notations for the estimates of $\sum_{\alpha\in\mathbb B^{(k)}}
%\prod_{j=1}^{2k+1}\alpha_j!$.

Set
$$
\mathcal H(N;d):=\left\{\alpha\in\mathbb Z^{N}: \sum_{j}\alpha_j=d,\alpha_j\geq0\right\}.
$$
For any $\alpha \in \mathcal H(N;d)$, set
\[
+(\alpha) := \{ i : \alpha_i>0 \}, \quad j_\ast(\alpha) := \min \left\{ j : \alpha_j \triangleq \min_{i\in+(\alpha)} \alpha_i \right\},
\]
and define $\phi(\alpha):=(\phi_1(\alpha),\cdots,\phi_N(\alpha))$, where
\begin{align*}
\phi_j(\alpha):=
\begin{cases}
\alpha_j,&j\neq j_\ast(\alpha);\\
\alpha_j-1,&j= j_\ast(\alpha),\\
\end{cases}
\quad j=1,2,\cdots,N.
\end{align*}

\begin{property}\label{phi}
For the mapping $\phi$, we have
\begin{enumerate}
  \item $\phi$ maps $\mathcal H(N;d)$ into $\mathcal H(N;d-1)$.
  \item $\phi_{j_\ast(\alpha)}(\alpha)<\min_{\{j:\phi_j(\alpha)>0,j\neq j_\ast(\alpha)\}}\phi_j(\alpha)$.
  \item If $\phi(\alpha)=\phi(\alpha^\prime)$, then $\alpha_j=\alpha_j^\prime$ for all $j\notin\{j_\ast(\alpha),j_\ast(\alpha^\prime)\}$.
  \item If $\phi(\alpha)=\phi(\alpha^\prime)$ and $j_\ast(\alpha)=j_\ast(\alpha^\prime)$, then $\alpha_j=\alpha_j^\prime$ for all $j=1,2,\cdots,N$.
  \item If $\phi(\alpha)=\phi(\alpha^\prime)$, $\alpha_{j_\ast(\alpha)}>1$ and $\alpha^{\prime}_{j_\ast(\alpha^\prime)}>1$, then $\alpha=\alpha^\prime$.
\end{enumerate}
\end{property}

\begin{proof}
\begin{enumerate}
  \item For any $\alpha\in\mathcal H(N;d)$, $\phi(\alpha)=(\phi_1(\alpha),\cdots,\phi_{N}(\alpha))\in\mathbb Z^N$. Also,
  \begin{eqnarray*}
  |\phi(\alpha)|&=&\sum_{j=1}^{N}\phi_j(\alpha)\\
  %&=\sum_{j=1}^N\alpha_j-\chi_{\{j_\ast(\alpha)\}}
  &=&\sum_{j=1,j\neq j_\ast(\alpha)}^{N}\phi_j(\alpha)+\phi_{j_\ast}(\alpha)\\
  &=&\sum_{j=1,j\neq j_\ast(\alpha)}^{N}\alpha_j+(\alpha_{j_\ast}-1)\\
  &=&\sum_{j=1}^{N}\alpha_j-1\\
 % &=&|\alpha|-1\\
  &=&d-1.
  \end{eqnarray*}
  Hence $\phi(\alpha)\in\mathcal H(N;d-1)$.
  \item This follows from the definition of $j_\ast(\alpha)$.
  \item If $\phi(\alpha)=\phi(\alpha^\prime)$, that is $\phi_j(\alpha)=\phi_j(\alpha^\prime)$ for all $j=1,2,\cdots,N$. By the definition of $\phi_j$, one has $\alpha_j=\alpha_j^\prime$ for all $j\notin\{j_\ast(\alpha),j_\ast(\alpha^\prime)\}$.
  \item By (3), we have $\alpha_j=\alpha_j^\prime$ for all $j\notin\{j_\ast(\alpha),j_\ast(\alpha^\prime)\}$. Since $j_\ast(\alpha)=j_\ast(\alpha^\prime)$, we have $\alpha_{j_\ast(\alpha)}-1=\alpha^{\prime}_{j_\ast(\alpha^{\prime})}-1$, and hence $\alpha_{j_\ast(\alpha)}=\alpha^{\prime}_{j_\ast(\alpha^\prime)}$. Thus, $\alpha = \alpha^\prime$.
  \item By (3), we have $\alpha_j=\alpha_j^\prime$ for all $j\notin\{j_\ast(\alpha),j_\ast(\alpha^\prime)\}$. Hence we discuss only $j\in\{j_\ast(\alpha),j_\ast(\alpha^\prime)\}$. If $\phi(\alpha)=\phi(\alpha^\prime), \alpha_{j_\ast(\alpha)}>1$ and $\alpha^\prime_{j_\ast(\alpha^\prime)}>1$, then $j_\ast(\alpha)=j_\ast(\alpha^\prime)$. If not, that is $j_\ast(\alpha)\neq j_\ast(\alpha^\prime)$. It follows from $\phi(\alpha)=\phi(\alpha^\prime)$ that $\phi_{j_\ast(\alpha^\prime)}(\alpha)=\phi_{j_\ast(\alpha^\prime)}(\alpha^\prime)=\alpha_{j_\ast(\alpha^\prime)}^\prime-1>0$.
       %Since $\alpha^\prime_{j_\ast(\alpha^\prime)}>1$, then $\phi_{j_\ast(\alpha^\prime)}(\alpha^\prime)=\alpha_{j_\ast(\alpha^\prime)}^\prime-1>0$.
       Hence $j_\ast(\alpha^\prime)\in\{j:\phi_j(\alpha)>0,j\neq j_\ast(\alpha)\}.$ By (2), one has
      \begin{align*}
      \phi_{j_\ast(\alpha)}(\alpha)<\min_{\{j:\phi_j(\alpha)>0,j\neq j_\ast(\alpha)\}}\phi_j(\alpha)\leq\phi_{j_\ast(\alpha^\prime)}(\alpha).
      \end{align*}
  Analogously, $j_\ast(\alpha)\in\{j:\phi_j(\alpha^\prime)>0,j\neq j_\ast(\alpha^\prime)\}$  and
  \begin{align*}
      \phi_{j_\ast(\alpha^\prime)}(\alpha^\prime)<\min_{\{j:\phi_j(\alpha^\prime)>0,j\neq j_\ast(\alpha^\prime)\}}\phi_j(\alpha^\prime)\leq\phi_{j_\ast(\alpha)}(\alpha^\prime).
      \end{align*}
  Since $\phi(\alpha)=\phi(\alpha^\prime)$, then $\phi_{j_\ast(\alpha^\prime)}(\alpha)=\phi_{j_\ast(\alpha^\prime)}(\alpha^\prime)$. Hence $\phi_{j_\ast(\alpha)}(\alpha)<\phi_{j_\ast(\alpha)}(\alpha^\prime)=\phi_{j_\ast(\alpha)}(\alpha)$. This leads to a contradiction. Hence $j\in\{j_\ast(\alpha),j_\ast(\alpha^\prime)\}$. By (4), one has $\alpha=\alpha^\prime$.
\end{enumerate}
\end{proof}

\begin{rema}
Set
\[\mathcal H^\prime(N;d):=\{\alpha\in\mathcal H(N;d);\alpha_{j_\ast(\alpha)}>1\}, \quad \mathcal H^{\prime\prime}(N;d):=\mathcal H(N;d)\backslash\mathcal H^\prime(N;d).\]
It follows from Property \ref{phi} that $\phi$ is injective if it is restricted on $\mathcal H^{\prime}(N;d)$.
\end{rema}

%With the help of these abstract symbols, we have
\begin{lemm}\label{kk}
For $d\leq N$, we have
\begin{align*}
\sum_{\alpha\in\mathcal H(N;d)}\prod_j\alpha_j!\leq(d+N)\sum_{\alpha\in\mathcal H(N;d-1)}\prod_j\alpha_j!\leq\cdots<(2N)^d.
\end{align*}
\end{lemm}

\begin{proof}
We first have
\begin{eqnarray*}
&&\sum_{\alpha\in\mathcal H(N;d)}\prod_j\alpha_j!\\
&=&\sum_{\alpha\in\mathcal H(N;d)}\alpha_{j_{\ast}(\alpha)}!\prod_{j\neq j_{\ast}(\alpha)}\alpha_j!\\
&\stackrel{(\text{\rm def})}{=}&\sum_{\alpha\in\mathcal H(N;d)}\alpha_{j_{\ast}(\alpha)}\left(\alpha_{j_{\ast}(\alpha)}-1\right)!\prod_{j\neq j_{\ast}(\alpha)}\alpha_j\\
&\stackrel{(\text{\rm def})}{=}&\sum_{\alpha\in\mathcal H(N;d)}\alpha_{j_{\ast}(\alpha)}\phi_{j_{\ast}(\alpha)}(\alpha)!\prod_{j\neq j_{\ast}(\alpha)}\phi_j(\alpha)!\\
&\stackrel{}{=}&\sum_{\alpha\in\mathcal H(N;d)}\alpha_{j_{\ast}(\alpha)}\prod_{j}\phi_j(\alpha)!\\
&\stackrel{(\text{\rm def})}{=}&\sum_{\alpha\in\mathcal H^{\prime}(N;d)}\alpha_{j_{\ast}(\alpha)}\prod_{j}\phi_j(\alpha)!+\sum_{\alpha\in\mathcal H^{\prime\prime}(N;d)}\alpha_{j_{\ast}(\alpha)}\prod_{j}\phi_j(\alpha)!.
\end{eqnarray*}
On the one hand,
\begin{eqnarray*}
&&\sum_{\alpha\in\mathcal H^{\prime}(N;d)}\alpha_{j_{\ast}(\alpha)}\prod_{j}\phi_j(\alpha)!\\
&\stackrel{(\alpha_{j_{\ast}(\alpha)}\leq d~~\text{\rm in}~~\mathcal H^{\prime}(N;d))}{\leq}&d\sum_{\alpha\in\mathcal H^{\prime}(N;d)}\prod_{j}\phi_j(\alpha)!\\
&\stackrel{}{=}&d\sum_{\beta\in\phi(\mathcal H^{\prime}(N;d))}\prod_j\beta_j!\\
&\stackrel{(\phi:~\mathcal H^{\prime}(N;d)\rightarrow\mathcal H(N;d-1))}{\leq}&d\sum_{\beta\in\mathcal H(N;d-1)}\prod_j\beta_j!\\
&\stackrel{}{=}&d\sum_{\alpha\in\mathcal H(N;d-1)}\prod_j\alpha_j!.
\end{eqnarray*}
On the other hand,
\begin{eqnarray*}
&&\sum_{\alpha\in\mathcal H^{\prime\prime}(N;d)}\alpha_{j_{\ast}(\alpha)}\prod_{j}\phi_j(\alpha)!\\
&\stackrel{(\alpha_{j_{\ast}(\alpha)}=1~~\text{\rm in}~~\mathcal H^{\prime\prime}(N;d))}{=}&\sum_{\alpha\in\mathcal H^{\prime}(N;d)}\prod_{j}\phi_j(\alpha)!\\
&\stackrel{}{=}&\sum_{\beta\in\phi(\mathcal H^{\prime\prime}(N;d))}\sum_{\alpha\in\phi^{-1}(\beta)}\prod_{j}\phi_j(\alpha)!\\
&\stackrel{(\text{\rm card}~\phi^{-1}(\beta)\leq N)}{\leq}&N\sum_{\beta\in\phi(\mathcal H^{\prime\prime}(N;d))}\beta_j(\alpha)!\\
&\stackrel{(\phi:~\mathcal H^{\prime\prime}(N;d)\rightarrow\mathcal H(N;d-1))}{\leq}&N\sum_{\beta\in\mathcal H(N;d-1)}\prod_j\beta_j!\\
&\stackrel{}{=}&N\sum_{\alpha\in\mathcal H(N;d-1)}\prod_j\alpha_j!.
\end{eqnarray*}
Hence
\begin{eqnarray*}
&&\sum_{\alpha\in\mathcal H(N;d)}\prod_j\alpha_j!\\
&=&\sum_{\alpha\in\mathcal H^{\prime}(N;d)}\alpha_{j_{\ast}(\alpha)}\prod_{j}\phi_j(\alpha)!+\sum_{\alpha\in\mathcal H^{\prime\prime}(N;d)}\alpha_{j_{\ast}(\alpha)}\prod_{j}\phi_j(\alpha)!\\
&\leq&d\sum_{\alpha\in\mathcal H(N;d-1)}\prod_j\alpha_j!+N\sum_{\alpha\in\mathcal H(N;d-1)}\prod_j\alpha_j!\\
&=&(d+N)\sum_{\alpha\in\mathcal H(N;d-1)}\prod_j\alpha_j!\\
&\stackrel{(\text{\rm successive iteration})}{\leq}&(d+N)(d-1+N)\sum_{\alpha\in\mathcal H(N;d-2)}\prod_{j=1}\alpha_j!\\
&\vdots&\\
&\stackrel{(\text{\rm successive iteration})}{\leq}&(d+N)(d-1+N)\cdots(1+N)\sum_{\alpha\in\mathcal H(N;0)}\prod_{j=1}\alpha_j!\\
&=&(d+N)(d-1+N)\cdots(2+N)(1+N)\cdot1\\
&\stackrel{(d\leq N)}{\leq}&\underbrace{(2N)\cdot(2N)\cdot\cdots\cdot(2N)}_{d}\\
&=&(2N)^d.
\end{eqnarray*}
This completes the proof Lemma \ref{kk}.
\end{proof}

\subsection{Continuation of Estimates of $|c_{k}(t,n)-c_{k-1}(t,n)|$}

\begin{lemm}[Continuation of estimates of $|c_{k}(t,n)-c_{k-1}(t,n)|$]\label{1e}
For all $k\geq1$, we have
\[
|c_{k}(t,n)-c_{k-1}(t,n)|\leq\Theta e^{-\frac{\kappa}{4}|n|+\frac{1}{2}}\left(4e\Box^{2}(12\kappa^{-1})^{2\nu+1}|\omega|t\right)^k.
\]
\end{lemm}

\begin{proof}
By the Stirling formula,
\[
n!=\sqrt{2\pi n}\left(\frac{n}{e}\right)^n\left(1+o\left(1/n\right)\right)\gtrsim\left(\frac{n}{e}\right)^n,
\]
we find
\begin{eqnarray*}
&&|c_{k}(t,n)-c_{k-1}(t,n)|\\
&\stackrel{(\text{\scriptsize Lemma \ref{dd}})}{\leq}&\Theta e^{-\frac{\kappa}{4}|n|}\frac{\left(\Box^{2}(12\kappa^{-1})^{2\nu+1}|\omega|t\right)^k}{k!}\sum_{\alpha\in\mathbb B^{(k)}}\prod_j\alpha_j!~\\
&\stackrel{(\text{\scriptsize Lemma \ref{kk}})}{\leq}&\Theta e^{-\frac{\kappa}{4}|n|}\frac{\left(\Box^{2}(12\kappa^{-1})^{2\nu+1}|\omega|t\right)^k}{k!}\left(2N\right)^{k}~(N=2k+1)~~\\
&\leq&\Theta e^{-\frac{\kappa}{4}|n|}\left(\Box^{2}(12\kappa^{-1})^{2\nu+1}|\omega|t\right)^k\left(\frac{e}{k}\right)^k\left(4k+2\right)^{k}\\
&\leq&\Theta e^{-\frac{\kappa}{4}|n|}\left(\Box^{2}(12\kappa^{-1})^{2\nu+1}|\omega|t\right)^k(4e)^k\left(1+\frac{1}{2k}\right)^{k}\\
&\leq&\Theta e^{-\frac{\kappa}{4}|n|+\frac{1}{2}}\left(4e\Box^{2}(12\kappa^{-1})^{2\nu+1}|\omega|t\right)^k.
\end{eqnarray*}
This completes the proof of Lemma \ref{1e}.
\end{proof}

\subsection{Estimates of $|c_{k+\bigstar}(t,n)-c_{k}(t,n)|$}\label{cthm}

\begin{lemm}[Estimates of $|c_{k+\bigstar}(t,n)-c_{k}(t,n)|$]\label{gthm}
For
$$
0 < t < \min \left\{ \frac{\kappa^{2\nu+1}}{4e\Box^212^{2\nu+1}|\omega|}, \frac{\kappa^{2\nu+1}}{16\mathcal A|\omega|12^{2\nu+1}} \right\} \triangleq t_0 > 0,
$$
and all $k\geq1$, we have
\begin{align}\label{cce}
|c_{k+\bigstar}(t,n)-c_{k}(t,n)|\leq\frac{\Theta e^{-\frac{\kappa}{4}|n|+\frac{1}{2}}}{1-4e\Box^{2}(12\kappa^{-1})^{2\nu+1}|\omega|t}\left(4e\Box^{2}
(12\kappa^{-1})^{2\nu+1}|\omega|t\right)^{k+1}\quad(\text{uniformly for}~\bigstar).
\end{align}
\end{lemm}

\begin{proof}
We have
\begin{eqnarray*}
&&|c_{k+\bigstar}(t,n)-c_{k}(t,n)|\\&\leq&\sum_{j=1}^{\bigstar}|c_{k+j}(t,n)-c_{k+j-1}(t,n)|\\
&\leq&\Theta e^{-\frac{\kappa}{4}|n|+\frac{1}{2}}\sum_{j=1}^{\bigstar}\left(4e\Box^{2}(12\kappa^{-1})^{2\nu+1}|\omega|t\right)^{k+j}\\
&\leq&\Theta e^{-\frac{\kappa}{4}|n|+\frac{1}{2}}\left(4e\Box^{2}(12\kappa^{-1})^{2\nu+1}|\omega|t\right)^{k}\sum_{j=1}^{\bigstar}
\left(4e\Box^{2}(12\kappa^{-1})^{2\nu+1}|\omega|t\right)^{j}\\
&\leq&\Theta e^{-\frac{\kappa}{4}|n|+\frac{1}{2}}\left(4e\Box^{2}(12\kappa^{-1})^{2\nu+1}|\omega|t\right)^{k}\sum_{j=1}^{+\infty}
\left(4e\Box^{2}(12\kappa^{-1})^{2\nu+1}|\omega|t\right)^{j}\\
&=&\frac{\Theta e^{-\frac{\kappa}{4}|n|+\frac{1}{2}}}{1-4e\Box^{2}(12\kappa^{-1})^{2\nu+1}|\omega|t}\left(4e\Box^{2}(12\kappa^{-1})^{2\nu+1}|\omega|t\right)^{k+1}
\end{eqnarray*}
provided that $0<t<t_0$.
\end{proof}

\section{Proof of the Main Results}

\subsection{Proof of Theorem \ref{eethm} in the Case $\mathfrak p=3$}

In this subsection, we prove the existence of a spatially quasi-periodic solution to $3$-gKdV with spatially quasi-periodic initial data \eqref{iei}, as stated in Theorem \ref{eethm}.

\begin{proof}[Proof of Theorem \ref{eethm}.]
By Lemma \ref{gthm}, it is easy to see that the Picard sequence $\{c_k(t,n)\}$ is a Cauchy sequence. Hence there exists a limit function
$$
c^\dag(t,n) := \lim_{k\rightarrow\infty}c_k(t,n), \quad0\leq t<t_0,  n\in\mathbb Z^\nu.
$$
Letting $\bigstar\rightarrow\infty$ in \eqref{cce}, we find
\[
|c^\dag(t,n)-c_k(t,n)|\leq\frac{\Theta e^{-\frac{\kappa}{4} |n| + \frac{1}{2}}}{1-4e\Box^{2}(12 \kappa^{-1})^{2\nu+1} |\omega|t} \left( 4e \Box^{2}
(12\kappa^{-1})^{2\nu+1} |\omega| t \right)^{k+1} \rightarrow 0 \quad \text{as} \quad k \rightarrow +\infty.
\]
By Proposition \ref{thm}, we have
\begin{align*}
|c^\dag(t,n)|\leq|c^\dag(t,n)-c_k(t,n)|+|c_k(t,n)|\leq\Box e^{-\frac{\kappa}{4}|n|}.
\end{align*}
Letting $k\rightarrow+\infty$ in \eqref{ppp}, one can obtain that
\[c^\dag(t,n)=c_0(t,n)-\frac{{\rm i}n\cdot\omega}{3}\int_{0}^{t}e^{{\rm i}(n\cdot\omega)^3(t-\tau)}
\sum_{p,q,r\in\mathbb Z^\nu:~ p+q+r=n}c^\dag(\tau,p)c^\dag(\tau,q)c^\dag(\tau,r){\rm d}\tau.\]
Define the spatially $\omega$-quasi-periodic function
\[u^\dag(t,x):=\sum_{n\in\mathbb Z^\nu}c^\dag(t,n)e^{{\rm i}n\cdot\omega x}, \quad 0\leq t<t_0, x\in\mathbb R.\]
Obviously, $u=u^\dag$ satisfies the following differential equation,
\[
\partial_tu+\partial_x^3u-v^\dag=0,
\]
where
\[
v^\dag(t,x)=\sum_{n\in\mathbb Z^\nu}g(t,n)e^{{\rm i}(n\cdot\omega)x}, \quad g^\dag(t,n)=-\frac{{\rm i}n\cdot\omega}{3}
\sum_{p,q,r\in\mathbb Z^\nu:~ p+q+r=n}c^\dag(\tau,p)c^\dag(\tau,q)c^\dag(\tau,r).
\]
Notice that
\[
-(u^\dag)^2\partial_xu^\dag=\partial_x\left(-\frac{(u^\dag)^3}{3}\right) = v^\dag.
\]
Hence $u^\dag$ satisfies 3-gKdV with spatially quasi-periodic initial data \eqref{iei}. This completes the proof of Theorem \ref{eethm} in the case $\mathfrak p=3$.
\end{proof}

\subsection{Proof of Theorem \ref{euthm} in the Case $\mathfrak p=3$}

In this subsection, we prove the uniqueness of the spatially quasi-periodic solution to $3$-gKdV with spatially quasi-periodic initial data \eqref{iei}, as formulated in Theorem \ref{euthm}.

\begin{lemm}\label{ueee}
For all $k\geq1$, we have
\begin{align}\label{ll}
|c(t,n)-b(t,n)|\leq\frac{\Box^{2k+1}(|\omega|t)^k}{k!}\sum_{\substack{(n_1,\cdots,n_{2k},n_{2k+1})\in\mathbb Z^{(2k+1)\nu}\\ \sum_jn_j=n}}\sum_{\alpha\in\mathbb B^{(k)}}
\prod_{j}|n_j|^{\alpha_j}e^{-\rho|n_j|}.
\end{align}
\end{lemm}

\begin{proof}
Since $c(0,n)=b(0,n)$, we have
\begin{eqnarray*}
&&|c(t,n)-b(t,n)|\\
&\leq&\frac{|n||\omega|}{3}\int_0^t\sum_{p,q,r\in\mathbb Z^\nu:~p+q+r=n}|c(\tau,p)c(\tau,q)c(\tau,r)-b(\tau,p)b(\tau,q)b(\tau,r)|{\rm d}\tau\\
&\leq&|n||\omega|t\sum_{p,q,r\in\mathbb Z^\nu:~p+q+r=n}\Box^3e^{-\rho(|p|+|q|+|r|)}\\
&\leq&\Box^3|\omega|t\sum_{p,q,r\in\mathbb Z^\nu:~p+q+r=n}(|p|+|q|+|r|)e^{-\rho(|p|+|q|+|r|)}\\
&=&\Box^3|\omega|t\sum_{(n_1,n_2,n_3)\in\mathbb Z^{3\nu}:~ n_1+n_2+n_3=n}\sum_{\alpha\in\mathbb B^{(1)}}\prod_{j}|n_j|^{\alpha_j}e^{-\rho|n_j|}.
\end{eqnarray*}
This shows that \eqref{ll} holds for $k=1$.

Let $k\geq2$. Suppose that \eqref{ll} is true for all $1\leq k^\prime\leq k-1$. For $k$, we have
\begin{eqnarray*}
&&|c(t,n)-b(t,n)|\\
&\leq&\frac{|n||\omega|}{3}\int_0^t\sum_{p,q,r\in\mathbb Z^\nu:~p+q+r=n}|c(\tau,p)c(\tau,q)c(\tau,r)-b(\tau,p)b(\tau,q)b(\tau,r)|{\rm d}\tau\\
&\leq&\frac{|n||\omega|}{3}\int_0^t\sum_{p,q,r\in\mathbb Z^\nu:~p+q+r=n}|c(\tau,p)-b(\tau,p)||c(\tau,q)||c(\tau,r)|{\rm d}\tau\\
&+&\frac{|n||\omega|}{3}\int_0^t\sum_{p,q,r\in\mathbb Z^\nu:~p+q+r=n}|c(\tau,q)-b(\tau,q)||b(\tau,p)||c(\tau,r)|{\rm d}\tau\\
&+&\frac{|n||\omega|}{3}\int_0^t\sum_{p,q,r\in\mathbb Z^\nu:~p+q+r=n}|c(\tau,r)-b(\tau,r)||b(\tau,p)||b(\tau,q)|{\rm d}\tau\\
&\triangleq&\psi_1+\psi_2+\psi_3,
\end{eqnarray*}
where
\begin{eqnarray*}
\psi_1&=&\frac{|n||\omega|}{3}\int_0^t\sum_{p,q,r\in\mathbb Z^\nu:~p+q+r=n}|c(\tau,p)-b(\tau,p)||c(\tau,q)||c(\tau,r)|{\rm d}\tau,\\
\psi_2&=&\frac{|n||\omega|}{3}\int_0^t\sum_{p,q,r\in\mathbb Z^\nu:~p+q+r=n}|c(\tau,q)-b(\tau,q)||b(\tau,p)||c(\tau,r)|{\rm d}\tau,\\
\psi_3&=&\frac{|n||\omega|}{3}\int_0^t\sum_{p,q,r\in\mathbb Z^\nu:~p+q+r=n}|c(\tau,r)-b(\tau,r)||b(\tau,p)||b(\tau,q)|{\rm d}\tau.
\end{eqnarray*}
For $\psi_1$ one can derive that
\begin{eqnarray*}
&&\psi_1\\
&=&\frac{|n||\omega|}{3}\int_0^t\sum_{p,q,r\in\mathbb Z^\nu:~p+q+r=n}|c(\tau,p)-b(\tau,p)||c(\tau,q)||c(\tau,r)|{\rm d}\tau\\
&\leq&\frac{|n||\omega|}{3}\int_0^t\sum_{\substack{p,q,r\in\mathbb Z^\nu\\ p+q+r=n}}\frac{\Box^{(2k-1)}(|\omega|\tau)^{k-1}}{(k-1)!}\sum_{\substack{
n^\prime\in\mathbb Z^{(2k-1)\nu}\\ \mu(n^\prime)=p}}
\sum_{\alpha\in\mathbb B^{(k-1)}}\prod_{j}|n_j^\prime|^{\alpha_j}e^{-\rho|n_j^\prime|}\cdot\Box e^{-\rho|q|}\cdot\Box e^{-\rho|r|}{\rm d}\tau\\
&\leq&\frac{\Box^{2k+1}(|\omega|t)^k|n|}{3\cdot k!}\sum_{\substack{p,q,r\in\mathbb Z^\nu\\p+q+r=n}}\sum_{\substack{n^\prime\in\mathbb Z^{(2k-1)\nu}\\ \mu(n^\prime)=p}}
\sum_{\alpha\in\mathbb B^{(k-1)}}\prod_{j}|n_j^\prime|^{\alpha_j}e^{-\rho|n_j^\prime|}\cdot e^{-\rho|q|}\cdot e^{-\rho|r|}\\
&=&\frac{\Box^{2k+1}(|\omega|t)^k|n|}{3\cdot k!}\sum_{\substack{(n^\prime,n_{2k},n_{2k+1})\in\mathbb Z^{(2k+1)\nu}\\ \mu(n^\prime)+n_{2k}+n_{2k+1}=n}}\sum_{\alpha\in\mathbb B^{(k-1)}}
\prod_{j=1}^{2k-1}|n_j^\prime|^{\alpha_j}e^{-\rho|n_j^\prime|}\cdot e^{-\rho|n_{2k}|}\cdot e^{-\rho|n_{2k+1}|}\\
&\leq&\frac{\Box^{2k+1}(|\omega|t)^k}{3\cdot k!}\sum_{\substack{(n_1,\cdots,,n_{2k},n_{2k+1})\in\mathbb Z^{(2k+1)\nu}\\ \sum_jn_j=n}}\sum_{j=1}^{2k+1}|n_j|
\sum_{\alpha\in\mathbb B^{(k-1)}}\prod_{j=1}^{2k-1}|n_j|^{\alpha_j}\prod_{j=1}^{2k+1}e^{-\rho|n_j|}\\
&=&\frac{\Box^{2k+1}(|\omega|t)^k}{3\cdot k!}\sum_{\substack{(n_1,\cdots,,n_{2k},n_{2k+1})\in\mathbb Z^{(2k+1)\nu}\\ \sum_jn_j=n}}
\sum_{\alpha\in\mathfrak g^{(k)}}\prod_j|n_j|^{\alpha_j}\sum_{\alpha\in\mathbb B^{(k-1)}\times\{0\}\times\{0\}}\prod_j|n_j|^{\alpha_j}\prod_{j=1}^{2k+1}e^{-\rho|n_j|}\\
&=&\frac{\Box^{2k+1}(|\omega|t)^k}{3\cdot k!}\sum_{\substack{(n_1,\cdots,n_{2k},n_{2k+1})\in\mathbb Z^{(2k+1)\nu}\\ \sum_jn_j=n}}\sum_{\alpha\in\mathbb B^{(k)}}
\prod_{j}|n_j|^{\alpha_j}e^{-\rho|n_j|}.
\end{eqnarray*}
Analogously we have
\begin{eqnarray*}
\psi_2&\leq&
\frac{\Box^{2k+1}(|\omega|t)^k}{3\cdot k!}\sum_{\substack{(n_1,\cdots,n_{2k},n_{2k+1})\in\mathbb Z^{(2k+1)\nu}\\ \sum_jn_j=n}}\sum_{\alpha\in\mathbb B^{(k)}}
\prod_{j}|n_j|^{\alpha_j}e^{-\rho|n_j|},\\
\psi_2&\leq&\frac{\Box^{2k+1}(|\omega|t)^k}{3\cdot k!}\sum_{\substack{(n_1,\cdots,n_{2k},n_{2k+1})\in\mathbb Z^{(2k+1)\nu}\\ \sum_jn_j=n}}\sum_{\alpha\in\mathbb B^{(k)}}
\prod_{j}|n_j|^{\alpha_j}e^{-\rho|n_j|}.
\end{eqnarray*}
Hence
\begin{eqnarray*}
|c(t,n)-b(t,n)|&\leq&\psi_1+\psi_2+\psi_3\\
&\leq&\frac{\Box^{2k+1}(|\omega|t)^k}{k!}\sum_{\substack{(n_1,\cdots,n_{2k},n_{2k+1})\in\mathbb Z^{(2k+1)\nu}\\ \sum_jn_j=n}}\sum_{\alpha\in\mathbb B^{(k)}}
\prod_{j}|n_j|^{\alpha_j}e^{-\rho|n_j|}.
\end{eqnarray*}
This completes the proof of Lemma \ref{ueee}.
\end{proof}

\begin{lemm}\label{dh}
For all $k\geq1$, we have
\[
|c(t,n) - b(t,n)| \leq (12\rho^{-1})^{\nu} \Box e^{-\frac{\rho}{2}|n|} \left( 4e \Box^2(12\rho^{-1})^{2\nu+1} |\omega| t \right)^k.
\]
\end{lemm}

\begin{proof}
We have
\begin{eqnarray*}
&&|c(t,n)-b(t,n)|\\
&\leq&\frac{\Box^{2k+1}(|\omega|t)^k}{k!}\sum_{\substack{(n_1,\cdots,n_{2k},n_{2k+1})\in\mathbb Z^{(2k+1)\nu}\\ \sum_jn_j=n}}\sum_{\alpha\in\mathbb B^{(k)}}
\prod_{j}|n_j|^{\alpha_j}e^{-\rho|n_j|}\\
&\leq&\frac{\Box^{2k+1}(|\omega|t)^k}{k!}\sum_{\substack{(n_1,\cdots,n_{2k},n_{2k+1})\in\mathbb Z^{(2k+1)\nu}\\ \sum_jn_j=n}}\sum_{\alpha\in\mathbb B^{(k)}}
\prod_{j}|n_j|^{\alpha_j}e^{-\frac{\rho}{2}|n_j|}\prod_{j}e^{-\frac{\rho}{2}|n_j|}\\
&\leq&\frac{\Box^{2k+1}(|\omega|t)^k}{k!}e^{-\frac{\rho}{2}|n|}\sum_{\substack{(n_1,\cdots,n_{2k},n_{2k+1})\in\mathbb Z^{(2k+1)\nu}\\ \sum_jn_j=n}}\sum_{\alpha\in\mathbb B^{(k)}}
\prod_{j}|n_j|^{\alpha_j}e^{-\frac{\rho}{2}|n_j|}\\
&=&\frac{\Box^{2k+1}(|\omega|t)^k}{k!}e^{-\frac{\rho}{2}|n|}\sum_{\alpha\in\mathbb B^{(k)}}\sum_{\substack{(n_1,\cdots,n_{2k},n_{2k+1})\in\mathbb Z^{(2k+1)\nu}\\ \sum_jn_j=n}}
\prod_{j}|n_j|^{\alpha_j}e^{-\frac{\rho}{2}|n_j|}\\
&\stackrel{(\text{\scriptsize Lemma \ref{4}})}{\leq}&\frac{\Box^{2k+1}(|\omega|t)^k}{k!}e^{-\frac{\rho}{2}|n|}\sum_{\alpha\in\mathbb B^{(k)}}(12\rho^{-1})^{|\alpha|+(2k+1)\nu}\prod_j\alpha_j!
~~\\
&\stackrel{(\text{\scriptsize Property \ref{ggg}})}{=}&(12\rho^{-1})^{\nu}\frac{\Box^{2k+1}\left((12\rho^{-1}\right)^{2\nu+1}|\omega|t)^k}{k!}e^{-\frac{\rho}{2}|n|}\sum_{\alpha\in\mathbb B^{(k)}}\prod_j\alpha_j!~~\\
&\leq&(12\rho^{-1})^{\nu}\frac{\Box^{2k+1}\left((12\rho^{-1}\right)^{2\nu+1}|\omega|t)^k}{k!}e^{-\frac{\rho}{2}|n|}(2N)^k~(N=2k+1)\\
&\leq&(12\rho^{-1})^{\nu}\Box^{2k+1}\left((12\rho^{-1}\right)^{2\nu+1}|\omega|t)^ke^{-\frac{\rho}{2}|n|+\frac{1}{2}}(4e)^k\\
&=&(12\rho^{-1})^{\nu}\Box e^{-\frac{\rho}{2}|n|+\frac{1}{2}}\left(4e\Box^2(12\rho^{-1})^{2\nu+1}|\omega|t\right)^k.
\end{eqnarray*}
This completes the proof of Lemma \ref{dh}.
\end{proof}

Now, we are able to prove Theorem \ref{euthm} in the case $\mathfrak p=3$:

\begin{proof}[Proof of Theorem \ref{euthm}.]
By Lemma \ref{dh} and due to the arbitrariness of $k$, one can derive that
$$
|c(t,n)-b(t,n)| \rightarrow 0, \quad k \rightarrow \infty
$$
if $0 < t < \frac{1}{4e\Box^2(12\rho^{-1})^{2\nu+1} |\omega|}$. Hence $c(t,n) \equiv b(t,n)$ for all
$$
0 \leq t < \min \left\{ t_0,\frac{1}{4e\Box^2(12\rho^{-1})^{2\nu+1}|\omega|}\right\}
$$
and $x\in\mathbb R$, that is,
\[
u_1(t,x) \equiv u_2(t,x), \quad \forall 0 \leq t < \min \left\{ t_0, \frac{1}{4e\Box^2(12\rho^{-1})^{2\nu+1} |\omega|} \right\}~\text{and}~x \in \mathbb R.
\]
This completes the proof of Theorem \ref{euthm} in the case $\mathfrak p=3$.
\end{proof}

\part{$\mathfrak p$-gKdV}

In this part we consider the case of general $\mathfrak p$ and prove Theorems~\ref{eethm} and \ref{euthm} in the formulation given in the introduction.

\section{Preliminaries}

Consider the $\mathfrak p$-gKdV \eqref{ppkdv} with spatially quasi-periodic initial data \eqref{iei}. The case $\mathfrak p=2$ was studied by Damanik and Goldstein in \cite{damanik16}, and the case of $\mathfrak p=3$ was discussed above.  Now, we give a similar local result for all $\mathfrak p \geq 2$, and in particular for $\mathfrak p \geq 4$.

If the wave vector $\omega$ is non-resonant, then $\mathfrak p$-gKdV \eqref{ppkdv} can be regarded as an infinite system of coupled nonlinear ordinary differential equations,
\begin{align}\label{ded}
\frac{{\rm d}}{{\rm d}t}c(t,n)-{\rm i}(n\cdot\omega)^3c(t,n)+
\frac{{\rm i}n\cdot\omega }{\mathfrak p}c^{\ast\mathfrak p}(t,n)=0, \quad \forall n\in\mathbb Z^\nu,
\end{align}
where $c^{\ast\mathfrak p}$ stands for the Cauchy product of infinite series, that is, the operation of discrete convolution, defined by letting
\[
c^{{\ast}{\mathfrak p}}(t,n)\triangleq\underbrace{c\ast c\cdots\ast c}_{\mathfrak p}~(t,n):=\sum_{\substack{n_1,\cdots,n_{\mathfrak p}\in\mathbb Z^\nu:~ \sum_{j=1}^{{\mathfrak p}}n_j=n}}\prod_{j=1}^{\mathfrak p}c(t,n_j).
\]
Obviously, $c(t,0)=c(0)$. Hence we need to solve $c(t,n)$ for all $n\in\mathbb Z^\nu\backslash\{0\}$. Under spatially quasi-periodic initial data \eqref{iei}, the ODE \eqref{ded} is equivalent to the following nonlinear integral equation,
\begin{align}
c(t,n)=e^{{\rm i}(n\cdot\omega)^3t}c(n)-\frac{{\rm i}n\cdot\omega}{\mathfrak p}\int_0^te^{{\rm i}(n\cdot\omega)^3(t-\tau)}c^{\ast\mathfrak p}(\tau,n){\rm d}\tau.
\end{align}
According to the feedback of nonlinearity, we define the following Picard iteration
\begin{align}
c_k(t,n):=
\begin{cases}
e^{{\rm i}(n\cdot\omega)^3t}c(n),&k=0;\\
\label{ppp}c_0(t,n)-\frac{{\rm i}n\cdot\omega}{\mathfrak p}\int_{0}^{t}e^{{\rm i}(n\cdot\omega)^3(t-\tau)}c_{k-1}^{\ast\mathfrak p}(\tau,n){\rm d}\tau,& k\in\mathring{\mathbb N}\triangleq\mathbb N\backslash\{0\}
\end{cases}
\end{align}
to solve for these Fourier coefficients.

Set
\begin{align*}
\mathfrak T^{(1)}&:=\{0,1\},\quad\mathfrak T^{(k)}:=\{0\}\cup\prod_{j=1}^{\mathfrak p}\mathfrak T^{(k-1)},\quad \text{for all}~~k\in\mathring{\mathbb N},\\
\mathfrak N^{(k,\gamma)}&:=
\begin{cases}
\mathbb Z^\nu, &\gamma=0\in\mathfrak T^{(k)}, k\geq1;\\
\prod_{j=1}^{\mathfrak p}\mathbb Z^\nu, &\gamma=1\in\mathfrak T^{(1)};\\
\prod_{j=1}^{\mathfrak p}\mathfrak N^{(k-1,\gamma_j)}, &\gamma=(\gamma_j)_{1\leq j\leq\mathfrak p}\in\prod_{j=1}^{\mathfrak p}\mathfrak T^{(k-1)}, k\geq2,
\end{cases}\\
\mathfrak F^{(k,\gamma)}(n^{(k)})&:=
\begin{cases}
1, &\gamma=0\in\mathfrak T^{(k)}, n^{(k)}\in\mathfrak N^{(k,0)}, k\geq 1;\\
-\frac{{\rm i}\mu(n^{(1)})\cdot\omega}{\mathfrak p}, &\gamma=1\in\mathfrak T^{(1)}, n^{(1)}\in\mathfrak N^{(1,1)};\\
-\frac{{\rm i}\mu(n^{(k)})\cdot\omega}{\mathfrak p}\prod_{j=1}^{\mathfrak p}\mathfrak F^{(k-1,\gamma_j)}(n_j), &\gamma=(\gamma_j)_{1\leq j\leq \mathfrak p}\in\prod_{j=1}^{\mathfrak p}\mathfrak T^{(k-1)},\\
&n^{(k)}=(n_j)_{1\leq j\leq \mathfrak p}\in\mathfrak N^{(k,\gamma)}, k\geq 2,
\end{cases}\\
\mathfrak I^{(k,\gamma)}(t,n^{(k)})&:=
\begin{cases}
e^{i\left(\mu(n^{(k)})\cdot\omega\right)^3t}, &\gamma=0\in\mathfrak T^{(k)}, n^{(k)}\in\mathfrak N^{(k,0)},\\
&k\geq 1;\\
\int_0^te^{{\rm i}\left(\mu(n^{(1)})\cdot\omega\right)^3(t-\tau)}\prod_{j=1}^{\mathfrak p}e^{{\rm i}(n_j\cdot\omega)^{3}\tau}{\rm d}\tau, &\gamma=1\in\mathfrak T^{(1)},\\
&n^{(1)}=(n_j)_{1\leq j\leq{\mathfrak p}}\in\mathfrak N^{(1,1)};\\
\int_0^te^{{\rm i}\left(\mu(n^{(k)})\cdot\omega\right)^{3}(t-\tau)}\prod_{j=1}^{\mathfrak p}\mathfrak I^{(k-1,\gamma_j)}(\tau,n_j){\rm d}\tau, &\gamma=(\gamma_j)_{1\leq j\leq{\mathfrak p}}\in\prod_{j=1}^{\mathfrak p}\mathfrak T^{(k-1)},\\
&n^{(k)}=(n_j)_{1\leq j\leq{\mathfrak p}}\in\mathfrak N^{(k,\gamma)}, \\
&k\geq 2,
\end{cases}\\
\mathfrak C^{(k,\gamma)}(t,n^{(k)})&:=
\begin{cases}
c(n^{(k)}), &\gamma=0\in\mathfrak T^{(k)}, n^{(k)}\in\mathfrak N^{(k,0)},k\geq 1;\\
\prod_{j=1}^{\mathfrak p}c(n_j), &\gamma=1\in\mathfrak T^{(1)}, n^{(1)}=(n_j)_{1\leq j\leq{\mathfrak p}}\in\mathfrak N^{(1,1)};\\
\prod_{j=1}^{\mathfrak p}\mathfrak C^{(k-1,\gamma_j)}(n_j), &\gamma=(\gamma_j)_{1\leq j\leq{\mathfrak p}}\in\prod_{j=1}^{\mathfrak p}\mathfrak T^{(k-1)},\\
&n^{(k)}=(n_j)_{1\leq j\leq{\mathfrak p}}\in\mathfrak N^{(k,\gamma)}, k\geq 2.
\end{cases}
\end{align*}
By induction we have
\begin{align}\label{te}
c_k(t,n)=\sum_{\gamma\in\mathfrak T^{(k)}}\sum_{n^{(k)}\in\mathfrak N^{(k,\gamma)}:~\mu(n^{(k)})=n}\mathfrak F^{(k,\gamma)}(n^{(k)})\mathfrak I^{(k)}(t,n^{(k)})\mathfrak C^{(k,\gamma)}(n^{(k)}).
\end{align}
In fact, it is obvious that \eqref{te} holds for $k=1$. Let $k\geq2$. Assume that it is true for $1\leq k^\prime\leq k-1$. For $k$, we have
\begin{align*}
c_{k}(t,n)=e^{{\rm i}(n\cdot\omega)^3t}c(n)-\frac{{\rm i}n\cdot\omega}{\mathfrak p}\int_0^te^{{\rm i}(n\cdot\omega)^3(t-\tau)}\sum_{n_1,\cdots,n_\mathfrak p\in\mathbb Z^\nu:~ \sum_{j=1}^{\mathfrak p}n_j=n}\prod_{j=1}^{\mathfrak p}c_{k-1}(\tau,n_j){\rm d}\tau,
\end{align*}
where
\begin{eqnarray*}
&&\prod_{j=1}^{\mathfrak p}c_{k-1}(\tau,n_j)\\
&=&\prod_{j=1}^{\mathfrak p}\sum_{\gamma_j\in\mathfrak T^{(k-1)}}\sum_{\substack{\clubsuit_j\in\mathfrak N^{(k-1,\gamma_j)}\\\mu(\clubsuit_j)=n_j}}\mathfrak F^{(k-1,\gamma_j)}(\clubsuit_j)\mathfrak I^{(k-1,\gamma_j)}(\tau,\clubsuit_j)\mathfrak C^{(k-1,\gamma_j)}(\clubsuit_j)\\
&=&\sum_{\substack{(\gamma_1,\cdots,\gamma_\mathfrak p)\in\\\prod_{j=1}^{\mathfrak p}\mathfrak T^{(k-1)}}}\sum_{\substack{n^{(k)}\triangleq(\clubsuit_1,\cdots,\clubsuit_\mathfrak p)\in\prod_{j=1}^{\mathfrak p}\mathfrak N^{(k-1,\gamma_j)}\\\mu(n^{(k)})=n_1+\cdots+n_\mathfrak p}}\prod_{j=1}^{\mathfrak p}\mathfrak F^{(k-1,\gamma_j)}(\clubsuit_j)\mathfrak I^{(k-1,\gamma_j)}(\tau,\clubsuit_j)\mathfrak C^{(k-1,\gamma_j)}(\clubsuit_j).
\end{eqnarray*}
Hence
\begin{eqnarray*}
c_k(t,n)&=&\sum_{n^{(k)}\in\mathfrak N^{(k,0)}:~\mu(n^{(k)})=n}\mathfrak F^{(k-1,0)}(n^{(k)})\mathfrak I^{(k-1,0)}(t,n^{(k)})\mathfrak C^{(k-1,0)}(n^{(k)})\\
&+&\sum_{\substack{(\gamma_1,\cdots,\gamma_\mathfrak p)\in\\\prod_{j=1}^{\mathfrak p}\mathfrak T^{(k-1)}}}\sum_{\substack{n^{(k)}\triangleq(\clubsuit_1,\cdots,\clubsuit_\mathfrak p)\in\prod_{j=1}^{\mathfrak p}\mathfrak N^{(k-1,\gamma_j)}\\\mu(n^{(k)})=n}}
\int_0^te^{{\rm i}\left(\mu(n^{(k)}\cdot\omega\right)^3(t-\tau)}\prod_{j=1}^{\mathfrak p}\mathfrak I^{(k-1,\gamma_j)}(t,\clubsuit_j){\rm d}\tau\\
&\times&\left\{-\frac{{\rm i}\mu(n^{(k)})\cdot\omega}{\mathfrak p}\right\}\prod_{j=1}^{\mathfrak p}\mathfrak F^{(k-1,\gamma_j)}(\clubsuit_j)\times \prod_{j=1}^{\mathfrak p}\mathfrak C^{(k-1,\gamma_j)}(\clubsuit_j)\\
&=&\sum_{\gamma\in\mathfrak T^{(k)}}\sum_{n^{(k)}\in\mathfrak N^{(k,\gamma)}:~\mu(n^{(k)})=n}\mathfrak F^{(k,\gamma)}(n^{(k)})\mathfrak I^{(k)}(t,n^{(k)})\mathfrak C^{(k,\gamma)}(n^{(k)}).
\end{eqnarray*}
This shows that $c_k(t,n)$ can be viewed as an infinite series
in which the summation index is the branch $\gamma \in \mathfrak T^{(k)}$ and runs over the tree $\mathfrak T^{(k)}$. On each branch $\gamma \in \mathfrak T^{(k)}$, it splits over the following restriction: total distance $=n$. With the help of this tree form, we can deal with the operation of discrete convolution $c^{\ast\mathfrak p}(n)$.

\section{Properties of the Picard Sequence}

In this section, we prove that the Picard sequence is exponentially decaying and fundamental (i.e., a Cauchy sequence).

\subsection{Exponential Decay}

\begin{lemm}\label{gfhj}
If
\[
|c(n)|\leq\mathcal A^{\frac{1}{\mathfrak p-1}}e^{-\kappa|n|},
\]
then
\begin{align}\label{ce}
|\mathfrak C^{(k,\gamma)}(n^{(k)})|\leq\mathcal A^{\sigma(\gamma)}e^{-\kappa|n^{(k)}|}, \quad\forall k\geq1,
\end{align}
where
\begin{align*}
\sigma(\gamma):=
\begin{cases}
\frac{1}{\mathfrak p-1},&\gamma=0\in\mathfrak T^{(k)},k\geq1;\\
\frac{\mathfrak p}{\mathfrak p-1},&\gamma=1\in\mathfrak T^{(1)};\\
\sum_{j=1}^{\mathfrak p}\sigma(\gamma_j),&\gamma=(\gamma_j)_{1\leq j\leq\mathfrak p}\in\prod_{j=1}^{\mathfrak p}\mathfrak T^{(k-1)}, k\geq2.
\end{cases}
\end{align*}
\end{lemm}
\begin{proof}
Obviously, \eqref{ce} is true for $k=1$.

Let $k\geq2$. Assume that \eqref{ce} holds for $1\leq k^\prime\leq k-1$. For $k$, it is clear that \eqref{ce} holds for $\gamma = 0 \in \mathfrak T^{(k)}$. For $\gamma=(\gamma_j)_{1\leq j\leq\mathfrak p}\in\prod_{j=1}^{\mathfrak p}\mathfrak T^{(k-1)}, n^{(k)}=(n_j)_{1\leq j\leq\mathfrak p}\in\prod_{j=1}^{\mathfrak p}\mathfrak T^{(k-1)}$, we have
\begin{eqnarray*}
|\mathfrak C^{(k,\gamma)}(n^{(k)})|&=&\prod_{j=1}^{\mathfrak p}|\mathfrak C^{(k-1,\gamma_j)}(n_j)|\\
&\leq&\prod_{j=1}^{\mathfrak p}\mathcal A^{\sigma(\gamma_j)}e^{-\kappa|n_j|}\\
&=&\mathcal A^{\sum_{j=1}^{\mathfrak p}\sigma(\gamma_j)}e^{-\kappa\sum_{j=1}^{\mathfrak p}|n_j|}\\
&=&\mathcal A^{\sigma(\gamma)}e^{-\kappa|n^{(k)}|}.
\end{eqnarray*}
The proof of Lemma \ref{gfhj} is completed.
\end{proof}

\begin{rema}
By the definition of $\mathfrak N^{(k,\gamma)}$, we have
\begin{align}\label{je}
n\in\mathfrak N^{(k,\gamma)}&\Rightarrow n\in(\mathbb Z^\nu)^{(\mathfrak p-1)\sigma(\gamma)}, \quad \forall k\geq1.
\end{align}

In fact, it is clear that \eqref{je} is true for $k=1$.

Let $k\geq2$. Assume that it is true for $1\leq k^\prime\leq k-1$. For $k$, the case of $\gamma = 0 \in \mathfrak T^{(k)}$ is trivial. For $\gamma = (\gamma_j)_{1 \leq j \leq \mathfrak p} \in \prod_{j=1}^{\mathfrak p} \mathfrak T^{(k-1)}$, $n^{(k)}=(n_j)_{1\leq j\leq\mathfrak p}\in\prod_{j=1}^{\mathfrak p}\mathfrak N^{(k-1,\gamma_j)}$, by assumption, we know that $n_j\in(\mathbb Z^\nu)^{(\mathfrak p-1)\sigma(\gamma_j)}$ for all $j=1,\cdots,\mathfrak p$. Hence $n^{(k)}\in(\mathbb Z^\nu)^{(\mathfrak p-1)\sum_{j=1}^{\mathfrak p}\sigma(\gamma_j)}\simeq(\mathbb Z^\nu)^{(\mathfrak p-1)\sigma(\gamma)}$.
\end{rema}

\begin{lemm}\label{lsdfg}
We have
\begin{align}\label{cse}
|\mathfrak I^{(k,\gamma)}(t,n^{(k)})|\leq\frac{t^{\ell(\gamma)}}{\mathfrak D(\gamma)}, \quad \forall k\geq1,
\end{align}
where
\begin{align*}
\ell(\gamma)&:=
\begin{cases}
0&\gamma=0\in\mathfrak T^{(k)},k\geq1;\\
1&\gamma=1\in\mathfrak T^{(1)};\\
\sum_{j=1}^{\mathfrak p}\ell(\gamma_j)+1,&\gamma=(\gamma_j)_{1\leq j\leq\mathfrak p}\in\prod_{j=1}^{\mathfrak p}\mathfrak T^{(k-1)}, k\geq2,
\end{cases}\\
\mathfrak D(\gamma)&:=
\begin{cases}
1,&\gamma=0\in\mathfrak T^{(k)},k\geq1;\\
1,&\gamma=1\in\mathfrak T^{(1)};\\
\ell(\gamma)\prod_{j=1}^{\mathfrak p}\mathfrak D(\gamma_j),&\gamma=(\gamma_j)_{1\leq j\leq\mathfrak p}\in\prod_{j=1}^{\mathfrak p}\mathfrak T^{(k-1)}, k\geq 2.
\end{cases}
\end{align*}
\end{lemm}

\begin{proof}
Obviously, \eqref{cse} is true for $k=1$.

Let $k\geq2$. Assume that \eqref{cse} holds for $1\leq k^\prime\leq k-1$. For $k$, it is clear that \eqref{ce} holds for $\gamma = 0 \in \mathfrak T^{(k)}$. For $\gamma = (\gamma_j)_{1 \leq j \leq \mathfrak p} \in \prod_{j=1}^{\mathfrak p} \mathfrak T^{(k-1)}, n^{(k)} = (n_j)_{1 \leq j \leq \mathfrak p} \in \prod_{j=1}^{\mathfrak p}\mathfrak T^{(k-1)}$, we have
\begin{eqnarray*}
|\mathfrak I^{(k,\gamma)}(t,n^{(k)})|&\leq&\prod_{j=1}^{\mathfrak p}|\mathfrak I^{(k-1,\gamma_j)}(\tau,n_j)|\\
&\leq&\int_0^t\prod_{j=1}^{\mathfrak p}\frac{\tau^{\ell(\gamma_j)}}{\mathfrak D(\gamma_j)}{\rm d}\tau\\
&=&\frac{t^{\ell(\gamma)}}{\mathfrak D(\gamma)}.
\end{eqnarray*}
This completes the proof of Lemma \ref{lsdfg}.
\end{proof}

\begin{rema}
By the definition of $\sigma$, $\ell$, and $\mathfrak N^{(k,\gamma)}$, we have
\begin{align}\label{see}
\sigma(\gamma)=\ell(\gamma)+\frac{1}{\mathfrak p-1}, \quad \gamma \in \mathfrak T^{(k)}, \quad\forall k\geq1.
\end{align}

Indeed, it is clear that \eqref{see} is true for $k=1$.

Let $k\geq2$. Assume that it is true for $1\leq k^\prime\leq k-1$. For $k$, the case of $\gamma = 0 \in \mathfrak T^{(k)}$ is trivial. For $\gamma=(\gamma_j)_{1\leq j\leq\mathfrak p}\in\prod_{j=1}^{\mathfrak p}\mathfrak T^{(k-1)}$,
\begin{eqnarray*}
\sigma(\gamma)&=&\sum_{j=1}^{\mathfrak p}\sigma(\gamma_j)\\
&=&\sum_{j=1}^{\mathfrak p}\left(\ell(\gamma_j)+\frac{1}{\mathfrak p-1}\right)\\
&=&\sum_{j=1}^{\mathfrak p}\ell(\gamma_j)+1+\frac{1}{\mathfrak p-1}\\
&=&\ell(\gamma)+\frac{1}{\mathfrak p-1}.
\end{eqnarray*}
\end{rema}

\begin{lemm}\label{245}
For all $k\geq1$, we have
\begin{eqnarray}
|\mathfrak F^{(k,\gamma)}(n^{(k)})|&\leq&|\omega|^{\ell(\gamma)}\mathfrak B(n^{(k)})\label{cseb}\\
&\leq&|\omega|^{\ell(\gamma)}\sum_{\alpha=(\alpha_j)_{1\leq j\leq(\mathfrak p-1)\sigma(\gamma)}\in\mathfrak R^{(k,\gamma)}}\prod_{j}|(n^{(k)})_j|^{\alpha_j},\label{csea}
\end{eqnarray}
where
\begin{align*}
\mathfrak B(n^{(k)})&:=
\begin{cases}
1&\gamma=0\in\mathfrak T^{(k)},k\geq1;\\
|\mu(n^{(1)})|&\gamma=1\in\mathfrak T^{(1)};\\
|\mu(n^{(k)})|\prod_{j=1}^{\mathfrak p}\mathfrak B(n_j),&\gamma=(\gamma_j)_{1\leq j\leq\mathfrak p}\in\prod_{j=1}^{\mathfrak p}\mathfrak T^{(k-1)}, n^{(k)}=(n_j)_{1\leq j\leq\mathfrak p}\in\mathfrak N^{(k,\gamma)}, \\
&k\geq 2.
\end{cases}\\
\mathfrak R^{(k,\gamma)}&:=
\begin{cases}
\{0\in\mathbb Z\}, &\gamma=0\in\mathfrak T^{(k)}, k\geq1;\\
\{(\underbrace{1,0,\cdots,0}_{\mathfrak p}),\cdots,(\underbrace{0,0,\cdots,1}_{\mathfrak p})\}, &\gamma=1\in\mathfrak T^{(1)};\\
\prod_{j=1}^{\mathfrak p}\mathfrak R^{(k-1,\gamma_j)}+e^{(k,\gamma)}, &\gamma=(\gamma_j)_{1\leq j\leq\mathfrak p}\in\prod_{j=1}^{\mathfrak p}\mathfrak N^{(k-1,\gamma_j)},
\end{cases}
\end{align*}
and
\[e^{(k,\gamma)}:=\{\alpha\in\mathbb Z^{(\mathfrak p-1)\sigma(\gamma)}: |\alpha|=1,\alpha_j\geq0\}.\]
\end{lemm}

\begin{proof}
Clearly, \eqref{cseb} and \eqref{csea} are true for $k=1$.

Let $k\geq2$. Assume that they hold for $1\leq k^\prime\leq k-1$. For $k$, $\gamma = (\gamma_j)_{1 \leq j \leq \mathfrak p} \in \prod_{j=1}^{\mathfrak p} \mathfrak T^{(k-1)}, n^{(k)} = (\clubsuit_j)_{1\leq j\leq\mathfrak p} \in
\prod_{j=1}^{\mathfrak p} \mathfrak N^{(k-1,\gamma_j)} \subset \prod_{j=1}^{(\mathfrak p-1) \sigma(\gamma_j)} \mathbb Z^\nu \simeq \prod_{j=1}^{(\mathfrak p-1)\sigma(\gamma)} \mathbb Z^\nu$, where $\clubsuit_1:=(n_1,\cdots,n_{(\mathfrak p-1)\sigma(\gamma_1)})$, and
%\clubsuit_2:=(n_{(\mathfrak p-1)\sigma(\gamma_1)+1},\cdots,n_{(\mathfrak p-1)\sigma(\gamma_1)+(\mathfrak p-1)\sigma(\gamma_2)}),\cdots,
%$
$$
\clubsuit_j:=(n_{\sum_{i=1}^{j-1}(\mathfrak p-1)\sigma(\gamma_{i})+1},\cdots,n_{\sum_{i=1}^{j-1}(\mathfrak p-1)\sigma(\gamma_{i})+(\mathfrak p-1)\sigma(\gamma_j)}=n_{\sum_{i=1}^{j}(\mathfrak p-1)\sigma(\gamma_{i})})$$ for $j=2,3\cdots,\mathfrak p$, that is,
\begin{align*}
(n^{(k)})_i=
\begin{cases}
\clubsuit_i,&1\leq i\leq(\mathfrak p-1)\sigma(\gamma_1);\\
\spadesuit_{i-\sum_{i=1}^{j-1}(\mathfrak p-1)\sigma(\gamma_{i})},&\sum_{i=1}^{j-1}(\mathfrak p-1)\sigma(\gamma_{i})+1\leq i\leq\sum_{i=1}^{j}(\mathfrak p-1)\sigma(\gamma_{i}), ~j=2,\cdots,\mathfrak p.
\end{cases}
\end{align*}
Hence we have
\begin{eqnarray*}
|\mathfrak F^{(k,\gamma)}(n^{(k)})|&\leq&|\omega||\mu(n^{(k)})|\prod_{j=1}^{\mathfrak p}|\omega|^{\ell(\gamma_j)}\mathfrak B(\clubsuit_j)\\
&=&|\omega|^{\ell(\gamma)}|\mu(n^{(k)})|\prod_{j=1}^{\mathfrak p}\mathfrak B(\clubsuit_j)\\
&=&|\omega|^{\ell(\gamma)}\mathfrak B(n^{(k)}).
\end{eqnarray*}
Notice that
\begin{eqnarray*}
&&|\mu(n^{(k)})|\prod_{j=1}^{\mathfrak p}\mathfrak B(\clubsuit_j)\\
&\leq&\sum_{j=1}^{(\mathfrak p-1)\sigma(\gamma)}|n_j|\prod_{j=1}^{\mathfrak p}
\sum_{\alpha^{(j)}=(\alpha^{(j)}_i)_{1\leq i\leq(\mathfrak p-1)\sigma(\gamma_j)}\in\mathfrak R^{(k-1,\gamma_j)}}\prod_{i=1}^{(\mathfrak p-1)\sigma(\gamma_j)}|\clubsuit_i|^{\alpha^{(j)}_i}\\
&=&\sum_{\alpha=(\alpha_j)_{1\leq j\leq(\mathfrak p-1)\sigma(\gamma)}\in e^{(k,\gamma)}}\prod_{j=1}^{(\mathfrak p-1)\sigma(\gamma)}|(n^{(k)})_j|^{\alpha_j}
\sum_{
\alpha=(\alpha^{(j)})_{1\leq j\leq\mathfrak p}\in\prod_{j=1}^{\mathfrak p}\mathfrak R^{(k-1,\gamma_j)}
}\prod_{j=1}^{\mathfrak p}\prod_{i^{(j)}=1}^{(\mathfrak p-1)\sigma(\gamma_j)}||\clubsuit_{i^{(j)}}|^{\alpha_{i^{(j)}}^{(j)}}\\
&=&\sum_{\alpha=(\alpha_j)_{1\leq j\leq(\mathfrak p-1)\sigma(\gamma)}\in e^{(k,\gamma)}}\prod_{j}|(n^{(k)})_j|^{\alpha_j}
\sum_{
\alpha=(\alpha_j)_{1\leq j\leq(\mathfrak p-1)\sigma(\gamma)}\in\prod_{j=1}^{\mathfrak p}\mathfrak R^{(k-1,\gamma_j)}
}\prod_{j}|(n^{(k)})_j|^{\alpha_j}\\
&=&\sum_{\alpha=(\alpha_j)_{1\leq j\leq(\mathfrak p-1)\sigma(\gamma)}\in\prod_{j=1}^{\mathfrak p}\mathfrak R^{(k-1,\gamma_j)}+e^{(k,\gamma)}}|(n^{(k)})_j|^{\alpha_j}\\
&=&\sum_{\alpha=(\alpha_j)_{1\leq j\leq(\mathfrak p-1)\sigma(\gamma)}\in\mathfrak R^{(k-1,\gamma)}}|(n^{(k)})_j|^{\alpha_j}.
\end{eqnarray*}
Hence
\begin{align*}
|\mathfrak F^{(k,\gamma)}(n^{(k)})|\leq|\omega|^{\ell(\gamma)}\sum_{\alpha=(\alpha_j)_{1\leq j\leq(\mathfrak p-1)\sigma(\gamma)}\in\mathfrak R^{(k-1,\gamma)}}\prod_j|(n^{(k)})_j|^{\alpha_j}.
\end{align*}
The proof of Lemma \ref{245} is completed.
\end{proof}

\begin{rema}
By the definition of $\mathfrak R^{(k,\gamma)}$, we have
\begin{align}\label{ccc}
\alpha\in\mathfrak R^{(k,\gamma)}\Rightarrow |\alpha|=\ell(\gamma),\quad\forall k\geq1.
\end{align}

Indeed, it is clear that \eqref{ccc} is true for $k=1$.

Let $k\geq2$. Assume that it is true for $1\leq k^\prime\leq k-1$. For $k$, the case of $\gamma=0\in\mathfrak T^{(k)}$ is trivial. For $\gamma=(\gamma_j)_{1\leq j\leq\mathfrak p}\in\prod_{j=1}^{\mathfrak p}\mathfrak T^{(k-1)}$, there exist $(\alpha_j)\in\prod_{j=1}^{\mathfrak p}\mathfrak R^{(k-1,\gamma_j)}$ and $\beta\in e^{(k,\gamma)}$ such that $\alpha=(\alpha_j)_{1\leq j\leq\mathfrak p}+\beta$. Hence
\begin{eqnarray*}
|\alpha|&=&\sum_{j=1}^{\mathfrak p}|\alpha_j|+|\beta|\\
&=&\sum_{j=1}^{\mathfrak p}\ell(\gamma_j)+1\\
&=&\ell(\gamma).
\end{eqnarray*}
\end{rema}

\begin{lemm}\label{lmm}
For $0 \leq {\bf t} \leq \min \left\{ \frac{1}{\mathfrak p}, \frac{1}{2^{\mathfrak p+1}}\right\}=\frac{1}{2^{\mathfrak p+1}}$, we have
\begin{align}\label{10}
\sum_{\gamma\in\mathfrak T^{(k)}}\frac{{\bf t}^{\ell(\gamma)}}{\mathfrak D(\gamma)}\sum_{
\alpha=(\alpha_j)_{1\leq j\leq(\mathfrak p-1)\sigma(\gamma)}\in\mathfrak R^{(k,\gamma)}
}\prod_{j=1}^{(\mathfrak p-1)\sigma(\gamma)}\alpha_j!\leq2,\quad\forall k\geq1.
\end{align}
\end{lemm}

\begin{proof}
For $k=1$,
\begin{eqnarray*}
&&\sum_{\gamma\in\mathfrak T^{(1)}}\frac{{\bf t}^{\ell(\gamma)}}{\mathfrak D(\gamma)}\sum_{
\alpha=(\alpha_i)_{1\leq i\leq(\mathfrak p-1)\sigma(\gamma)}\in\mathfrak R^{(1,\gamma)}
}\prod_i\alpha_i!\\
&=&\frac{{\bf t}^{\ell(0)}}{\mathfrak D(0)}\sum_{\alpha=(\alpha_i)_{1\leq i\leq(\mathfrak p-1)\sigma(0)}\in\mathfrak R^{(1,0)}}\prod_i\alpha_i!+\frac{{\bf t}^{\ell(1)}}{\mathfrak D(1)}\sum_{
\alpha=(\alpha_i)_{1\leq i\leq(\mathfrak p-1)\sigma(1)}\in\mathfrak R^{(1,1)}
}\prod_i\alpha_i!\\
&=&1+\mathfrak pt\\
&\leq&2,
\end{eqnarray*}
provided that $0 \leq t \leq \frac{1}{\mathfrak p}$.

Let $k\geq2$. Assume that \eqref{10} is true for all $1\leq k^\prime\leq k-1$. For $k$, $\gamma = 0 \in \mathfrak T^{(k)}$, it is the same as above. For $\gamma=(\gamma_j)_{1\leq j\leq\mathfrak p}\in\prod_{j=1}^{\mathfrak p}\mathfrak T^{(k-1)}$, $\alpha=(\alpha^{(j)})_{1\leq j\leq(\mathfrak p-1)\sigma(\gamma)}+\beta$, where $\alpha^{(j)}\in\mathfrak R^{(k-1,\gamma_j)},\beta\in e^{(k,\gamma)}$.
Hence, we can derive that
\begin{eqnarray*}
&&\sum_{\gamma\in\mathfrak T^{(k)}\backslash\{0\}}\frac{{\bf t}^{\ell(\gamma)}}{\mathfrak D(\gamma)}\sum_{
\alpha=(\alpha_i)_{1\leq i\leq2\sigma(\gamma)}\in\mathfrak R^{(k,\gamma)}
}\prod_i\alpha_i!\\
&=&\sum_{\substack{\gamma_j\in\mathfrak T^{(k-1)}\\j=1,\cdots,\mathfrak p}}\frac{{\bf t}}{\sum_{j=1}^\mathfrak p\ell(\gamma_j)+1}\prod_{j=1}^\mathfrak p\frac{{\bf t}^{\ell(\gamma_j)}}{\mathfrak D(\gamma_j)}\sum_{j_0=1}^{\mathfrak p}\sum_{i_0=1}^{\sigma(\gamma_{j_0})}
\left((\alpha^{(j_0)})_{i_0}+1\right)\prod_{j=1}^{\mathfrak p}
\prod_{j=1}^{\sigma(\gamma_j)
}\left((\alpha^{(j)})_i\right)!\\
&=&\sum_{\substack{\gamma_j\in\mathfrak T^{(k-1)}\\j=1,\cdots,\mathfrak p}}\frac{{\bf t}}{\sum_{j=1}^\mathfrak p\ell(\gamma_j)+1}\prod_{j=1}^\mathfrak p\frac{{\bf t}^{\ell(\gamma_j)}}{\mathfrak D(\gamma_j)}\sum_{j_0=1}^{\mathfrak p}\left(\ell(\gamma_{j_0})+\sigma(\gamma_{j_0})\right)
\prod_{j=1}^{\mathfrak p}
\prod_{j=1}^{\sigma(\gamma_j)
}\left((\alpha^{(j)})_i\right)!\\
&=&\sum_{\substack{\gamma_j\in\mathfrak T^{(k-1)}\\j=1,\cdots,\mathfrak p}}\frac{{\bf t}}{\sum_{j=1}^\mathfrak p\ell(\gamma_j)+1}\prod_{j=1}^\mathfrak p\frac{{\bf t}^{\ell(\gamma_j)}}{\mathfrak D(\gamma_j)}\sum_{j_0=1}^{\mathfrak p}\left(2\ell(\gamma_{j_0})+\frac{1}{\mathfrak p-1}\right)\prod_{j=1}^\mathfrak p
\prod_{j=1}^{\sigma(\gamma_j)
}\left((\alpha^{(j)})_i\right)!\\
&=&{\bf t}\sum_{\substack{\gamma_j\in\mathfrak T^{(k-1)}\\j=1,\cdots,\mathfrak p}}\frac{\sum_{j_0=1}^\mathfrak p\left(2\ell(\gamma_{j_0})+\frac{1}{\mathfrak p-1}\right)}
{\sum_{j=1}^3\ell(\gamma_j)+1}\sum_{\substack{\alpha^{(j)}\in\mathfrak R^{(k-1,\gamma_j)}\\j=1,\cdots,\mathfrak p}}\prod_{j=1}^\mathfrak p\frac{{\bf t}^{\ell(\gamma_j)}}{\mathfrak D(\gamma_j)}\prod_{i=1}^{\sigma(\gamma_j)}
(\alpha^{(j)})_i!\\
&=&{\bf t}\sum_{\substack{\gamma_j\in\mathfrak T^{(k-1)}\\j=1,\cdots,\mathfrak p}}\frac{2\left(\sum_{j_0=1}^\mathfrak p\ell(\gamma_{j_0})+1\right)-\frac{\mathfrak p-2}{\mathfrak p-1}}{\sum_{j=1}^\mathfrak p\ell(\gamma_j)+1}\sum_{\substack{\alpha^{(j)}\in\mathfrak R^{(k-1,\gamma_j)}\\j=1,\cdots,\mathfrak p}}\prod_{j=1}^\mathfrak p\frac{{\bf t}^{\ell(\gamma_j)}}{\mathfrak D(\gamma_j)}\prod_{i=1}^{\sigma(\gamma_j)}
(\alpha^{(j)})_i!\\
&\leq&2{\bf t}\sum_{\substack{\gamma_j\in\mathfrak T^{(k-1)}\\j=1,\cdots,\mathfrak p}}\sum_{\substack{\alpha^{(j)}\in\mathfrak R^{(k-1,\gamma_j)}\\j=1,\cdots,\mathfrak p}}\prod_{j=1}^\mathfrak p\frac{{\bf t}^{\ell(\gamma_j)}}{\mathfrak D(\gamma_j)}\prod_{i=1}^{\sigma(\gamma_j)}
(\alpha^{(j)})_i!\\
&=&2{\bf t}\prod_{j=1}^\mathfrak p\sum_{\gamma_j\in\mathfrak T^{(k-1)}}\sum_{\alpha^{(j)}\in\mathfrak R^{(k-1,\gamma_j)}}\frac{{\bf t}^{\ell(\gamma_j)}}{\mathfrak D(\gamma_j)}\prod_{i=1}^{\sigma(\gamma_j)}(\alpha^{(j)})_i!\\
&\leq&2{\bf t}\times2^\mathfrak p\\
&=&2^{\mathfrak p+1}{\bf t}.
\end{eqnarray*}
Hence
\begin{eqnarray*}
&&\sum_{\gamma\in\mathfrak T^{(k)}}\frac{{\bf t}^{\ell(\gamma)}}{\mathfrak D(\gamma)}\sum_{
\alpha=(\alpha_i)_{1\leq i\leq({\mathfrak p}-1)\sigma(\gamma)}\in\mathfrak R^{(k,\gamma)}
}\prod_i\alpha_i!\\
&=&\sum_{\gamma=0\in\mathfrak T^{(k)}}\frac{{\bf t}^{\ell(\gamma)}}{\mathfrak D(\gamma)}\sum_{
\alpha=(\alpha_i)_{1\leq i\leq({\mathfrak p}-1)\sigma(\gamma)\sigma(\gamma)}\in\mathfrak R^{(k,\gamma)}
}\prod_j\alpha_j!\\
&+&\sum_{\gamma\in\mathfrak T^{(k)}\backslash\{0\}}\frac{{\bf t}^{\ell(\gamma)}}{\mathfrak D(\gamma)}\sum_{
\alpha=(\alpha_i)_{1\leq i\leq({\mathfrak p}-1)\sigma(\gamma)}\in\mathfrak R^{(k,\gamma)}
}\prod_i\alpha_i!\\
&\leq&1+2^{\mathfrak p+1}{\bf t}\\
&\leq&2,
\end{eqnarray*}
provided that $0\leq{\bf t}\leq\frac{1}{2^{\mathfrak p+1}}$. Thus, we have seen that if $0\leq{\bf t}\leq\min\left\{\frac{1}{\mathfrak p},\frac{1}{2^{\mathfrak p+1}}\right\}=\frac{1}{2^{\mathfrak p+1}}$, then \eqref{10} holds true. This completes the proof of Lemma \ref{lmm}.
\end{proof}

\begin{theo}\label{expe}
If $0<t\leq \frac{\kappa^{(\mathfrak p-1)\nu+1}}{2^{\mathfrak p+1}6^{(\mathfrak p-1)\nu+1}\mathcal A|\omega|}$, then
\begin{align}
|c_k(t,n)|\leq \Box e^{-\frac{\kappa}{2}|n|}, \quad \forall k\geq1,
\end{align}
where $\Box\triangleq2(6\kappa^{-1})^{\nu}\mathcal A^{\frac{1}{\mathfrak p-1}}$.
\end{theo}

\begin{proof}
We have
\begin{eqnarray*}
&&|c_k(t,n)|\\
&\stackrel{\eqref{te}}{\leq}&\sum_{\gamma\in\mathfrak T^{(k)}}\sum_{n^{(k)}\in\mathfrak N^{(k,\gamma)}:\mu(n^{(k)})=n}
|\mathfrak F^{(k,\gamma)}(n^{(k)})||\mathfrak I^{(k,\gamma)}(t,n^{(k)})||\mathfrak C^{(k,\gamma)}(n^{(k)})|\\
&\stackrel{\eqref{csea},\eqref{cse},\eqref{ce},\eqref{see}}{\leq}&\mathcal A^{\frac{1}{\mathfrak p-1}}
\sum_{\gamma\in\mathfrak T^{(k)}}\frac{(\mathcal A|\omega|t)^{\ell(\gamma)}}{\mathfrak D(\gamma)}
\sum_{\alpha=(\alpha_j)_{1\leq j\leq(\mathfrak p-1)\sigma(\gamma)}\in\mathfrak R^{(k,\gamma)}}\\
&&\sum_{\substack{n^{(k)}\in\mathfrak N^{(k,\gamma)}\\ \mu(n^{(k)})=n}}\prod_{j}|(n^{(k)})_j|^{\alpha_j}e^{-\kappa|n^{(k)}|}\\
&\stackrel{\text{Lemma \ref{4}},\eqref{see},\eqref{je},\eqref{ccc}, \text{Lemma \ref{lmm}}}{\leq}&(6\kappa^{-1})^{\nu}\mathcal A^{\frac{1}{\mathfrak p-1}}e^{-\frac{\kappa}{2}|n|}
\sum_{\gamma\in\mathfrak T^{(k)}}\frac{\left((6\kappa^{-1})^{(\mathfrak p-1)\nu+1}\mathcal A|\omega|t\right)^{\ell(\gamma)}}{\mathfrak D(\gamma)}\\
&&\sum_{\alpha=(\alpha_j)_{1\leq j\leq(\mathfrak p-1)\sigma(\gamma)}\in\mathfrak R^{(k,\gamma)}}
\prod_{j=1}^{(\mathfrak p-1)\sigma(\gamma)}\alpha_j!\\
&\leq&\Box e^{-\frac{\kappa}{2}|n|},
\end{eqnarray*}
provided that $0<t\leq\frac{\kappa^{(\mathfrak p-1)\nu+1}}{2^{\mathfrak p+1}6^{(\mathfrak p-1)\nu+1}\mathcal A|\omega|}$. This completes the proof of Theorem \ref{expe}.
\end{proof}

\subsection{Cauchy Sequence}

\begin{lemm}\label{149}
For all $k\geq1$, we have
\begin{eqnarray}\label{cau}
&&|c_k(t,n)-c_{k-1}(t,n)|\nonumber\\
&\leq&\frac{\Box^{(\mathfrak p-1)k+1}(|\omega|t)^k}{k!}\sum_{\substack{(n_j)_{1\leq j\leq(\mathfrak p-1)k+1}\in(\mathbb Z^\nu)^{(\mathfrak p-1)k+1}\\ \sum_{j=1}^{(\mathfrak p-1)k+1}n_j=n}}\sum_{\alpha\in\mathbb B^{(k)}}\prod_{j=1}^{(\mathfrak p-1)k+1}|n_j|^{\alpha_j}e^{-\frac{\kappa}{2}|n_j|}\\
\nonumber&\leq&\Theta e^{-\frac{\kappa}{4}|n|}\frac{\left(\Box^{\mathfrak p-1}(12\kappa^{-1})^{(\mathfrak p-1)\nu+1}|\omega|t\right)^{k}}{k!}\sum_{\alpha\in\mathbb B^{(k)}}\prod_{j}\alpha_j, \quad\text{\rm where}~~\Theta=(12\kappa^{-1})^\nu\Box\\
\nonumber&\leq&\Theta e^{-\frac{\kappa}{4}|n|+\frac{1}{\mathfrak p-1}}\left(2(\mathfrak p-1)e\Box^{\mathfrak p-1}(12\kappa^{-1})^{(\mathfrak p-1)\nu+1}|\omega|t\right)^{k},
\end{eqnarray}
where
\begin{align*}
\mathbb B^{(k)}:=
\begin{cases}
\{(\underbrace{1,0,\cdots,0}_{\mathfrak p}),\cdots,(\underbrace{0,\cdots,0,1}_{\mathfrak p})\},&k=1;\\
\mathbb B^{(k-1)}\times\prod_{j=1}^{\mathfrak p-1}\{0\in\mathbb Z\}+\mathfrak g^{(k)},&k\geq2,
\end{cases}
\end{align*}
and
\[\mathfrak g^{(k)}:=\left\{\alpha\in\mathbb Z^{(\mathfrak p-1)k+1}: \sum_{j}\alpha_j=1,\alpha_j\geq0\right\}.\]
\end{lemm}

\begin{rema}
By the definition of $\mathbb B^{(k)}$, we have
\begin{align}\label{bb}
\alpha\in\mathbb B^{(k)}\Rightarrow \alpha\in\mathbb R^{(\mathfrak p-1)k+1}~~\text{and}~~~|\alpha|=k.
\end{align}

Indeed, it is clear that \eqref{bb} is true for $k=1$.

Let $k\geq2$. Assume that it is true for $1\leq k^\prime\leq k-1$. For $k$, there exist $\alpha^{\prime}\in\mathbb B^{(k-1)}$ and $\beta\in\mathfrak g^{(k)}$ such that $\alpha=(\alpha^{\prime},\underbrace{0,\cdots,0}_{\mathfrak p-1})+\beta$. Hence $\alpha\in\in\mathbb R^{(\mathfrak p-1)k+1}$ and
\begin{eqnarray*}
|\alpha|&=&|\alpha^{\prime}|+\beta\\
&=&(k-1)+1\\
&=&k.
\end{eqnarray*}
\end{rema}

\begin{proof}[Proof of Lemma~\ref{149}.]
For $k=1$, we have
\begin{eqnarray*}
|c_1(t,n)-c_0(t,n)|&\leq&|n||\omega|\int_0^t\sum_{n_1,\cdots,n_{\mathfrak p}\in\mathbb Z^\nu: \sum_{j=1}^{\mathfrak p}n_j=n}\prod_{j=1}^{\mathfrak p}\Box e^{-\frac{\kappa}{2}|n_j|}{\rm d}\tau\\
&\leq&\Box^{\mathfrak p}|\omega|t\sum_{(n_j)_{1\leq j\leq\mathfrak p}\in\mathbb Z^{\mathfrak p\nu}:~ \sum_{j=1}^{\mathfrak p}n_j=n}\sum_{j=1}^{\mathfrak p}|n_j|\prod_{j=1}^{\mathfrak p}e^{-\frac{\kappa}{2}|n_j|}\\
&=&\Box^{\mathfrak p}|\omega|t\sum_{(n_j)_{1\leq j\leq\mathfrak p}\in\mathbb Z^{\mathfrak p\nu}:~ \sum_{j=1}^{\mathfrak p}n_j=n}\sum_{\alpha\in\mathbb B^{(1)}}\prod_{j=1}^{\mathfrak p}|n_j|^{\alpha_j}e^{-\frac{\kappa}{2}|n_j|}.
\end{eqnarray*}
Let $k\geq2$. Assume that \eqref{cau} holds for $1\leq k^\prime\leq k-1$. For $k$, we have
\begin{eqnarray*}
|c_{k}(t,n)-c_{k-1}(t,n)|
&\leq&\frac{|n||\omega|}{\mathfrak p}\int_0^t\sum_{n_1,\cdots,n_{\mathfrak p}\in\mathbb Z^\nu:~ \sum_{j=1}^{\mathfrak p}|n_j|=n}\left|\prod_{j=1}^{\mathfrak p}c_{k-1}(\tau,n_j)-\prod_{j=1}^{\mathfrak p}c_{k-2}(\tau,n_j)\right|{\rm d}\tau\\
&\triangleq&\frac{1}{\mathfrak p}\sum_{j=1}^{\mathfrak p}\flat^j,
\end{eqnarray*}
where
\begin{eqnarray*}
\flat^1&=&|n||\omega|\int_0^t\sum_{n_1,\cdots,n_{\mathfrak p}\in\mathbb Z^\nu:~ \sum_{j=1}^{\mathfrak p}n_j=n}|c_{k-1}(\tau,n_1)-c_{k-2}(\tau,n_1)||c_{k-2}(\tau,n_2)|\cdots|c_{k-2}(\tau,n_{\mathfrak p})|{\rm d}\tau,\\
\flat^2&=&|n||\omega|\int_0^t\sum_{n_1,\cdots,n_{\mathfrak p}\in\mathbb Z^\nu:~ \sum_{j=1}^{\mathfrak p}n_j=n}|c_{k-2}(\tau,n_1)||c_{k-1}(\tau,n_2)-c_{k-2}(\tau,n_2)|\\
&&|c_{k-2}(\tau,n_3)|\cdots|c_{k-2}(\tau,n_{\mathfrak p})|{\rm d}\tau,\\
&\vdots&\\
\flat^{\mathfrak p}&=&|n||\omega|\int_0^t\sum_{n_1,\cdots,n_{\mathfrak p}\in\mathbb Z^\nu:~ \sum_{j=1}^{\mathfrak p}n_j=n}|c_{k-2}(\tau,n_1)|\cdots|c_{k-2}(\tau,n_{\mathfrak p-1})||c_{k-1}(\tau,n_{\mathfrak p})-c_{k-2}(\tau,n_{\mathfrak p})|{\rm d}\tau.
\end{eqnarray*}
For $\flat^1$, we have
\begin{eqnarray*}
\flat^1&\leq&|n||\omega|\int_0^t\sum_{\substack{n_1,\cdots,n_{\mathfrak p}\in\mathbb Z^\nu\\ \sum_{j=1}^{\mathfrak p}n_j=n}}\frac{\Box^{(\mathfrak p-1)(k-1)+1}(|\omega\tau|)^{k-1}}{(k-1)!}
\sum_{\substack{n^{\prime}\in(\mathbb Z^\nu)^{(\mathfrak p-1)(k-1)+1}\\\sum_{j=1}^{(\mathfrak p-1)(k-1)+1}n_j^{\prime}=n_1}}\sum_{\alpha\in\mathbb B^{(k-1)}}\\
&&\prod_{j=1}^{(\mathfrak p-1)(k-1)+1}|n_j^\prime|^{\alpha_j}e^{-\frac{\kappa}{2}|n_j^\prime|}\times\Box e^{-\frac{\kappa}{2}|n_2|}\times\cdots\times\Box e^{-\frac{\kappa}{2}|n_{\mathfrak p}|}{\rm d}\tau\\
&=&\frac{\Box^{(\mathfrak p-1)k+1}(|\omega|t)^k}{k!}|n|\sum_{\substack{n_1,\cdots,n_{\mathfrak p}\in\mathbb Z^\nu\\ \sum_{j=1}^{\mathfrak p}n_j=n}}\sum_{\substack{n^{\prime}\in(\mathbb Z^\nu)^{(\mathfrak p-1)(k-1)+1}\\\sum_{j=1}^{(\mathfrak p-1)(k-1)+1}n_j^{\prime}=n_1}}\sum_{\alpha\in\mathbb B^{(k-1)}}\prod_{j=1}^{(\mathfrak p-1)(k-1)+1}|n_j^\prime|^{\alpha_j}e^{-\frac{\kappa}{2}|n_j^\prime|}\\
&&\prod_{j=1,j\neq 1}^{\mathfrak p}e^{-\frac{\kappa}{2}|n_j|}\\
&=&\frac{\Box^{(\mathfrak p-1)k+1}(|\omega|t)^k}{k!}\sum_{\substack{(n^{\prime},n_2,\cdots,n_{\mathfrak p})\in(\mathbb Z^\nu)^{(\mathfrak p-1)k+1}\\ \sum_{j=1}^{(\mathfrak p-1)(k-1)+1}n_j^{\prime}+\sum_{j=2}^{\mathfrak p}n_j=n}}
\sum_{\alpha\in\mathfrak g^{(k)}}\prod_{j=1}^{(p-1)k+1}|n_j|^{\alpha_j}
\sum_{\alpha\in\mathbb B^{(k-1)}}\\
&&\prod_{j=1}^{(\mathfrak p-1)(k-1)+1}|n_j^\prime|^{\alpha_j}e^{-\frac{\kappa}{2}|n_j^\prime|}\sum_{\alpha\in\prod_{j=1,j\neq1}^{\mathfrak p}\{0\in\mathbb Z\}}\prod_{j=1,j\neq1}^{\mathfrak p}|n_j|^{\alpha_j}e^{-\frac{\kappa}{2}|n_j|}\\
&=&\frac{\Box^{(\mathfrak p-1)k+1}(|\omega|t)^k}{k!}\sum_{\substack{(n_1,n_2,\cdots,n_{(\mathfrak p-1)k+1})\in(\mathbb Z^\nu)^{(\mathfrak p-1)k+1}\\\sum_{j=1}^{(\mathfrak p-1)k+1}n_j=n}}\sum_{\alpha\in\mathbb B^{(k-1)}\times\prod_{j=1}^{\mathfrak p-1}\{0\in\mathbb Z\}+\mathfrak g^{(k)}}\prod_{j=1}^{(\mathfrak p-1)k+1}|n_j|^{\alpha_j}e^{-\frac{\kappa}{2}|n_j|}\\
&=&\frac{\Box^{(\mathfrak p-1)k+1}(|\omega|t)^k}{k!}\sum_{\substack{(n_j)_{1\leq j\leq(\mathfrak p-1)k+1}\in(\mathbb Z^\nu)^{(\mathfrak p-1)k+1}:~\sum_{j=1}^{\mathfrak p-1}n_j=n}}\sum_{\alpha\in\mathbb B^{(k)}}\prod_{j=1}^{(\mathfrak p-1)k+1}|n_j|^{\alpha_j}e^{-\frac{\kappa}{2}|n_j|}.
\end{eqnarray*}
Analogously,
\[\flat^j\leq\frac{\Box^{(\mathfrak p-1)k+1}(|\omega|t)^k}{k!}\sum_{\substack{(n_j)_{1\leq j\leq(\mathfrak p-1)k+1}\in(\mathbb Z^\nu)^{(\mathfrak p-1)k+1}\\\sum_{j=1}^{\mathfrak p-1}n_j=n}}\sum_{\alpha\in\mathbb B^{(k)}}\prod_{j=1}^{(\mathfrak p-1)k+1}|n_j|^{\alpha_j}e^{-\frac{\kappa}{2}|n_j|}, \quad \forall j=1,2,\cdots,\mathfrak p.\]
Hence
\begin{eqnarray*}
&&|c_{k}(t,n)-c_{k-1}(t,n)|\\
&\leq&\frac{\Box^{(\mathfrak p-1)k+1}(|\omega|t)^k}{k!}\sum_{\substack{(n_j)_{1\leq j\leq(\mathfrak p-1)k+1}\in(\mathbb Z^\nu)^{(\mathfrak p-1)k+1}\\\sum_{j=1}^{(\mathfrak p-1)k+1}n_j=n}}\sum_{\alpha\in\mathbb B^{(k)}}\prod_{j=1}^{(\mathfrak p-1)k+1}|n_j|^{\alpha_j}e^{-\frac{\kappa}{2}|n_j|}\\
&=&\frac{\Box^{(\mathfrak p-1)k+1}(|\omega|t)^k}{k!}\sum_{\alpha\in\mathbb B^{(k)}}\sum_{\substack{(n_j)_{1\leq j\leq(\mathfrak p-1)k+1}\in(\mathbb Z^\nu)^{(\mathfrak p-1)k+1}\\\sum_{j=1}^{(\mathfrak p-1)k+1}n_j=n}}\prod_{j=1}^{(\mathfrak p-1)k+1}|n_j|^{\alpha_j}e^{-\frac{\kappa}{2}|n_j|}\\
&\stackrel{(\scriptsize\text{Lemma \ref{4}})}{\leq}&e^{-\frac{\kappa}{4}|n|}\frac{\Box^{(\mathfrak p-1)k+1}(|\omega|t)^k}{k!}\sum_{\alpha\in\mathbb B^{(k)}}(12\kappa^{-1})^{|\alpha|+\left((\mathfrak p-1)k+1\right)\nu}\prod_{j}\alpha_j!\\
&=&\Theta e^{-\frac{\kappa}{4}|n|}\frac{\left(\Box^{\mathfrak p-1}(12\kappa^{-1})^{(\mathfrak p-1)\nu+1}|\omega|t\right)^{k}}{k!}\sum_{\alpha\in\mathbb B^{(k)}}\prod_{j}\alpha_j!\\
&\leq&\Theta e^{-\frac{\kappa}{4}|n|}\frac{\left(\Box^{\mathfrak p-1}(12\kappa^{-1})^{(\mathfrak p-1)\nu+1}|\omega|t\right)^{k}}{k!}\left\{2\left[(\mathfrak p-1)k+1\right]\right\}^{k}\\\
&\stackrel{(\scriptsize{\text{Stirling formula}})}{\leq}&\Theta e^{-\frac{\kappa}{4}|n|}\left(\Box^{\mathfrak p-1}(12\kappa^{-1})^{(\mathfrak p-1)\nu+1}|\omega|t\right)^{k}\left(\frac{e}{k}\right)^k\left\{2\left[(\mathfrak p-1)k+1\right]\right\}^{k}\\
&=&\Theta e^{-\frac{\kappa}{4}|n|}\left(\Box^{\mathfrak p-1}(12\kappa^{-1})^{(\mathfrak p-1)\nu+1}|\omega|t\right)^{k}(2(\mathfrak p-1)e)^k\left(1+\frac{1}{(\mathfrak p-1)k}\right)^k\\
&\leq&\Theta e^{-\frac{\kappa}{4}|n|+\frac{1}{\mathfrak p-1}}\left(2(\mathfrak p-1)e\Box^{\mathfrak p-1}(12\kappa^{-1})^{(\mathfrak p-1)\nu+1}|\omega|t\right)^{k}.
\end{eqnarray*}
This completes the proof Lemma \ref{149}.
\end{proof}

\begin{theo}\label{cthm}
For $0<t<\min\left\{\frac{\kappa^{(\mathfrak p-1)\nu+1}}{2^{\mathfrak p+1}6^{(\mathfrak p-1)\nu+1}\mathcal A|\omega|},\frac{\kappa^{(\mathfrak p-1)\nu+1}}{2(\mathfrak p-1)e\Box^{\mathfrak p-1}12^{(\mathfrak p-1)\nu+1}|\omega|}\right\}\triangleq t_0>0$, and all $k\geq1$, we have
\begin{align}\label{le}
|c_{k+\bigstar}(t,n)-c_{k}(t,n)|\leq\frac{\Theta e^{-\frac{\kappa}{4}|n|+\frac{1}{\mathfrak p-1}}\left\{2(\mathfrak p-1)e\Box^{\mathfrak p-1}(12\kappa^{-1})^{(\mathfrak p-1)\nu+1}|\omega|t\right\}^{k+1}}{1-2(\mathfrak p-1)e\Box^{\mathfrak p-1}(12\kappa^{-1})^{(\mathfrak p-1)\nu+1}|\omega|t}~~(\text{uniformly for}~\bigstar).
\end{align}
\end{theo}

\begin{proof}
By Lemma \ref{149}, we have
\begin{eqnarray*}
&&|c_{k+\bigstar}(t,n)-c_{k}(t,n)|\\
&\leq&\sum_{j=1}^{\bigstar}|c_{k+j}(t,n)-c_{k+j-1}(t,n)|\\
&\leq&\Theta e^{-\frac{\kappa}{4}|n|+\frac{1}{\mathfrak p-1}}\sum_{j=1}^{\bigstar}\left(2(\mathfrak p-1)e\Box^{\mathfrak p-1}(12\kappa^{-1})^{(\mathfrak p-1)\nu+1}|\omega|t\right)^{k+j}\\
&\leq&\Theta e^{-\frac{\kappa}{4}|n|+\frac{1}{\mathfrak p-1}}\left(2(\mathfrak p-1)e\Box^{\mathfrak p-1}(12\kappa^{-1})^{(\mathfrak p-1)\nu+1}|\omega|t\right)^{k}\\
&&\sum_{j=1}^{\bigstar}
\left(2(\mathfrak p-1)e\Box^{\mathfrak p-1}(12\kappa^{-1})^{(\mathfrak p-1)\nu+1}|\omega|t\right)^{j}\\
&\leq&\Theta e^{-\frac{\kappa}{4}|n|+\frac{1}{\mathfrak p-1}}\left(2(\mathfrak p-1)e\Box^{\mathfrak p-1}(12\kappa^{-1})^{(\mathfrak p-1)\nu+1}|\omega|t\right)^{k+j}\\
&&\sum_{j=1}^{+\infty}
\left(2(\mathfrak p-1)e\Box^{\mathfrak p-1}(12\kappa^{-1})^{(\mathfrak p-1)\nu+1}|\omega|t\right)^{k+j}\\
&=&\frac{\Theta e^{-\frac{\kappa}{4}|n|+\frac{1}{\mathfrak p-1}}}{1-2(\mathfrak p-1)e\Box^{\mathfrak p-1}(12\kappa^{-1})^{(\mathfrak p-1)\nu+1}|\omega|t}\left(2(\mathfrak p-1)e\Box^{\mathfrak p-1}(12\kappa^{-1})^{(\mathfrak p-1)\nu+1}|\omega|t\right)^{k+1},
\end{eqnarray*}
provided that $0<t<t_0$. The proof of Theorem \ref{cthm} is completed.
\end{proof}

\section{Proof of Theorem \ref{eethm}}

\begin{proof}[Proof of Theorem \ref{eethm}.]
By Theorem \ref{cthm}, we know that the Picard sequence $\{c_k(t,n)\}$ is fundamental (Cauchy sequence). Hence there exists a limit function
$$
c^\dag(t,n):=\lim_{k\rightarrow\infty}c_k(t,n), \quad0\leq t<t_0,  n\in\mathbb Z^\nu.
$$
Letting $\bigstar\rightarrow\infty$ in \eqref{le}, we have
\[
|c^\dag(t,n)-c_k(t,n)|\leq\frac{\Theta e^{-\frac{\kappa}{4}|n|+\frac{1}{\mathfrak p-1}}\left(2(\mathfrak p-1)e\Box^{\mathfrak p-1}(12\kappa^{-1})^{(\mathfrak p-1)\nu+1}|\omega|t\right)^{k+1}}{1-2(\mathfrak p-1)e\Box^{\mathfrak p-1}(12\kappa^{-1})^{(\mathfrak p-1)\nu+1}|\omega|t}\rightarrow0\quad\text{as}\quad k\rightarrow+\infty.
\]
By Theorem \ref{expe}, we have
\begin{eqnarray*}
|c^\dag(t,n)|&\leq&|c^\dag(t,n)-c_k(t,n)|+|c_k(t,n)|\\
&\leq&\Box e^{-\frac{\kappa}{2}|n|}.
\end{eqnarray*}
Letting $k\rightarrow+\infty$ in \eqref{ppp}, one can obtain that
\[c^\dag(t,n)=c_0(t,n)-\frac{{\rm i}n\cdot\omega}{\mathfrak p}\int_{0}^{t}e^{{\rm i}(n\cdot\omega)^3(t-\tau)}(c^{\dag})^{\ast\mathfrak p}(\tau,n){\rm d}\tau.\]
Define the spatially $\omega$-quasi-periodic function
\[
u^\dag(t,x):=\sum_{n\in\mathbb Z^\nu}c^\dag(t,n)e^{{\rm i}(n\cdot\omega)x}, \quad 0\leq t<t_0, x\in\mathbb R.
\]
Obviously, $u^\dag$ satisfies $\mathfrak p$-gKdV \eqref{ppkdv} with spatially quasi-periodic initial data \eqref{iei}.
\end{proof}

\section{Proof of Theorem \ref{euthm}}

\begin{lemm}\label{k}
For all $k\geq1$, we have
\begin{eqnarray}
&&|c(t,n)-b(t,n)|\nonumber\\
&\leq&\frac{\Box^{(\mathfrak p-1)k+1}(|\omega|t)^k}{k!}\sum_{\substack{(n_j)_{1\leq j\leq(\mathfrak p-1)k+1})\in(\mathbb Z^\nu)^{(\mathfrak p-1)k+1}\\ \sum_jn_j=n}}\sum_{\alpha\in\mathbb B^{(k)}}
\prod_{j}|n_j|^{\alpha_j}e^{-\rho|n_j|}\label{h}\\
\nonumber&\leq&(12\rho^{-1})^\nu\Box e^{-\frac{\rho}{2}|n|+\frac{1}{\mathfrak p-1}}\left(2(\mathfrak p-1)e\Box^2(12\rho^{-1})^{(\mathfrak p-1)\nu+1}|\omega|\right)^k.
\end{eqnarray}
\end{lemm}

\begin{proof}
Since $c(0,n)=b(0,n)$, we have
\begin{eqnarray*}
|c(t,n)-b(t,n)|&\leq&\frac{|n||\omega|}{\mathfrak p}\int_0^t\sum_{n_1,\cdots,n_{\mathfrak p}\in\mathbb Z^\nu:~\sum_{j=1}^{\mathfrak p}n_j=n}|\prod_{j=1}^{\mathfrak p}c(\tau,n_j)-\prod_{j=1}^{n}b(\tau,n_j)|{\rm d}\tau\\
&\leq&\Box^{\mathfrak p}|n||\omega|t\sum_{n_1,\cdots,n_{\mathfrak p}\in\mathbb Z^\nu:~\sum_{j=1}^{\mathfrak p}n_j=n}\prod_{j=1}^{\mathfrak p} e^{-\rho|n_j|}\\
&\leq&\Box^{\mathfrak p}|\omega|t\sum_{(n_1,\cdots,n_{\mathfrak p})\in\mathbb Z^{\mathfrak p\nu}:~\sum_{j=1}^{\mathfrak p}n_j=n}\sum_{\alpha\in\mathbb B^{(1)}}\prod_{j=1}^{\mathfrak p}|n_j|^{\alpha_j}e^{-\rho|n_j|}.
\end{eqnarray*}
Let $k\geq2$. Assume that \eqref{h} holds for $1\leq k^\prime\leq k-1$. For $k$, we have
\begin{eqnarray*}
|c(t,n)-b(t,n)|&\leq&\frac{|n||\omega|}{\mathfrak p}\int_0^t\sum_{n_1,\cdots,n_{\mathfrak p}\in\mathbb Z^\nu: \sum_{j=1}^{\mathfrak p}n_j=n}|\prod_{j=1}^{\mathfrak p}c(\tau,n_j)-\prod_{j=1}^{\mathfrak p}b(\tau,n_j)|{\rm d}\tau\\
&\triangleq&\frac{1}{\mathfrak p}\sum_{j=1}^{\mathfrak p}\psi_j,
\end{eqnarray*}
where
\begin{eqnarray*}
\psi_1&=&|n||\omega|\int_0^t\sum_{n_1,\cdots,n_{\mathfrak p}\in\mathbb Z^\nu: \sum_{j=1}^{\mathfrak p}n_j=n}|c(\tau,n_1)-b(\tau,n_1)||b(\tau,n_2)|\cdots|b(\tau,n_{\mathfrak p})|{\rm d}\tau,\\
\psi_2&=&|n||\omega|\int_0^t\sum_{n_1,\cdots,n_{\mathfrak p}\in\mathbb Z^\nu: \sum_{j=1}^{\mathfrak p}n_j=n}|b(\tau,n_1)||c(\tau,n_2)-b(\tau,n_2)||b(\tau,n_3)|\cdots|b(\tau,n_{\mathfrak p})|{\rm d}\tau,\\
&\vdots&\\
\psi_{\mathfrak p}&=&|n||\omega|\int_0^t\sum_{n_1,\cdots,n_{\mathfrak p}\in\mathbb Z^\nu: \sum_{j=1}^{\mathfrak p}n_j=n}|b(\tau,n_1)|\cdots|b(\tau,n_{\mathfrak p-1})||c(\tau,n_{\mathfrak p})-b(\tau,n_{\mathfrak p})|{\rm d}\tau.
\end{eqnarray*}
For $\psi_1$, we have
\begin{eqnarray*}
\psi_1&\leq&|n||\omega|\int_0^t\sum_{\substack{n_1,\cdots,n_{\mathfrak p}\in\mathbb Z^\nu\\ \sum_{j=1}^{\mathfrak p}n_j=n}}\frac{\Box^{(\mathfrak p-1)(k-1)+1}(|\omega\tau|)^{k-1}}{(k-1)!}
\sum_{\substack{n^{\prime}\in(\mathbb Z^\nu)^{(\mathfrak p-1)(k-1)+1}\\\sum_{j=1}^{(\mathfrak p-1)(k-1)+1}n_j^{\prime}=n_1}}\sum_{\alpha\in\mathbb B^{(k-1)}}\\
&&\prod_{j=1}^{(\mathfrak p-1)(k-1)+1}|n_j^\prime|^{\alpha_j}e^{-\frac{\kappa}{2}|n_j^\prime|}\times\Box e^{-\rho|n_2|}\times\cdots\times\Box e^{-\rho|n_{\mathfrak p}|}{\rm d}\tau\\
&=&\frac{\Box^{(\mathfrak p-1)k+1}(|\omega|t)^k}{k!}|n|\sum_{\substack{n_1,\cdots,n_{\mathfrak p}\in\mathbb Z^\nu\\ \sum_{j=1}^{\mathfrak p}n_j=n}}\sum_{\substack{n^{\prime}\in(\mathbb Z^\nu)^{(\mathfrak p-1)(k-1)+1}\\\sum_{j=1}^{(\mathfrak p-1)(k-1)+1}n_j^{\prime}=n_1}}\sum_{\alpha\in\mathbb B^{(k-1)}}\prod_{j=1}^{(\mathfrak p-1)(k-1)+1}|n_j^\prime|^{\alpha_j}e^{-\rho|n_j^\prime|}\\
&&\prod_{j=1,j\neq 1}^{\mathfrak p}e^{-\rho|n_j|}\\
&=&\frac{\Box^{(\mathfrak p-1)k+1}(|\omega|t)^k}{k!}\sum_{\substack{(n^{\prime},n_2,\cdots,n_{\mathfrak p})\in(\mathbb Z^\nu)^{(\mathfrak p-1)k+1}\\ \sum_{j=1}^{(\mathfrak p-1)(k-1)+1}n_j^{\prime}+\sum_{j=2}^{\mathfrak p}n_j=n}}
\sum_{\alpha\in\mathfrak g^{(k)}}\prod_{j=1}^{(p-1)k+1}|n_j|^{\alpha_j}
\sum_{\alpha\in\mathbb B^{(k-1)}}\\
&&\prod_{j=1}^{(\mathfrak p-1)(k-1)+1}|n_j^\prime|^{\alpha_j}e^{-\rho|n_j^\prime|}\sum_{\alpha\in\prod_{j=1,j\neq1}^{\mathfrak p}\{0\in\mathbb Z\}}\prod_{j=1,j\neq1}^{\mathfrak p}|n_j|^{\alpha_j}e^{-\rho|n_j|}\\
&=&\frac{\Box^{(\mathfrak p-1)k+1}(|\omega|t)^k}{k!}\sum_{\substack{(n_1,n_2,\cdots,n_{(\mathfrak p-1)k+1})\in(\mathbb Z^\nu)^{(\mathfrak p-1)k+1}\\\sum_{j=1}^{(\mathfrak p-1)k+1}n_j=n}}\sum_{\alpha\in\mathbb B^{(k-1)}\times\prod_{j=1}^{\mathfrak p-1}\{0\in\mathbb Z\}+\mathfrak g^{(k)}}\prod_{j=1}^{(\mathfrak p-1)k+1}|n_j|^{\alpha_j}e^{-\rho|n_j|}\\
&=&\frac{\Box^{(\mathfrak p-1)k+1}(|\omega|t)^k}{k!}\sum_{\substack{(n_j)_{1\leq j\leq(\mathfrak p-1)k+1}\in(\mathbb Z^\nu)^{(\mathfrak p-1)k+1}:\sum_{j=1}^{(\mathfrak p-1)k+1}n_j=n}}\sum_{\alpha\in\mathbb B^{(k)}}\prod_{j=1}^{(\mathfrak p-1)k+1}|n_j|^{\alpha_j}e^{-\rho|n_j|}.
\end{eqnarray*}
Similarly,
\[\psi_j\leq\frac{\Box^{(\mathfrak p-1)k+1}(|\omega|t)^k}{k!}\sum_{\substack{(n_j)_{1\leq j\leq(\mathfrak p-1)k+1}\in(\mathbb Z^\nu)^{(\mathfrak p-1)k+1}\\ \sum_{j=1}^{(\mathfrak p-1)k+1}n_j=n}}\sum_{\alpha\in\mathbb B^{(k)}}\prod_{j=1}^{(\mathfrak p-1)k+1}|n_j|^{\alpha_j}e^{-\rho|n_j|},\quad\forall j=1,\cdots,\mathfrak p.\]
Hence
\begin{eqnarray*}
&&|c(t,n)-b(t,n)|\\
&\leq&\frac{\Box^{(\mathfrak p-1)k+1}(|\omega|t)^k}{k!}\sum_{\substack{(n_j)_{1\leq j\leq(\mathfrak p-1)k+1}\in(\mathbb Z^\nu)^{(\mathfrak p-1)k+1}\\ \sum_{j=1}^{(\mathfrak p-1)k+1}n_j=n}}\sum_{\alpha\in\mathbb B^{(k)}}\prod_{j=1}^{(\mathfrak p-1)k+1}|n_j|^{\alpha_j}e^{-\rho|n_j|}\\
&\leq&\frac{\Box^{(\mathfrak p-1)k+1}(|\omega|t)^k}{k!}e^{-\frac{\rho}{2}|n|}\sum_{\alpha\in\mathbb B^{(k)}}\sum_{\substack{(n_j)_{1\leq j\leq(\mathfrak p-1)k+1}\in(\mathbb Z^\nu)^{(\mathfrak p-1)k+1}\\ \sum_{j=1}^{(\mathfrak p-1)k+1}n_j=n}}\prod_{j=1}^{(\mathfrak p-1)k+1}|n_j|^{\alpha_j}e^{-\frac{\rho}{2}|n_j|}\\
&\leq&(12\rho^{-1})^\nu\Box e^{-\frac{\rho}{2}|n|}\frac{\left(\Box^{\mathfrak p-1}(12\rho^{-1})^{(\mathfrak p-1)\nu+1}|\omega|t\right)^k}{k!}\sum_{\alpha\in\mathbb B^{(k)}}\prod_j\alpha_j!\quad\text{\scriptsize by Lemma \ref{4}}\\
&\leq&(12\rho^{-1})^\nu\Box e^{-\frac{\rho}{2}|n|}\frac{\left(\Box^{\mathfrak p-1}(12\rho^{-1})^{(\mathfrak p-1)\nu+1}|\omega|t\right)^k}{k!}\left\{2\left[(\mathfrak p-1)k+1\right]\right\}^{k}\\
&\leq&(12\rho^{-1})^\nu\Box e^{-\frac{\rho}{2}|n|}\left(\Box^{\mathfrak p-1}(12\rho^{-1})^{(\mathfrak p-1)\nu+1}|\omega|t\right)^k\left(\frac{e}{k}\right)^k\left\{2\left[(\mathfrak p-1)k+1\right]\right\}^{k}\quad\scriptsize{\text{by the Stirling formula}}\\
&=&(12\rho^{-1})^\nu\Box e^{-\frac{\rho}{2}|n|}\left(\Box^{\mathfrak p-1}(12\rho^{-1})^{(\mathfrak p-1)\nu+1}|\omega|t\right)^k(2(\mathfrak p-1)e)^k\left(1+\frac{1}{(\mathfrak p-1)k}\right)^k\\
&\leq&(12\rho^{-1})^\nu\Box e^{-\frac{\rho}{2}|n|+\frac{1}{\mathfrak p-1}}\left(2(\mathfrak p-1)e\Box^{\mathfrak p-1}(12\rho^{-1})^{(\mathfrak p-1)\nu+1}|\omega|t\right)^k.
\end{eqnarray*}
This completes the proof of Lemma \ref{k}.
\end{proof}

\begin{proof}[Proof of Theorem \ref{euthm}.]
By Lemma \ref{k} and the arbitrariness of $k$, we know that if
$$
0 < t < \frac{1}{2(\mathfrak p-1)e\Box^2(12\rho^{-1})^{(\mathfrak p-1)\nu+1}|\omega|},
$$
then $|c(t,n)-b(t,n)|\rightarrow0$ as $k\rightarrow\infty.$
Hence for all $0\leq t<\min\left\{t_0,\frac{1}{2(\mathfrak p-1)e\Box^2(12\rho^{-1})^{(\mathfrak p-1)\nu+1}|\omega|}\right\}$ and $x\in\mathbb R$, one has $c(t,n)\equiv b(t,n)$ that is
\[u_1(t,x)\equiv u_2(t,x), \quad \forall0\leq t<\min\left\{t_0,\frac{1}{4e\Box^2(12\rho^{-1})^{2\nu+1}|\omega|}\right\}~\text{and}~x\in\mathbb R.\]
This completes the proof of Theorem \ref{euthm}.
\end{proof}

\begin{rema}
This paper studies the existence and uniqueness of spatially quasi-periodic solutions to the generalized KdV equation using a combinatorial method. It should be emphasized that it is not only technical but also can be developed into a systematic theory on dealing with the problem of higher dimensional discrete convolution operation, generated by the nonlinearity and quasi-periodic Fourier series.
\end{rema}

\part{Appendix}

In this appendix we discuss some combinatorial inequalities that are needed in the main body of the paper.

\section{Some Combinatorial Inequalities}

\begin{lemm}\label{1}
For $0 < \kappa \leq1$, we have
$$
\sum_{m\in\mathbb Z}e^{-\kappa|m|}\leq 3\kappa^{-1}.
$$
\end{lemm}

\begin{proof}
It is easy to see that
\begin{align*}
  \sum_{m\in\mathbb Z}e^{-\kappa|m|}&=2\sum_{m\in\mathbb N}e^{-\kappa m}-1\\
  &=2\times\frac{1-(e^{-\kappa})^{+\infty}}{1-e^{-\kappa}}-1\\
  &=\frac{1+e^{-\kappa}}{1-e^{-\kappa}}\\
&<3\kappa^{-1},
\end{align*}
provided that $0<\kappa\leq1$. The  proof of Lemma \ref{1} is completed.
\end{proof}

\begin{lemm}\label{2}
For $0 < \kappa \leq 1$ and $z \geq 0$, we have
$$
z^\alpha\leq \alpha!(2\kappa^{-1})^\alpha e^{\frac{\kappa}{2}z},\quad\forall \alpha\in\mathbb N.
$$
\end{lemm}

\begin{proof}
Obviously,
  \begin{align*}
  z^\alpha=\alpha!(2\kappa^{-1})^\alpha\frac{\left(\frac{\kappa}{2}z\right)^\alpha}{\alpha!}\leq \alpha!(2\kappa^{-1})^\alpha\sum_{\alpha=0}^{\infty}\frac{\left(\frac{\kappa}{2}z\right)^\alpha}{\alpha!}\leq \alpha!(2\kappa^{-1})^\alpha e^{\frac{\kappa}{2}z}.
  \end{align*}
  This completes the proof of Lemma \ref{2}.
\end{proof}

\begin{lemm}\label{3}
For $0 < \kappa \leq 1$, we have
$$
\sum_{(m_1,\cdots,m_r)\in(\mathbb Z^{\nu})^r}\prod_{j=1}^{r}|m_j|^{\alpha_j}e^{-\kappa|m_j|}\leq(6\kappa^{-1})^{|\alpha|+\nu r}\prod_{j=1}^{r}\alpha_j!.
$$
\end{lemm}

\begin{proof}
By Lemma \ref{2} and Lemma \ref{1}, we find
\begin{eqnarray*}
  &&\sum_{(m_1,\cdots,m_r)\in(\mathbb Z^{\nu})^r}\prod_{j=1}^{r}|m_j|^{\alpha_j}e^{-\kappa|m_j|}\\
  &\stackrel{(\text{Lemma \ref{2})}}{\leq}&\sum_{(m_1,\cdots,m_r)\in(\mathbb Z^{\nu})^r}\prod_{j=1}^{r}\alpha_j!(2\kappa^{-1})^{\alpha_j}e^{\frac{\kappa}{2}|m_j|}\cdot e^{-\kappa|m_j|}\\
  &=&(2\kappa^{-1})^{|\alpha|}\prod_{j=1}^{r}\alpha_j!\sum_{(m_1,\cdots,m_r)\in(\mathbb Z^{\nu})^r}\prod_{j=1}^{r}e^{-\frac{\kappa}{2}|m_j|}\\
  &=&(2\kappa^{-1})^{|\alpha|}\prod_{j=1}^{r}\alpha_j!\left(\sum_{m\in\mathbb Z^\nu}e^{-\frac{\kappa}{2}|m|}\right)^r\\
  &=&(2\kappa^{-1})^{|\alpha|}\prod_{j=1}^{r}\alpha_j!\left(\sum_{m\in\mathbb Z}e^{-\frac{\kappa}{2}|m|}\right)^{\nu r}\\
  &\stackrel{(\text{Lemma \ref{1})}}{\leq}&(2\kappa^{-1})^{|\alpha|}\cdot(6\kappa^{-1})^{\nu r}\prod_{j=1}^{r}\alpha_j!\\
  &\leq&(6\kappa^{-1})^{|\alpha|+\nu r}\prod_{j=1}^{r}\alpha_j!.
\end{eqnarray*}
The proof of Lemma \ref{3} is completed.
\end{proof}

\begin{lemm}\label{4}
For $0 < \kappa \leq 1$, we have
$$
\sum_{\substack{m=(m_1,\cdots,m_r)\in(\mathbb Z^{\nu})^r: \mu(m)=n}}\prod_{j=1}^{r}|m_j|^{\alpha_j}e^{-\kappa|m_j|}\leq(6\kappa^{-1})^{|\alpha|+\nu r}\prod_{j=1}^{r}\alpha_j!e^{-\frac{\kappa}{2}|n|}.
$$
\end{lemm}

\begin{proof}
By Lemma \ref{3}, we have
  \begin{eqnarray*}
  &&\sum_{\substack{m=(m_1,\cdots,m_r)\in(\mathbb Z^{\nu})^r:~\mu(m)=n}}\prod_{j=1}^{r}|m_j|^{\alpha_j}e^{-\kappa|m_j|}\\
  &=&\sum_{\substack{m=(m_1,\cdots,m_r)\in(\mathbb Z^{\nu})^r:~\mu(m)=n}}\prod_{j=1}^{r}|m_j|^{\alpha_j}e^{-\frac{\kappa}{2}|m_j|}\cdot \prod_{j=1}^{r}e^{-\frac{\kappa}{2}|m_j|}\\
  &\leq&\sum_{\substack{m=(m_1,\cdots,m_r)\in(\mathbb Z^{\nu})^r:~\mu(m)=n}}\prod_{j=1}^{r}|m_j|^{\alpha_j}e^{-\frac{\kappa}{2}|m_j|}\cdot e^{-\frac{\kappa}{2}|n|}\\
  &\leq&\sum_{\substack{m=(m_1,\cdots,m_r)\in(\mathbb Z^{\nu})^r}}\prod_j|m_j|^{\alpha_j}e^{-\frac{\kappa}{2}|m_j|}\cdot e^{-\frac{\kappa}{2}|n|}\\
  &\stackrel{(\text{Lemma}~~\ref{3})}\leq&(12\kappa^{-1})^{|\alpha|+\nu r}\prod_j\alpha_j!e^{-\frac{\kappa}{2}|n|}.\\
  \end{eqnarray*}
This completes the proof of Lemma \ref{4}.
\end{proof}

\bibliographystyle{alpha}

\bibliography{gkdv}

\end{document}